\documentclass{amsart}
\usepackage{graphicx,epsfig,color}%in case of figures
\usepackage{amsmath,amsthm,amssymb,hyperref}
\usepackage{framed}
\usepackage{tikz}
\usepackage{dsfont}
\usepackage[all]{xy}
\usepackage{amsfonts}
\usepackage{amsmath}
\usepackage{epsfig}
\usepackage{graphicx}
\usepackage{amssymb}

\newtheorem{theoreme}{Theorem}[section]
\newtheorem{proposition}[theoreme]{Proposition}
\newtheorem{lemma}[theoreme]{Lemma}
\newtheorem{corollary}[theoreme]{Corollary}

\theoremstyle{definition}
\newtheorem*{definition}{Definition}
\newtheorem{remarque}[theoreme]{Remark}

\def\edc{\end{document}}
\numberwithin{equation}{section}
\begin{document}

\author{Nasab Yassine}
\address{Universit\'e Bretagne Sud
\\
LMBA\\
CNRS UMR 6205\\
BP 573,56017 Vannes, France}
\email{nasab.yassine@univ-ubs.fr}

\title[Quantitative recurrence for $\mathbb{Z}$-extension of three-dimensional Axiom A flows]{Quantitative recurrence for $\mathbb{Z}$-extension of three-dimensional Axiom A flows}
\begin{abstract}
In this paper, we study the quantitative recurrence properties in the case of $\mathbb{Z}$-extension of Axiom A flows on a Riemannian manifold. We study the asymptotic behavior of the first return time to a small neighborhood of the starting point. We establish results of almost everywhere convergence, and of convergence in distribution with respect to any probability measure absolutely continuous with respect to the infinite invariant measure. In particular, our results apply to geodesic flows on $\mathbb{Z}$-cover of compact smooth surfaces of negative curvature.\\

\end{abstract}

\maketitle

\section{Introduction}

%In this chapter we recall the necessary notions and definitions of an Axiom A flow. We start by defining a hyperbolic set for a flow. We introduce the definition of a Markov section which is a classical method of studying the symbolic representation of a dynamical system. These were developed by Bowen \cite{Bowen} and Ratner \cite{Ratner}. Then we represent the special flow over a subshift, where we define the Poincar\'e section and the coding map. We refer to Barreira \cite{barreira} for a general reference of the definitions and properties. Proceeding in this chapter, we work on  establishing some properties concerning balls and coding. For an $x\in X$, and $r>0$, under some conditions taken on the dilatation and the contraction in the unstable and stable direction of the flow respectively, we showed that the ball $B(x,r)$ contains and is contained in a cylinder. Proving this property, serves in the sense that the asymptotic behavior of the return time to a ball will be studied through considering the return time to a cylinder. And that was the strategy followed in proving the almost sure convergence result in Chapter 5, and for the convergence in distribution result in chapter 6.

Quantitative recurrence has been recently a popular topic of dynamical systems and ergodic theory since the first works obtained by \cite{Doeblin} for the Gauss map and \cite{Pitskel} for the case of Markov chains. These properties are set by estimating the first return time of a dynamical system into a small neighborhood of its starting point. For further references, we refer to \cite{saussol} where some recurrence results in the context of probability preserving dynamical systems have been presented. We mention works in this concern \cite{abadigalves}, and for rapidly mixing systems \cite{saussolrapidly} and \cite{saussol}. Few studies have been established in the case where the dynamical system preserves an infinite measure, let us mention \cite{Bressaud-Zweimuller}, where Bressaud and Zweim\"uller have established first results in the case of piecewise affine maps of the interval with infinite measure. In \cite{yassine}, the recurrence of $\mathbb{Z}-$extension of subshifts of finite type was investigated, as well as the $\mathbb{Z}^2$ case in \cite{penesaussol}, and in \cite{penesaussol2024} quantitative recurrence results for $T$, $T^{-1}$ transformations have been proved. Other relevant contributions can be found in \cite{penesaussolbilliard}, \cite{pzs}, and \cite{pzs2}.\\
In this paper we extend the study to include dynamical systems with continuous time, in the case where the measure is infinite. More precisely, we consider the $\mathbb{Z}-$extension of Axiom A flows, including the geodesic flows on $\mathbb{Z}$-periodic negatively curved manifolds.\\ 
%We prove the pointwise convergence of the recurrence rate to the dimension in the case of a continuous time. 
Thus, we consider a Riemannian manifold $\tilde{M}$ of dimension 3, given by a $\mathbb{Z}-$extension, endowed with a $\sigma-$finite measure $\tilde{\mu}$, and a flow $(\tilde{g}_{t})_{t\in\mathbb{R}}$ on $\tilde{M}$ that preserves the measure $\tilde{\mu}$. We set $\Gamma$ to be an infinite group of isometries of $\tilde{M}$ that also preserves the measure $\tilde{\mu}$. We then suppose that $M=\tilde{M}/\Gamma$ is a compact manifold, and we define the quotient flow $(g_{t})_{t\in\mathbb{R}}$ on $M$. We define a measure $\mu$ on $M$ that is obtained from the measure $\tilde{\mu}$ by passing through the quotient. We further assume that $\mu$ is an equilibrium measure for the flow $g_{t}$ on $M$, and that $(M, (g_{t})_{t})$ is an Axiom A flow.\\
We are interested in the first return time $\tau_{\epsilon}(x)$ of the flow $\tilde{g}_{t}$ to an $\epsilon-$neighborhood of its starting point $x$, defined by:
\begin{equation*}
\tau_{\epsilon}(x):=\inf\ \{t>1: \tilde{g}_{t}(x)\in B(x,\epsilon)\},
\end{equation*}
where $B(x,\epsilon)$ is the ball of center $x$ and radius $\epsilon$.
Studying the asymptotic behavior of $\tau_{\epsilon}$ as $\epsilon\rightarrow0$, we first prove in Theorem \ref{theorem1flow} the following almost everywhere convergence result:

\begin{equation*}
\underset{\epsilon\rightarrow0}{\lim}\frac{\log\sqrt{\tau_{\epsilon}}}{-\log\epsilon}=\dim_H \mu -1,
\end{equation*}

where  $\dim_H \mu$ is the Hausdorff dimension of the measure $\mu$.
%For a $\tilde{y}\in\tilde{M}$, we describe the existence of a measure $\tilde{\nu}_{0}$ on a disk centered at $\tilde{y}$, which is transversal to the flow, and orthogonal to it at $\tilde{y}$. We call it measure transversal to the flow at $\tilde{y}$. 
% $\tilde{\mu}$-almost everywhere, where $h$ is the entropy of the measure $\mu$ and $L^{s}$ and $L^{u}$ the expansion rates of the stable and unstable leafs respectively (see Theorem \ref{theorem1flow} ). 
Moreover, the following convergence holds in distribution, with respect to any probability measure absolutely continuous with respect to $\tilde{\mu}$:

$$\tilde{\nu}_{0}(B(.,\epsilon))\sqrt{\tau_{\epsilon}(.)}\underset{\epsilon\rightarrow0}{\longrightarrow} \sigma_{flow}\frac{\mathcal{E}}{|\mathcal{N}|},$$ where $\tilde{\nu}_{0}$ is a sectional measure (defined in \eqref{nu_0}), $\sigma^2_{flow}$ is an asymptotic variance related to the flow, and $\mathcal{E}$ and $\mathcal{N}$ are independent random variables, $\mathcal{E}$ having an exponential distribution of mean 1 and $\mathcal{N}$ having a standard Gaussian distribution (the statement of this result is precised in Theorem \ref{theorem2}).\\

Moreover, we apply our result in the case of the unit geodesic flow on a $\mathbb{Z}-$cover of a $C^2$ compact surface $\mathcal{S}$ of negative curvature, the following convergence in distribution holds with respect to any probability measure absolutely continuous with respect to $\tilde{\mu}$:

$$\frac{\epsilon^2}{2 \text{Vol}( \mathcal{S})}\sqrt{\tau_{\epsilon}(.)}\underset{\epsilon\rightarrow0}{\longrightarrow} \sigma_{flow}\frac{\mathcal{E}}{|\mathcal{N}|}.$$
where $\mathcal{E}$ and $\mathcal{N}$ are two independent random variables with respective exponential distribution of mean 1 and standard normal distribution (See Theorem \ref{geodesicflow} for precised statement).\\

In this article, we start by introducing in section 2 all the necessary notions and results from hyperbolic dynamics concerning Axiom A flows and a hyperbolic set for a flow. In particular, we consider the Markov section constructed by Bowen \cite{Bowen} and Ratner \cite{Ratner} for a locally maximal hyperbolic set. We note that any smooth flow with a hyperbolic set can be represented as a suspension flow over a mixing subshift of finite type.\\
%Setting $X$ the Poincar\'e section for the flow $g_{t}$, and the height function $R:X\rightarrow(0,\infty)$ such that %$R(x)=\min\{s>0: g_{s}x\in X\}$, we give a notion of a suspension flow $\psi=\{\psi_{t}\}_{t\in\mathbb{R}}$ over a transformation map $T:X\rightarrow X$, with height function $R$. We consider the space $$X_{R}=\{(x,s)\in X\times\mathbb{R}: 0\leq s\leq R(x)\},$$ where the points $(x,R(x))$ and $(T(x),0)$ are identified for each $x\in X$.\\
Another point is to describe how a Markov section for a hyperbolic set gives rise to a symbolic dynamics (see Section 2.1).
We consider a set $\Sigma_{A}$ together with the shift map $\sigma$. And then we give a definition of the coding map $\chi:\Sigma_{A}\rightarrow X$. And thus in the same way we define the symbolic suspension flow $\mathbf{S}=\{\mathbf{S}_{t}\}_{t\in\mathbb{R}}$ over $\sigma_{|\Sigma_{A}}$.
%We introduce a section defining an equilibrium measure of a flow.
%We explain the existence of a unique equilibrium measure for an H\"older function $H$ for the symbolic suspension flow $\mathcal{S}=\{\mathcal{S}_{t}\}_{t\in\mathbb{R}}$. 
We establish some properties on balls and coding. We introduce in section 3 the model of $\mathbb{Z}-$extension we are interested in. We show that a given ball contains and is contained in a suitable cylinder. This serves in studying the asymptotic behavior of our return time to a cylinder and then inferring the return time to a ball, and yet establishing the desired almost sure convergence (in section 4) and  convergence in distribution results (in section 5). At the end we apply our results to prove the convergence in distribution theorem in the case where the flow is a geodesic flow on $\mathbb{Z}$-periodic negatively curved manifold (in section 6).\\

\section{Definition and properties of Axiom A Flows}
Let $M$ be a Riemannian smooth manifold of dimension 3. We consider a $C^{1}$ flow $(g_{t})_{t\in\mathbb{R}}$ on $M$. A compact $g_{t}$-invariant set $\Lambda\subset M$ is said to be a hyperbolic set for $g_{t}$ if there exists a continuous splitting of the tangent space, $T_{\Lambda}M= E^{s}\oplus E^{u}\oplus E^{0}$, and constants $c>0$ and $\lambda\in(0,1)$ such that for each $x\in \Lambda$ and $t\in \mathbb{R}$;
\begin{itemize}
\item[1.]the vector $\frac{d g_{t}(x)}{dt}\mid_{t=0}$ generates $E^{0}(x)$, which is the flow direction;
\item[2.]$d_{x}g_{t}E^{s}(x)=E^{s}(g_{t}(x))$ and $d_{x}g_{t}E^{u}(x)=E^{u}(g_{t}(x))$;
\item[3.] for all $t>0$, $||d_{x}g_{t}v||\leq c\lambda^{t}||v|| \text{ for } v\in E^{s}(x)$, and $
||d_{x}g_{-t}v||\leq c\lambda^{t}||v|| \text{ for } v\in E^{u}(x).$
\end{itemize}

%%%%%%%%%%%%%%%%%%%%%%%%%%%
\begin{center}
\begin{figure}[!h]
%\caption{The splitting of a hyperbolic set}
\centering
\includegraphics[scale=0.17]{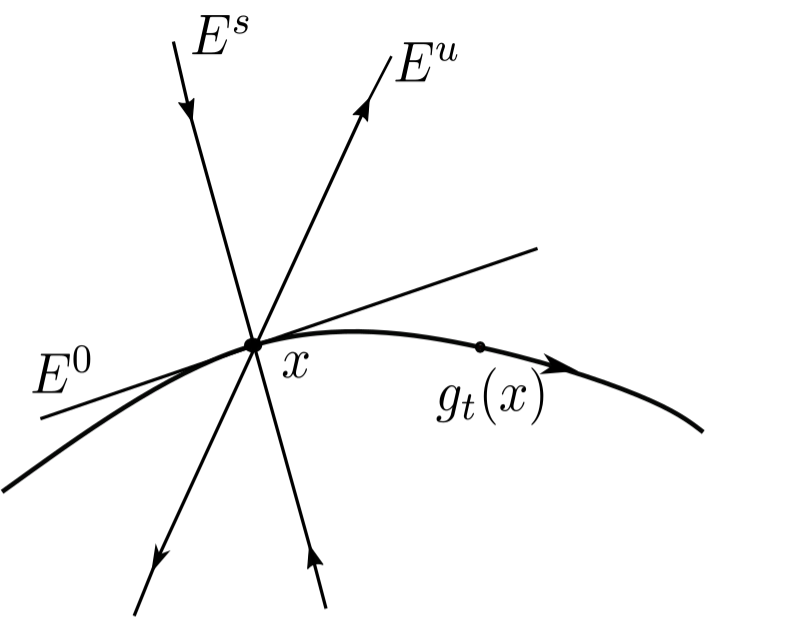}
\end{figure}
\label{hyperbolic}
\end{center}
The subspaces $E^{s}(x)$ and $E^{u}(x)$ are called the stable and unstable subspaces at $x$, they depend continuously on the point and are invariant.
%A locally maximal hyperbolic set is a finite union of disjoint closed invariant subsets, each of which is topologically transitive, that is, contains a dense orbit. Each of these transitive components for $g_{t}$ in turn is a finite union of closed subsets $\Lambda_{1},...,\Lambda_{k}$ such that $g_{t}(\Lambda_{i})=\Lambda_{i+1}$ when $0\leq i<k$ and $g_{t}(\Lambda_{k})=\Lambda_{0}$.\\
%%%%%%%%%%
% We assume that $g_{t}$ is topologically transitive on a locally maximal hyperbolic set $\Lambda$.\\
For each $x\in\Lambda$, there exist stable and unstable local manifolds at the point $x$, $W^{s}_{loc}(x)$ and $W^{u}_{loc}(x)$, which are tangent to the contracting and expanding subspaces and contain the point $x$:
\begin{itemize}
\item[1.] $x\in W^{s}_{loc}(x) $ and $x\in W^{u}_{loc}(x) $;
\item[2.] $T_{x}W^{s}_{loc}(x)=E^{s}(x)$ and $T_{x}W^{u}_{loc}(x)=E^{u}(x)$;
\item[3.] for each $t>0$, $g_{t}(W^{s}_{loc}(x))\subset W^{s}_{loc}(g_{t}(x))$, $g_{-t}(W^{u}_{loc}(x))\subset W^{u}_{loc}(g_{-t}(x))$;
\item[4.] there exist $\kappa>0$ and $\vartheta>0$ such that for each $t>0$, we have
\begin{equation}\label{ds}
d(g_{t}(x), g_{t}(y))\leq \kappa e^{-\vartheta t}d(x,y) \text{ for } y\in W^{s}_{loc}(x),
\end{equation}
and
\begin{equation*}
d(g_{-t}(x), g_{-t}(y))\leq \kappa e^{-\vartheta t}d(x,y) \text{ for } y\in W^{u}_{loc}(x),
\end{equation*}
\end{itemize}
where $d$ here is the distance induced in $M$ by the Riemannian metric.
%%%%%%%%%%%%%%%%%
A hyperbolic set $\Lambda$ is said to be locally maximal (with respect to a flow $g_{t}$) if there exists an open neighborhood $U$ of $\Lambda$ such that
\begin{equation*}
\Lambda=\bigcap_{t\in\mathbb{R}}g_{t}(U).
\end{equation*}
%%%%%%%%%%%%%%%%%%%%%%%%%%%%%%%%%%%%%%%%%%%%%%%%%%%
Let $\Lambda$ be a locally maximal hyperbolic set. For every $\epsilon>0$ there is $\delta>0$ such that for any $x,y$ in $\Lambda$ with $d(x,y)\leq\delta$ there exists a unique $t=t(x,y)\in[-\epsilon,\epsilon]$, for which the following intersection:
\begin{equation*}
[x,y]:= W^{s}_{loc}(g_{t}(x))\cap W^{u}_{loc}(y)
\end{equation*}
consists of a single point in $\Lambda$; moreover, the map $(x,y)\mapsto t(x,y)$ is continuous.\\
We define the global stable and unstable manifolds at a point $x\in\Lambda$ by:
$$W^{s}(x)=\bigcup_{t\geq 0}g_{-t}\left(W^{s}_{loc}(g_{t}(x))\right)\mbox{ , } W^{u}(x)=\bigcup_{t\geq 0}g_{t}\left(W^{u}_{loc}(g_{-t}(x))\right).$$
They can be characterized as follows:
\begin{equation*}
W^{s}(x)=\{y\in \Lambda: d(g_{t}(x), g_{t}(y))\rightarrow0 \text{ as } t\rightarrow\infty\},
\end{equation*}
\begin{equation*}
W^{u}(x)=\{y\in \Lambda: d(g_{-t}(x), g_{-t}(y))\rightarrow0 \text{ as } t\rightarrow\infty\}.
\end{equation*}

A flow $(g_{t})_{t\in\mathbb{R}}$ is said to be an Axiom $A$ flow if its set of non-wandering points is hyperbolic: a point $x \in M$ is called a non-wandering point if for every open neighborhood $\mathcal{O}$ of $x$ there is a time $t_{x} > 0$, such that $g_{t_{x}}(\mathcal{O})\cap \mathcal{O}\neq\emptyset$.
Throughout this paper, we consider $(g_{t})_{t\in\mathbb{R}}$ an Axiom A flow and topologically mixing on a locally maximal hyperbolic set $\Lambda$. A decomposition of this set provides an analysis of its dynamics, that is the construction of the Markov collection, which was developed by Bowen \cite{Bowen} and Ratner \cite{Ratner}.
%(!!! flow is topologically transitive on the hyperbolic set!!!)\\
%We assume that $g_{t}$ is topologically transitive on a locally maximal hyperbolic set $\Lambda$
%%We introduce the notion of a Markov system which was developed by Bowen \cite{Bowen} and Ratner \cite{Ratner}.
%%Let $\Lambda $ be a maximal hyperbo lic set with respect to the flow $g_{t}$. A decomposition of this set provides an analysis of its dynamics, which is the construction of the Markov collection.\medskip
%
%\noindent
Given a point $x\in\Lambda$, consider a small compact disk $D\subset M$ containing $x$ of co-dimension one which is transversal to the flow $g_{t}$. This disk is a local section of the flow, i.e., there exists $\tau>0$ such that the map $(x,t)\rightarrow g_{t}(x)$ is a diffeomorphism of the direct product $D\times[-\tau,\tau]$ onto a neighborhood $U_{\tau}(D)$:
\begin{eqnarray*}
\varphi_{D}: D\times [0,\tau]&\rightarrow& U_{\tau}(D):= \varphi_{D}(D\times[0,\tau])\subset M\\
(x,t)&\rightarrow& y= g_{t}(x)
\end{eqnarray*}
The function $\varphi_{D}$ is Lipschitz. We set $k_{\varphi_{D}}:=\underset{D}{\max}(Lip\ \varphi_{D}, Lip\ \varphi_{D}^{-1})$.
\\
The projection $P_{D}: U_{\tau}(D)\rightarrow  D$, defined such that $ P_{D}(z):=g_{-s}(z)$, where $s>0$ is the minimum value by $g_{t}(z)\in D$ is a differentiable map.
%The projection of the ball $B(y_{0},r)$, where $y_{0}\in M$, on $D$ is
%\begin{equation*}
%P_{D}(B(y_{0},r)):=\{x\in D \mbox{ : } \exists s<R(x) \mbox{ : } g_{s}(x)\in B(y_{0},r)\}.
%\end{equation*}
Consider now a closed set $\Pi\subset \Lambda\cap D$ which doesn't intersect the boundary $\partial D$. For any two points $y$, $z\in \Pi$, we set
\begin{equation}\label{proj}
\{y,z\}:=P_{D}([y,z]).
\end{equation}
The set $\Pi$ is said to be a rectangle if $\Pi=\overline{\text{int} \Pi}$ (where the interior is considered with respect to the induced topology of $\Lambda\cap D$ ) and $\{y,z\}\in \Pi$ for any $y,z\in \Pi$.\\
Since we are working on a three-dimensional manifold, the flow $(g_t)_t$ is conformal. We have:
\begin{proposition}{(\cite{barreira,hasselblatt})}\label{barreiracL}
Let $\Pi$ be a rectangle. The maps $x\mapsto E^{s}(x)\oplus E^{0}(x)$ and $x\mapsto E^{u}(x)\oplus E^{0}(x)$ are Lipschitz. This implies that there exists $c_{L}>0$ such that $\Pi\times\Pi \ni(x,y)\mapsto\{x,y\}\in \Pi$ is a Lipschitz map of Lipschitz constant $c_{L}>0$ and that $\{x,.\}$ and $\{.,y\}$ have uniformly Lipschitz inverses. 
\end{proposition}
Let $\Pi$ be a rectangle, then for every $x\in\Pi$, we set:\\
$$W^{s}_{loc}(x,\Pi)=\{\{x,y\}:y\in\Pi\}=\Pi\cap P_{D}\left(U_{\tau}(D)\cap W^{s}_{loc}(x)\right),$$ $$W^{u}_{loc}(x,\Pi)=\{\{x,y\}:y\in\Pi\}=\Pi\cap P_{D}\left(U_{\tau}(D)\cap W^{u}_{loc}(x)\right).$$
%\section{Representation by a special flow over a subshift}
Now we consider a collection of rectangles $\Pi_{1},...,\Pi_{n}$ (each contained in some open disk transversal to the flow) such that $\Pi_{i}\cap \Pi_{j}=\partial \Pi_{i}\cap\partial \Pi_{j} \mbox{ for } i\neq j.$ Set $X=\bigcup_{i=1}^{n}\Pi_{i}$. We assume that there exists $\beta_{0}>0$ such that:
\begin{itemize}
\item[1.] $\Lambda=\bigcup_{t\in[0,\beta_{0}]}g_{t}(X)$;
\item[2.] for each $i\neq j$, for every $t\in[0,\beta_{0}]$, at least $g_{t}(\Pi_{i})\cap \Pi_{j}=\emptyset \mbox{ or }g_{t}(\Pi_{j})\cap \Pi_{i}=\emptyset$.
%For every $x\in X$, we define $\tau(x)$ the smallest positive number such that $g_{\tau(x)}(x)\in X$
\end{itemize}
The set $X$ is a Poincar\'e section for the flow $g_{t}$. The height function $R$ is defined such that for every $ x\in X$, one can find the smallest positive number $R(x)$ such that $g_{R(x)}(x)\in X$. The Poincar\'e map $T:X\longrightarrow X$ is defined by $T(x)= g_{R(x)}(x)$.
%We define the projection map $P_{D}$ on $X$ such that \textcolor{red}{PROJECTION.....}\\
%We define thus
%\begin{itemize}
%\item the height function $R:X\rightarrow (0,\infty)$ by
%\begin{equation}\label{height}
%R(x):=\min\{s>0 : g_{s}x\in X\},
%\end{equation}\label{poincaremap}
%\item the Poincar\'e map $T:\Lambda\longrightarrow X$ by
%\begin{equation}\label{poincareT}
%T(x)= g_{R(x)}(x).
%\end{equation}
%\end{itemize}
%Furthermore, the map $R$ is H\"older continuous on each domain of continuity, and
%\begin{equation}
%0< \inf\{ R(x):x \in X\}\leq\sup\{R(x): x\in \Lambda\}<\infty.
%\end{equation}
%\end{definition}
%\begin{definition}
A Markov collection is a finite collection of rectangles $\Pi_{1},...,\Pi_{n}$ satisfying, for every $x\in \text{int} \Pi_{i}\cap \text{int}T^{-1}(\Pi_{j})$:
$$T\left(\text{int} (W_{loc}^{s}(x)\cap \Pi_{i})\right)\subset\text{int} \left(W_{loc}^{s}(T(x))\cap \Pi_{j}\right)$$
and for every  $x\in \text{int} \Pi_{i}\cap \text{int}T(\Pi_{k})$:
$$T^{-1}\left(\text{int}(W_{loc}^{u}(x)\cap \Pi_{i})\right)\subset\text{int}\left( W_{loc}^{u}(T^{-1}(x))\cap \Pi_{k}\right).$$
It follows from the work of Bowen and Ratner \cite{Bowen} and \cite{Ratner} that any locally maximal hyperbolic set $\Lambda$ has a Markov collection of arbitrary small diameter.
 Given a rectangle $\Pi_{i}$, we call the set
\begin{equation}
R_{i}=\bigcup_{x\in \Pi_{i}}\bigcup_{0\leq t\leq R(x)}g_{t}(x)\subset \Lambda
\end{equation}
a Markov set. Note that $R_{i}=\overline{\text{int} R_{i}}$ and $\text{int} R_{i}\cap \text{int} R_{j}=\emptyset$ for $i\neq j$.
%\end{definition}
%Consider a topologically mixing subshift of finite type $(\Sigma_{A},\sigma)$. For  an element $\omega\in \Sigma_{A}$ corresponds a point $x\in M$  such that $x=\chi(\omega)$. Hence we define $r:\Sigma_{A}\rightarrow(0,\infty)$ such that $r(w):=R\circ\chi(\omega)$.

\subsection{Suspension Flow}\label{suspensionflow}
$T:X\rightarrow X$ being the Poincar\'e map restricted on $X$, and $R$ the height function, we consider the space $ X_{R}=\{(x,s)\in X\times\mathbb{R} \text{ : } 0\leq s\leq R(x) \},$ where the points $(x,R(x))$ and $(T(x), 0)$ are identified for each $x\in X$.
%\begin{definition}\label{psi4.4def}
The suspension flow over $T$, with height function $R$ is the flow $\psi=\{\psi_{t}\}_{t\in\mathbb{R}}$ in $X_{R}$, given by
\begin{equation}\label{psit}
\psi_{t}(x,s)=(T^{n}x, s+t-S_{n}R(x)),
\end{equation}
where $n$ is the unique non negative integer such that $0\leq s+t-S_{n}R(x)<R(T^{n}x)$, i.e. $$S_{n}R(x)\leq s+t<S_{n+1}R(x),$$ where we set as usual $S_{n}R=\sum\limits_{k=0}^{n-1}R\circ T^{k}$.
%\end{definition}
%
\begin{center}
\begin{figure}[!h]
%\caption{A suspension flow $(\mathcal{S}_{t})_{t\in{\mathbb{R}}}$ over $T$}
\centering
\includegraphics[width=9cm, height=6.75cm]{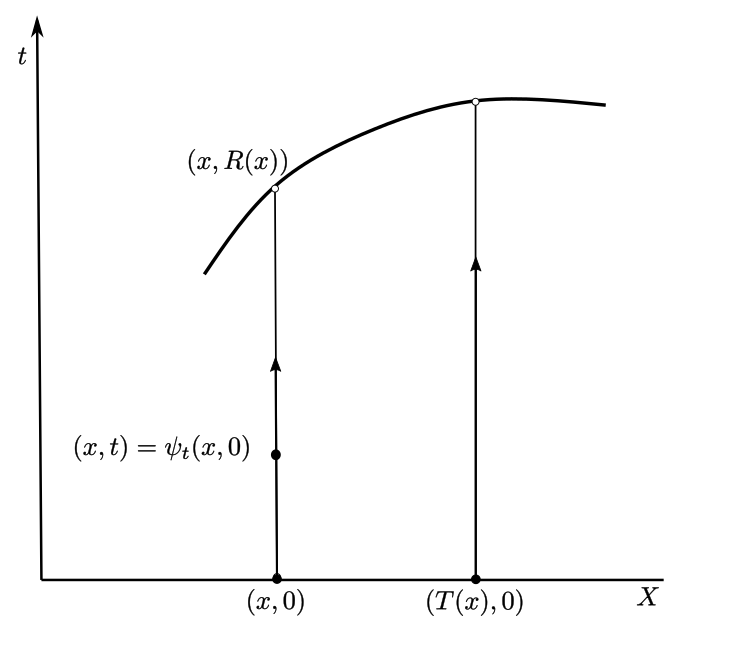}
\end{figure}
\label{suspension}
\end{center}
%\begin{figure}
%\includegraphics[scale=0.4]{suspension.png}
%\caption{A suspension flow $(\mathcal{S}_{t})_{t\in{\mathbb{R}}}$ over $T$}
%%\label{suspension}
%\end{figure}
Using a Markov collection of a hyperbolic set $\Lambda$, we can obtain a symbolic representation of the flow $g_{t}$.
Let $\Lambda$ be a locally maximal hyperbolic set for the flow $g_{t}$, and given $\Pi_{1},..,\Pi_{n}$ the associated Markov collection. Consider the $n\times n$ matrix $A$ with entries
\begin{align*}
a_{ij}=
\left
\lbrace
\begin{array}{l l}
1 \text{  if } \text{int} T(\Pi_{i})\cap \text{int}\Pi_{j}\neq\emptyset,\\
0 \text{  otherwise,}
 \end{array}
 \right.
\end{align*} 
where $T$ is the Poincar\'e map. We also consider the set $\Sigma_{A}$ given by:
\begin{equation*}
\Sigma_{A}=\{\omega=(\omega_{n})_{n\in \mathbb{Z}}:\forall n\in \mathbb{Z}, a_{\omega_{n}\omega_{n+1}}=1\},
\end{equation*}
together with the shift map $\sigma:\Sigma_{A}\rightarrow \Sigma_{A} $ defined by $\sigma((\omega_{n})_{n\in\mathbb{Z}})=(\omega_{n+1})_{n\in\mathbb{Z}}$. Since the flow is topologically mixing, the matrix $A$ is primitive aperiodic, i.e. there exists $n_{0}\in\mathbb{N}$ such that $A^{n_{0}}>0$. We endow $\Sigma_{A}$ with the metric $\hat{d}$ given by 
$$\hat{d}(\omega, \omega'):=e^{-m},$$
where $m$ is the greatest integer such that $\omega_{i}=\omega_{i}'$ whenever $|i| < m.$\\
Let $q,q'$ be two positive integers, and $a:=a_{-q},.., a_{0},..,a_{q'}$ be a finite symbol sequence, we define a $(-q,q')$-cylinder by $$C_{-q,q'}(a)=\{(\omega_{n})_{n\in\mathbb{Z}}\in \Sigma_{A}: \omega_{-q}=a_{-q},...,\omega_{q'}=a_{q'}\}.$$
We denote by $C_{-q,q'}(\omega)$ the $(-q,q')$-cylinder containing a point $\omega\in\Sigma_{A}$.

%\begin{definition}
%The map $\sigma|\Sigma_{A}$ is said to be a (two-sided) topological Markov chain with transition matrix $A$.
%\end{definition}
%%%%%%%%%%%%%% GIBBS MEASURE
\begin{proposition}\label{gibbs}
Let $\tilde{h}:\Sigma\rightarrow\mathbb{R}$ be a  H\"{o}lder continuous function. An invariant measure $\nu$ is called a Gibbs measure for the potential $\tilde{h}$, if there exists a constant $P_{\sigma}(\tilde{h})\in\mathbb{R}$, called the pressure, such that for some $k_{\tilde{h}}\geq 1$, for any $\omega$ and any $q,q'\geq0$,
\begin{equation}\label{gibbs.formula}
\frac{1}{k_{\tilde{h}}}\leq\frac{\nu(C_{-q,q'}(\omega))}{\exp\left(\sum_{k=-q}^{q'}\tilde{h}(\sigma^{k}(\omega))-(q+q'+1)P_{\sigma}(\tilde{h})\right)}\leq k_{\tilde{h}}
\end{equation}
\end{proposition}
In what follows, let $\nu$ be the Gibbs measure on $\Sigma_{A}$ associated to some  H\"{o}lder continuous potential $\tilde{h}$.
\begin{definition}
We consider the coding map $\chi:\Sigma_{A}\rightarrow X$, defined by:
$
\chi(\omega)=\underset{n}{\bigcap}\overline{\underset{|k|\leq n}{\bigcap}T^{-k}(\text{int}\Pi_{\omega_{k}})}
$.
One can easily verify that the following diagram
\begin{equation*}
\xymatrix{
    \Sigma_{A} \ar[r]^\sigma \ar[d]_\chi & \Sigma_{A} \ar[d]^{\chi} \\
    X \ar[r]_{T}       & X }
\end{equation*}

%$$\begin{CD}

%\Sigma_{A}     @> >>  \Sigma_{A}\\
%@VV\chi V        @VV\chi V\\
%X    @> T>>  X
%\end{CD}$$
is commutative i.e.
\begin{equation}\label{commute}
\chi\circ\sigma=T\circ\chi,
\end{equation}
and we have $\chi_{*}\nu=\nu_{X}$ the measure on $X$.
\end{definition}
As at the beginning of this section, we define in the same way the suspension flow $\mathbf{S}=\{\mathbf{S}_{t}\}_{t\in\mathbb{R}}$ over $\sigma|_{\Sigma_{A}}$ with the H\"older height function $r=R\circ\chi$.
We extend $\chi$ to finite-to-one onto map $\chi:\Delta\rightarrow\Lambda$ by $\chi(\omega,s)=(g_{s}\circ\chi)(\omega)$,
for $(\omega,s)\in\Delta$.
where $\Delta:=\{(\omega,s)\in \Sigma_{A}\times\mathbb{R} \text{ : } 0\leq s\leq r(\omega) \}$, with $r(\omega)=R(\chi(\omega))$. Due to \eqref{commute}, for every $t\in\mathbb{R}$ the following diagram

\begin{equation*}
\xymatrix{
    \Delta \ar[r]^{\mathbf{S}_{t}} \ar[d]_\chi & \Delta \ar[d]^{\chi} \\
    \Lambda \ar[r]_{g_{t}}       & \Lambda }
\end{equation*}

%$$\begin{CD}
%\Delta     @>\mathbf{S}_{t}>>  \Delta\\
%@VV\chi V        @VV\chi V\\
%     @>g_{t}>>  \Lambda
%\end{CD}$$
commutes, expressing that the flows $g_{t}$ and $\mathbf{S}_{t}$ are conjugated, i.e. $\chi\circ \mathbf{S}_{t}=g_{t}\circ\chi$. The measure $\mu_{\Delta}=\frac{\nu\times Leb|\Delta}{\int_{\Sigma_{A}}r d\nu}$ is a probability measure on $\Delta$, which is invariant by $\mathbf{S}_{t}$, where $Leb$ is the Lebesgue measure on $\mathbb{R}$. And we have $\chi_{*}\mu_{\Delta}=\mu$, where we recall that $\mu$ is an equilibrium measure for the flow $g_{t}$.\\

% \textcolor{blue}{We recall the Abramov's formula (see \cite{Abramov}) which establishes a relation between the entropy of the suspension flow $\mathbf{S}_{t}$ and the entropy of the invariant measure on $(\Sigma_A, \sigma)$, which we will use in the proof of Theorem \ref{theorem1flow}.
% \begin{equation}\label{Abramov1}
%     h_{\mu_{\Delta}}(\mathbf{S}_{t}):= h_{\mu_{\Delta}}(\mathbf{S}_{1})=\dfrac{h_{\nu}(\sigma)}{\int_{\Sigma_A}r d\nu}
% \end{equation}}
Let $h_\nu(\sigma)$ be the entropy of the invariant measure $\nu$ on the system $(\Sigma_A, \sigma)$. For flows, the entropy is defined as the entropy of the time-1 map. We recall the Abramov's formula (see \cite{Abramov}) which establishes a relation between the entropy of the suspension flow $\mathbf{S}_{t}$ and $h_\nu (\sigma)$ as follows:
\begin{equation}\label{Abramov1}
    h_{\mu_{\Delta}}(\mathbf{S}_{t}):= h_{\mu_{\Delta}}(\mathbf{S}_{1})=\dfrac{h_{\nu}(\sigma)}{\int_{\Sigma_A}r d\nu}.
\end{equation}
%We recall the Abramov's formula which establishes a relation between the entropy  $h_\mu:=h_\mu(g_1)$ of the flow $(M,(g_t)_t,\mu)$ and the entropy $h_{\nu_X}$ of the invariant measure $\nu_X$ on $(X,T)$,
% Let $h_{\nu_X}$ be the entropy of the invariant measure $\nu_X$ on $(X,T)$, for the Poincar\'e transformation. We recall that the Abramov's formula relates $h_{\nu_X}$ with the entropy of the flow $h_\mu:=h_\mu(g_1)$, as follows:
Moreover, since the systems $(\Lambda, g_t,\mu)$ and $(\Delta,\mathbf{S}_t,\mu_\Delta)$ are conjugated, they have the same entropy, thus $h_{\mu_{\Delta}}(\mathbf{S}_t)=h_{\mu}(g_t)$. For the same reason, $h_{\nu}(\sigma)=h_{\nu_X}(T)$. So we have the following relation 
 \begin{equation}\label{Abramov2}
    h_\mu(g_t)=\frac{h_{\nu_X}(T)}{\int_{X}R d\nu_{X}},
\end{equation}
 which we will use in the proof of Theorem \ref{theorem1flow}.\\

\subsection{Balls and Coding}\label{section balls and coding}
We denote by $\Sigma_{A}^{+}$ the set of (one-sided) sequences $(\omega_{n})_{n\geq0}$ such that: $$(\omega_{n})_{n\geq0}=(w_{n})_{n\geq0} \mbox{ for some }  (w_{n})_{n\in\mathbb{Z}}\in \Sigma_{A},$$
and by $\Sigma_{A}^{-}$ the set of (one-sided) sequences $(\omega_{n})_{n\leq0}$ such that: $$(\omega_{n})_{n\leq0}=(w_{n})_{n\leq0} \mbox{ for some }  (w_{n})_{n\in\mathbb{Z}}\in \Sigma_{A}.$$
We also consider the shift maps $\sigma^{+}:\Sigma_{A}^{+}\rightarrow\Sigma_{A}^{+}$ and $\sigma^{-}:\Sigma_{A}^{-}\rightarrow\Sigma_{A}^{-}$ defined by
$$ \sigma^{+}((\omega_{n})_{n\geq0})=(\omega_{n+1})_{n\geq0} \mbox{ and }\sigma^{-}((\omega_{n})_{n\leq0})=(\omega_{n-1})_{n\leq0}$$
Now we describe how we use symbolic dynamics to characterize  distinct points in a stable or unstable manifold. Given $x\in\Lambda$, take $\omega\in \Sigma_{A}$ such that $\chi(\omega)=x$. Let $\Pi(x)$ be a rectangle of the Markov collection containing $x$.
Let $\chi^{+}:\Sigma_{A}\rightarrow\Sigma_{A}^{+}$ and $\chi^{-}:\Sigma_{A}\rightarrow\Sigma_{A}^{-}$ be the projection maps, defined by:
$$\chi^{+}((\omega_{n})_{n\in\mathbb{Z}})=(\omega_{n})_{n\geq0} \mbox{ and }\chi^{-}((\omega_{n})_{n\in\mathbb{Z}})=(\omega_{n})_{n\leq0}.$$
For each $\omega'\in \Sigma_{A}$, we have
$$\chi(\omega')\in W^{u}_{loc}(x)\cap \Pi(x) \mbox{ whenever } \chi^{-}(\omega')=\chi^{-}(\omega),$$
and
$$\chi(\omega')\in W^{s}_{loc}(x)\cap \Pi(x) \mbox{ whenever } \chi^{+}(\omega')=\chi^{+}(\omega).$$
Therefore, writing $\omega=(\omega_{n})_{n\in\mathbb{Z}}$, the set $W^{u}_{loc}(x)\cap \Pi(x)$ can be identified with the cylinder set:
\begin{equation}
C^{+}_{\omega_{0}}=\{(a_{n})_{n\geq0}\in\Sigma_{A}^{+}: a_{0}=\omega_{0}\}\subset\Sigma_{A}^{+}
\end{equation}
and the set $W^{s}_{loc}(x)\cap \Pi(x)$ can be identified with the cylinder set:
\begin{equation}
C^{-}_{\omega_{0}}=\{(a_{n})_{n\leq0}\in\Sigma_{A}^{-}: a_{0}=\omega_{0}\}\subset\Sigma_{A}^{-}.
\end{equation}

%\begin{proposition}
%Let $\Lambda$ be a locally maximal hyperbolic set for the flow $g_{t}$ and $\varphi:\Lambda\rightarrow\mathbb{R}$ a H\"older continuous function. Then there exists a unique equilibrium measure $\nu_{\varphi}$ corresponding to $\varphi$. Moreover, the measure $\nu_{\varphi}$ is ergodic and positive on open sets.
%\end{proposition}
%\begin{definition}
We define the measure $\nu^{u}$ on $\Sigma_{A}^{+}$ such that for any cylinder $ C_{0,n}$, we have 
\begin{equation}\label{measure-u}
    \nu^{u}(\chi^{+}(C_{0,n}))=\nu(C_{0,n}).
\end{equation}
Similarly, we define the measure $\nu^{s}$ on $\Sigma_{A}^{-}$ such that for any cylinder $C_{-n,0}$, we have 
\begin{eqnarray}\label{measure-s}
    \nu^{s}(\chi^{-}(C_{-n,0}))=\nu(C_{-n,0}).
\end{eqnarray}
%\end{definition}
%\begin{definition}
% \textcolor{blue}{A $C^1$-flow $(g_t)_t$ on a locally maximal hyperbolic set $\Lambda$ is said to be conformal on $\Lambda$, if the maps $$d_x g_t|_{E^s(x)}: E^s(x)\rightarrow E^s(g_t (x)) \mbox{ and }  d_x g_t|_{E^u(x)}: E^u(x)\rightarrow E^u(g_t (x))$$ are multiple of isometries for all $x\in\Lambda$, and $t\in \mathbb{R}$, i.e. there exist continuous functions $A^{(u)}$ and $A^{(s)}$ on $\Lambda\times\mathbb{R}$ such that
% $$dg_t|_{E^{u}(x)}=A^{(u)}(x,t)I^{(u)}(x,t) \mbox { and } dg_t|_{E^{s}(x)}=A^{(s)}(x,t)I^{(s)}(x,t).$$
% }
Now, as we are considering a conformal Axiom $A$ flow $(g_t)_t$ on a locally maximal hyperbolic set $\Lambda$, it means that there exist continuous functions $A^{(u)}$ (respectively $A^{(s)}$) on $\Lambda\times\mathbb{R}$ such that for every $x\in\Lambda$, and $t\in\mathbb{R}$,
\begin{equation}\label{conformalu}
dg_t|_{E^{u}(x)}=A^{(u)}(x,t)I^{(u)}(x,t),
\end{equation}
respectively,
\begin{equation}\label{conformals}
dg_t|_{E^{s}(x)}=A^{(s)}(x,t)I^{(s)}(x,t),
 \end{equation}
 where $I^{u}(x,t): E^u(x)\rightarrow E^u(g_t (x)) $ and $I^{s}(x,t): E^s(x)\rightarrow E^s(g_t (x))$ are isometries.\\ 
Given $x\in\Lambda$, we define functions $a^{(u)}(x)$ and $a^{(s)}(x)$ by
\begin{equation}\label{au}
a^{(u)}(x):=\dfrac{\partial}{\partial t}\log A^{(u)}(x,t)|_{t=0}=\underset{t\rightarrow0}{\lim} \ \frac{\log||dg_{t}|_{E^{u}(x)}||}{t},
\end{equation}
\begin{equation*}
a^{(s)}(x):=\dfrac{\partial}{\partial t}\log A^{(s)}(x,t)|_{t=0}=\underset{t\rightarrow0}\lim \ \frac{\log||dg_{t}|_{E^{s}(x)}||}{t}.
\end{equation*}
%\end{definition}
Since the subspaces $E^{u}(x)$ and $E^{s}(x)$ depend H\"older continuously on $x$ the functions $a^{(u)}(x)$ and $a^{(s)}(x)$ are also H\"older continuous. Note that there exist constants $\bar{c}$ and $\bar{c}'$ such that $a^{(u)}(x)>\bar{c}>0$ and $a^{(s)}(x)<\bar c'<0$ for every $x\in \Lambda$. For any $x\in \Lambda$ and $t\in\mathbb{R}$, we have:
\begin{equation}\label{dgt au}
||dg_{t}(v)||=||v||\exp\int_{0}^{t} a^{(u)}(g_{\tau}(x))d\tau \quad \mbox{ for any } v\in E^{u}(x),
\end{equation}
and
\begin{equation}\label{dgt as}
||dg_{t}(w)||=||w||\exp\int_{0}^{t} a^{(s)}(g_{\tau}(x))d\tau \quad \mbox{ for any } w\in E^{s}(x).
\end{equation}
For every $x\in \Lambda$, the Lyapunov exponent at $x$ takes on two values which are given by
\begin{equation}\label{lyapunov-}
\lambda^{s}(x):=\underset{t\rightarrow\infty}\lim \ \frac{\log||dg_{t}(x)|_{E^{s}(x)}||}{t}=\underset{t\rightarrow \infty}{\lim} \ \dfrac{1}{t}\int_{0}^{t}a^{(s)}(g_{\xi}(x))d\xi <0,
\end{equation}
\begin{equation}\label{lyapunov+}
\lambda^{u}(x):=\underset{t\rightarrow \infty}{\lim} \ \frac{\log||dg_{t}(x)|_{E^{u}(x)}||}{t}= \underset{t\rightarrow \infty}{\lim} \ \dfrac{1}{t}\int_{0}^{t}a^{(u)}(g_{\xi}(x))d\xi >0.
\end{equation}
If $\mu$ is a $g_t$-invariant measure, then by Birkhoff's ergodic theorem, these limits exist $\mu-$almost everywhere.\\
Let $\omega\in\Sigma_A$, and $x\in \Lambda$, such that $x=\chi(\omega)$. et $\tilde{a}^{(u)}$ and $\tilde{a}^{(s)}$ be the pull back of the functions $a^{(u)}$ and $a^{(s)}$ by the coding map $\chi$, defined by

\begin{equation}\label{tilde-a-us}
\tilde a^{(u)}(\omega,t):=a^{(u)}(g_t(\chi(\omega)))  \mbox{ and }\tilde a^{(s)}(\omega,t):=a^{(s)}(g_t(\chi(\omega))).
\end{equation}
Let also $\mathbf{a}^{(u)}$ and $\mathbf{a}^{(s)}$ be the H\"older continuous function on $\Sigma_{A}$ defined by
\begin{equation}\label{textbf-a-us}
\mathbf{a}^{(u)}(\omega)=\exp\int_{0}^{r(\omega)}\tilde{a}^{(u)}(\omega,t)dt, \quad \mathbf{a}^{(s)}(\omega)=\exp\int_{0}^{r(\omega)}\tilde{a}^{(s)}(\omega,t)dt.
\end{equation}
These are respectively the dilatation and contraction in the unstable and stable direction of the flow $(g_{t})_{t\in\mathbb{R}}$ between two consecutive visits in the Poincar\'e section (the configuration  at these two visits are coded by $\omega$ and $\sigma(\omega)$ respectively). This what the following lemma says
\begin{lemma}\label{lemmaX}
For any $\omega\in\Sigma_{A}$ such that $x=\chi(\omega)$, we have
\begin{equation*}
\mathbf{a}^{(u)}(\omega)=|D_{x}T_{|E^{u}}|\mbox{ and } \mathbf{a}^{(s)}(\omega)=|D_{x}T_{|E^{s}}|.
\end{equation*}
\end{lemma}
\begin{proof}It comes directly from \eqref{dgt au} and \eqref{dgt as} applied to $x\in X$ and $t=R(x)$.
\end{proof}

%\begin{proof}
%Using the definition of $a^{u}(x)$ in \eqref{au}, then for a small $t$ we have $|D_{x}g_{t}\vert E^{u}|\simeq e^{t a^{u}(x)}$, and since $a^{u}(.)$ is continuous, we get:
%\begin{eqnarray*}
%\mathbf{a}^{(u)}(\omega)&=& \exp\int_{0}^{r(\omega)}a^{(u)}(g_{s}(x))ds\\
%&=&\prod_{i=1}^{m}\exp\int_{s_{i}}^{s_{i}+t_{i}}a^{(u)}(g_{s}(x))ds
%=\prod_{i=1}^{m}\exp \left( t_{i}a^{(u)}(g_{s_{i}}(x))\right)\\
%&=& \prod_{i=1}^{m}|{D_{g_{s_{i}(x)}}g_{t_{i}}}_{|E^{u}}|= |{D_{x}g_{t}}_{|E^{u}}|=|D_{x}T_{|E^{u}}|.
%\end{eqnarray*}
%\end{proof}
As we will work with the Poincar\'e section, these quantities will naturally appear in our calculations.

\begin{proposition}[]\label{diamcyl}
There exist $c>0$  and $\alpha_{1}>0$, such that for any positive integer $k$  and any 2-sided $k$-cylinder  $C_{-k,k}(\omega)$,
\begin{equation*}
diam\  \chi\left( C_{-k,k}(\omega)\right)\leq c e^{-\alpha_{1}k}
\end{equation*}
\end{proposition}
The following lemma is a result of regularity of $\mathbf{a}^{(u)}$ and of  $\mathbf{a}^{(s)}$ which will be very useful later.
\begin{lemma}\label{bfauholder}
There exist $\alpha>0$, $\alpha'>0$, $c_{\mathbf{a}}^{u}>0$ and $c_{\mathbf{a}}^{s}>0$ such that, for every integer $k$, for any $\omega$, $\omega'$ such that $\hat{d}(\omega,\omega')\leq e^{-k}$,
\begin{equation}\label{alpha,cau,cas}
\mathbf{a}^{(u)}(\omega)\leq \mathbf{a}^{(u)}(\omega')\exp( c_{\mathbf{a}}^{u}e^{-\alpha k}),\quad
\mathbf{a}^{(s)}(\omega)\leq \mathbf{a}^{(s)}(\omega')\exp(c_{\mathbf{a}}^{s}e^{-\alpha' k})
\end{equation}
\end{lemma}
\begin{proof}
%we have
%$$\frac{\mathbf{a}^{(u)}(\omega)}{\mathbf{a}^{(u)}(\omega')}= \exp\left( \int_{0}^{r(\omega)}\tilde{a}^{(u)}(\omega,t)dt-\int_{0}^{r(\omega')}\tilde{a}^{(u)}(\omega',t)dt\right). $$
The function $\tilde{a}^{(u)}(.,t)$ is H\"older continuous. Indeed, we have $\omega'\in C_{-k,k}(\omega)$, then due to Proposition \ref{diamcyl}, there exist $c,\alpha_1>0$ such that $d(\chi(\omega), \chi(\omega'))\leq ce^{-\alpha_{1}k}$. And since $a^{(u)}$ is H\"older continuous, then there exist $c_{a}>0$ and $\alpha_{2}>0$ such that $|\tilde{a}^{(u)}(\omega)-\tilde{a}^{(u)}(\omega')|\leq c_{a}c e^{-\alpha_{1}\alpha_{2}k}$.
%\begin{eqnarray*}
%|\tilde{a}^{(u)}(\omega)-\tilde{a}^{(u)}(\omega')|&=&|a^{(u)}(\chi(\omega))-a^{(u)}(\chi(\omega'))|\\
%&\leq& c_{a}d(\chi(\omega),\chi(\omega'))^{\alpha_{2}}\\
%&\leq& c_{a}c e^{-\alpha_{1}\alpha_{2}k}.
%\end{eqnarray*}
Moreover, the height function $r$ is also H\"older continuous, then there is $c_{r}>0$ and $\alpha_{3}>0$ such that
\begin{align*}
   \frac{\mathbf{a}^{(u)}(\omega)}{\mathbf{a}^{(u)}(\omega')} = &\exp\bigg(\int_{0}^{r(\omega)}\tilde{a}^{(u)}(\omega,t)dt-\int_{0}^{r(\omega)}\tilde{a}^{(u)}(\omega',t)dt\\
%\mbox{ \hspace{0.9cm}}
 & +\int_{0}^{r(\omega)}\tilde{a}^{(u)}(\omega',t)dt-\int_{0}^{r(\omega')}\tilde{a}^{(u)}(\omega',t)dt\bigg)
%&\leq&\exp\left(r(\omega)c_{a}ce^{-\alpha_{1}\alpha_{2}k}+ \sup|\tilde{a}^{u}(\omega',t)|c_{r}d(\omega,\omega')^{\alpha_{3}}\right)\\
\leq \exp(c^u_{\mathbf{a}} e^{-\alpha k}), 
\end{align*}
% \begin{equation*}
% \frac{\mathbf{a}^{(u)}(\omega)}{\mathbf{a}^{(u)}(\omega')} =\exp\bigg(\int_{0}^{r(\omega)}\tilde{a}^{(u)}(\omega,t)dt-\int_{0}^{r(\omega)}\tilde{a}^{(u)}(\omega',t)dt
% %\mbox{ \hspace{0.9cm}}
%  +\int_{0}^{r(\omega)}\tilde{a}^{(u)}(\omega',t)dt-\int_{0}^{r(\omega')}\tilde{a}^{(u)}(\omega',t)dt\bigg)\\
% %&\leq&\exp\left(r(\omega)c_{a}ce^{-\alpha_{1}\alpha_{2}k}+ \sup|\tilde{a}^{u}(\omega',t)|c_{r}d(\omega,\omega')^{\alpha_{3}}\right)\\
% \leq \exp(c^u_{\mathbf{a}} e^{-\alpha k}),
% \end{equation*}
where $c^u_{\mathbf{a}}=\max(||r||_{\infty}cc_{a},\sup|\tilde{a}^{u}(.,t)|c_{r})$ and $\alpha=\min(\alpha_{1}\alpha_{2},\alpha_{3})$. We prove the second inequality of \eqref{alpha,cau,cas} analogously.
\end{proof}
%Given $r>0$ and a point $\omega\in C^{+}_{i_{0}}$, choose the number $n(\omega)$ such that
%\begin{equation}
%\Pi_{k=0}^{n(\omega)-1}(\mathbf{a}^{(u)}(\sigma^{k}\omega))^{-1}>r, \quad \Pi_{k=0}^{n(\omega)}(\mathbf{a}^{(u)}(\sigma^{k}\omega))^{-1}\leq r.
%\end{equation}
%It's easy to see that $n(\omega)\rightarrow\infty$ as $r\rightarrow0$ uniformly in $\omega$.

Let $x\in\Lambda$ such that $x=\chi(\omega)$, where $\omega=(\omega_{n})_{n\in\mathbb{Z}}\in\Sigma_{A}$. We define $l_m^s(\omega)$ and $l_n^u(\omega)$ to be the lengths of the stable and unstable manifolds passing by $x$ and contained in the cylinder $\chi(C_{-m,n})$ containing $x$. We denote by $a_{n}^{u}(x):=\displaystyle\prod_{i=0}^{n-1}\left(\mathbf{a}^{u}(\sigma^{i}(\omega))\right)^{-1}$ and $a_{m}^{s}(x):=\displaystyle\prod_{i=1}^{m}\mathbf{a}^{s}(\sigma^{-i}(\omega))$.

\begin{proposition}\label{PropESS}
    There exists a constant $C>0$ such that for all $m,n\geq0$, all cylinder $C_{-m,n}$, and for all $x=\chi(\omega)\in \chi(C_{-m,n})$, we have the following:
   % \begin{itemize}
   \begin{equation}\label{unstablelength}
\frac{1}{C}a_{n}^{u}(x)\leq l_n^u(\omega)\leq C a_{n}^{u}(x)
         \mbox{\quad and \quad } \frac{1}{C} a_{m}^{s}(x)\leq l_m^s(\omega)\leq C a_{m}^{s}(x),
    \end{equation}     
         in addition we have
         $$\chi(C_{-m,n})\subset B(x,C\max(a_{m}^{s}(x),a_{n}^{u}(x))),$$ and
         there exists $x_{0}$ such that $$B(x_{0},C^{-1}\min(a_{m}^{s}(x),a_{n}^{u}(x)))\subset \chi(C_{-m,n}).$$
    %\end{itemize}
\end{proposition}

\begin{proof}
%Consider $W^{u}(T^{n}x)$ the unstable manifold at $T^{n}x$.
Set $x=\chi(\omega)\in X$, where $\omega\in\Sigma_A$, and let $\Pi$ be a rectangle containing $x$. Let also $\Pi_{\omega_n}$ be the rectangle containing $T^n x$. We consider the local unstable manifold $W^{(u)}_{loc}(T^{n}x,\Pi_{\omega_n})$. Set $\gamma :[a,b]\rightarrow T^{-n}\left(W^{(u)}_{loc}(T^{n}x,\Pi_{\omega_n})\right)$, a curve on $M$, such that $|\dot{\gamma}(s)|=1$, whose length is $l_n^u (\omega)$.\\
%therefore $\ell(\gamma)=b-a$. First let us compute
%$$\ell(W)=\int_{a}^{b}|(D(T^{n}\circ \gamma(s)))|ds.$$
% \begin{equation*}
% D(T\circ V(s))=D_{V(s)}T(V(s))\dot{V}(s),
% \end{equation*}
Let $\omega'_{s}\in \Sigma_{A}$ such that $\gamma(s)=\chi(\omega'_{s})$, then $T^{i}(\gamma(s))=\chi(\sigma^{i}\omega'_{s})$, according to \eqref{commute}. Using the chain rule and Lemma \ref{lemmaX}, we get
$$\ell\left(W^{(u)}_{loc}(T^{n}x,\Pi_n)\right)=\int_{a}^{b}|(D(T^{n}\circ \gamma(s)))|ds=l_n^u (\omega)\prod_{i=0}^{n-1} \mathbf{a}^{(u)}(\sigma^{i}\omega'_{s}).$$
% \begin{eqnarray*}
% \ell(W)&=&\int_{a}^{b}|(D(T^{n}\circ \gamma(s)))|ds\\
% &=&\int_{a}^{b} |D_{T^{n-1}\gamma(s)}T(T^{n-1}(\gamma(s)))|.|D_{T^{n-2}\gamma(s)}T(T^{n-2}(\gamma(s)))|...|D_{\gamma(s)}T(\gamma(s))|.|\dot{\gamma}(s)|ds\\
% &=& \int_{a}^{b}\mathbf{a}^{(u)}(\sigma^{(n-1)}(\omega'_{s})).\mathbf{a}^{(u)}(\sigma^{(n-2)}(\omega'_{s}))...\mathbf{a}^{(u)}(\sigma(\omega'_{s})).\mathbf{a}^{(u)}(\omega'_{s})ds\\
% &=&\int_{a}^{b}\prod_{i=0}^{n-1} \mathbf{a}^{(u)}(\sigma^{i}\omega'_{s}).
% \end{eqnarray*}
% %For all $s\in[a,b]$, we have $\omega'_{s}\in C_{\omega_{0},\omega_{n}}$ and $\gamma(s)\in\chi(C_{\omega_{0},\omega_{n}})$ then $\sigma^{i}\omega'_{s}\in \sigma^{i}(C_{\omega_{0},\omega_{n}})$, and hence we get:
% \begin{equation}\label{distwwi}
% d(\sigma^{i}\omega,\sigma^{i}\omega'_{s})\leq e^{-( n-i)}.
% \end{equation}
Let us set the constant $C_{a}^{\alpha}:=\exp\left(\frac{ c_{a}}{1-e^{-\alpha}}\right)$. Using Lemma \ref{bfauholder}, then for all $s_1,s_2\in (a,b)$, one can show that
\begin{equation}\label{NNN}
\prod_{i=0}^{n-1} \mathbf{a}^{(u)}(\sigma^{i}\omega'_{s_1})\leq C_{a}^{\alpha}\prod_{i=0}^{n-1} \mathbf{a}^{(u)}(\sigma^{i}\omega'_{s_2}).
\end{equation}
%\begin{eqnarray*}
%\prod_{i=0}^{n-1}\mathbf{a}^{u}(\sigma^{i}(\omega'_{s}))
%&\leq& \prod_{i=0}^{n-1}\mathbf{a}^{u}(\sigma^{i}(\omega'_{r}))\exp c_{\mathbf{a}} e^{-\alpha(n-i)}\\
%&\leq&\prod_{i=0}^{n-1}\mathbf{a}^{u}(\sigma^{i}(\omega'_{r}))\exp c_{\mathbf{a}}\sum_{i=0}^{n-1}e^{-\alpha(n-i)},
%\end{eqnarray*}
%%%but $\sum_{i=0}^{n-1}e^{-\alpha\min(i,n-i)}=\sum_{i=0}^{n/2}e^{-\alpha i}+\sum_{i=n/2}^{n}e^{-\alpha(n-i)}=2\sum_{i=0}^{n/2}e^{-\alpha i}\leq\frac{2}{1-e^{-\alpha i}}$,
%%but $\sum_{i=0}^{n-1}e^{-\alpha(n-i)}\leq\frac{1}{1-e^{-\alpha i}}$,
%%then setting the constant $C:=\exp\left(\frac{ c_{a}}{1-e^{-\alpha}}\right)$, we get
%\begin{equation}
%\prod_{i=0}^{n-1}\mathbf{a}^{u}(\sigma^{i}(\omega'_{s}))\leq C \prod_{i=0}^{n-1}\mathbf{a}^{u}(\sigma^{i}(\omega'_{r}))
%\end{equation}
Applying \eqref{NNN}, once with $\omega'_{s_1}=\omega$, and then with $\omega'_{s_2}=\omega$, we get respectively an upper bound and a lower bound as follows
$$(C_{a}^{\alpha})^{-1} l_n^u (\omega)\prod_{i=0}^{n-1}\mathbf{a}^{u}(\sigma^{i}(\omega))\leq\ell\left(W^{(u)}_{loc}(T^{n}x,\Pi_n)\right)\leq C_{a}^{\alpha} l_n^u (\omega)\prod_{i=0}^{n-1}\mathbf{a}^{u}(\sigma^{i}(\omega).$$
% $$\ell(W)=\int_{a}^{b} \prod_{i=0}^{n-1}\mathbf{a}^{u}(\sigma^{i}(\omega'_{s}))ds\leq C \prod_{i=0}^{n-1}\mathbf{a}^{u}(\sigma^{i}(\omega))\int_{a}^{b}ds =C \ell(V)\prod_{i=0}^{n-1}\mathbf{a}^{u}(\omega),$$
% now applying \eqref{NNN}, with $\omega'_{s}=\omega$, we get this time a lower bound
% $$C^{-1} \ell(V)\prod_{i=0}^{n-1}\mathbf{a}^{u}(\sigma^{i}(\omega))\leq \int_{a}^{b} \prod_{i=0}^{n-1}\mathbf{a}^{u}(\sigma^{i}(\omega'_{r}))dr=\ell(W).$$
Moreover, there exists $C'>1$ such that $\ell\left(W^{(u)}_{loc}(T^{n}x,\Pi_{\omega_n})\right)$ is uniformly bounded from below by $\frac{1}{C'}$ and from above by $C'$. Then by setting $C:=C'C_{a}^{\alpha}$, we get the desired framing for the unstable length in \eqref{unstablelength}
$$\frac{1}{C}\prod_{i=0}^{n-1}\left(\mathbf{a}^{u}(\sigma^{i}(\omega))\right)^{-1}\leq l_{n}^{u}(\omega)\leq C \prod_{i=0}^{n-1}\left(\mathbf{a}^{u}(\sigma^{i}(\omega))\right)^{-1}.$$
 And the proof for the stable length follows exactly the same scheme up to replacing $T$ by $T^{-1}$.  \\
 The proof of the second part of the proposition, comes directly using Lemma 6.1 in \cite{Pesin.sadovskaya}.
\end{proof}

\begin{remarque}\label{remarqueLuLs}
As a consequence of Proposition \ref{PropESS} and  of formulas \eqref{tilde-a-us} and \eqref{textbf-a-us}, going back to the definitions of $l_n^s$ and $l_n^u$, we have:
$$\frac{1}{C}\exp\left(-\sum_{k=0}^{n-1}\int_{0}^{r(\sigma^k(\omega))}a^{(u)}(g_t(\chi(\sigma^{k}\omega)))dt\right)\leq l_{n}^{u}(\omega)\leq C\exp\left(-\sum_{k=0}^{n-1}\int_{0}^{r(\sigma^k(\omega))}a^{(u)}(g_t(\chi(\sigma^{k}\omega)))dt\right)$$ and $$\frac{1}{C}\exp\left(\sum_{k=0}^{m-1}\int_{0}^{r(\sigma^{-k}(\omega))}a^{(s)}(g_t(\chi(\sigma^{-k-1}\omega)))dt\right)\leq l_{m}^{s}(\omega)\leq C\exp\left(\sum_{k=0}^{m-1}\int_{0}^{r(\sigma^{-k}(\omega))}a^{(s)}(g_t(\chi(\sigma^{-k-1}\omega)))dt\right).$$
By Birkhoff's ergodic theorem, the limits 

\begin{equation}\label{LsLu}
    L^s=\underset{n\rightarrow+\infty}{\lim}\frac{\log l_n^s(\omega) }{n} \quad  and \quad L^u=-\underset{n\rightarrow+\infty}{\lim}\frac{\log l_n^u(\omega) }{n}
\end{equation}
exist $\nu$-almost everywhere on $\Sigma_A$. These are the expansion rates of the stable and unstable leaves respectively.
% from which we can get that $\frac{\log l_n^s }{n}$  and  $\frac{\log l_n^u }{ n}$  converge almost everywhere on $\Sigma_{A}$ (due to the Birkhoff theorem). We set $L^s$ and $L^u$ to be their limits respectively, which will appear in the statement of the first theorem:

\end{remarque}

 Recall that $R$ is the height function of the suspension flow $(\psi_{t})_t$ defined in Section \ref{suspensionflow}. We show that the Lyapunov exponents of the flow $g_t$ and those of the Poincar\'e transformation $T$, are related in the following way: 
\begin{proposition}\label{relationlLyapynov-flot-T}
For $L^s$ and $L^u$ defined in \eqref{LsLu}, and $\lambda^s$, $\lambda^u$ defined in \eqref{lyapunov-} and \eqref{lyapunov+} respectively, we have almost everywhere: 
    \begin{equation}
        \lambda^s= \dfrac{L^s}{\int_{X}Rd\nu_X}\quad \mbox{ and }  \lambda^u= \dfrac{L^u}{\int_{X}Rd\nu_X}. 
    \end{equation}
\end{proposition} 
\begin{proof}
Let $x\in X$ and $s\geq0$, then using the formula of the suspension flow in \eqref{psit}, $g_s(x)=\psi_s(x,0)=(T^nx,s-S_nR(x))$, where $n$ is the unique integer such that $S_{n}R(x)\leq s<S_{n+1}R(x)$. For all $k=0,..., n-1$, using the change of variable $\xi=s+S_kR(x)$, and that for all $i=1,...,k$ the points $(T^{n+i}x, 0)$ and $(T^{n+i-1}x, R(T^{n+i-1}x))$ are identified, we show that 
$$g_{s}(T^kx)=(T^{n+k}x, s-S_n R(T^kx) )=(T^nx,s-S_n R(x)+S_kR(x))=g_{\xi}(x),$$ 
and therefore we get:
\begin{eqnarray*}
    \displaystyle\sum_{k=0}^{n-1} \int_{0}^{R(T^kx)}a^{(u)}(g_{s}(T^kx))ds &=& \displaystyle\sum_{k=0}^{n-1}\int_{S_kR(x)}^{S_{k+1}R(x)}a^{(u)}(g_{\xi}(x))d\xi\\
    &=&\int_{0}^{S_nR(x)}a^{(u)}(g_{\xi}(x))d\xi
\end{eqnarray*}
Let $t$ be such that $t=S_nR(x)$, going back to the definition of $\lambda^u(x)$ in \eqref{lyapunov+}, and taking into consideration that $x=\chi(\omega)$, where $\omega\in\Sigma_A$, and that $T^k(x)=\chi(\sigma^k(\omega))$, we have: 
\begin{eqnarray*}
    \lambda^u (x)&=& \underset{t\rightarrow \infty}{\lim} \ \dfrac{1}{t}\int_{0}^{t}a^{(u)}(g_{\xi}(x))d\xi\\
    &=& \underset{n\rightarrow \infty}{\lim} \ \dfrac{1}{\dfrac{S_nR(x)}{n}}\dfrac{1}{n} \displaystyle\sum_{k=0}^{n-1} \int_{0}^{R(T^kx)}a^{(u)}(g_{s}(T^kx))ds\\
    &=& \dfrac{1}{\int_X Rd\nu_X}\underset{n\rightarrow \infty}{\lim} \ \dfrac{1}{n}\left(-\log(l_n^u(\omega))\right)=\dfrac{L^u}{\int_X Rd\nu_X} \quad \mbox{ a.e.,}
\end{eqnarray*}
where we recall that by Birkhoff's ergodic theorem $\underset{n\rightarrow+\infty}{\lim}\dfrac{-\log l_n^u(\omega) }{n}=L^u$ almost everywhere \eqref{LsLu} . We proceed similarly for the negative Lyapunov exponent.
\end{proof}
% \begin{equation}\label{lambda_s-lambda_u}
%    \lambda^{(u)}=\underset{n\rightarrow+\infty}{\lim}\frac{\log l_n^u }{S_n R(x)} \quad  and \quad \lambda^{(s)}=\underset{n\rightarrow+\infty}{\lim}\frac{\log l_n^s }{S_nR(x)}
% \end{equation}

Recall that $c_L$ is the Lipschitz constant  of $\{.,.\}$, we prove the following lemmas.
\begin{lemma}\label{lemma1}
Let $\omega\in \Sigma_{A}$, $r>0$ and $c_{L}>0$ be the Lipschitz constant of $\{.,.\}$. Consider the $(-q,q')$-cylinder $C_{-q,q'}(\omega)$ where $q$ and $q'$ are minimal such that $c_{L} l_{q}^{s}(\omega)\leq \frac{r}{2}$ and $c_{L}l_{q'}^{u}(\omega)\leq \frac{r}{2}$, then we have
\begin{equation*}
\chi(C_{-q,q'}(\omega))\subset B(\chi(\omega),r).
\end{equation*}
\end{lemma}
\begin{proof}
Set $x=\chi(\omega)$. Suppose that $x\in \text{int}\Pi_{\omega_{0}}$, where we consider $\Pi_{\omega_{0}},\Pi_{\omega_{1}},...,\Pi_{\omega_{n}}$  the Markov collection for $g_{t}$ on $\Lambda$. Each rectangle $\Pi_{\omega_{i}}$ is contained in a smooth disk $D_{i}$ and has the product structure $\{. , . \}$ as in \eqref{proj}.
%Due to \cite{ProjLipBarreira}, we have that
%\begin{equation}\label{lipshitz}
%\Pi_{i}\times\Pi_{i}\ni(x,y)\mapsto \{x,y\}\in \Pi_{i}
%\end{equation}
%is a Lipschitz map of Lipschitz constant $c_{L}$ and with Lipschitz inverse. Therefore,
Let $z$ and $z'$ be two points in $\Pi_{\omega_{i}}$, then there exist $x_{1}, x_{1}^{'}\in W^{u}_{loc}(x)$, $x_{2}, x_{2}^{'}\in W^{s}_{loc}(x)$ such that $z=\{x_{1},x_{2}\} \mbox{ and } z'=\{x'_{1},x'_{2}\},$ and we get
\begin{align}\label{l3eqn lemma cylinball}
 \nonumber   \text{diam} \left(\chi(C_{-q,q'}(\omega))\right)&= \max\left\{ d(z, z'), z,z'\in \bigcap_{i=-q}^{q'}T^{-i}(\Pi_{\omega_i})\right\}\\ \nonumber
    &\leq 2 c_{L}\left(d(x_{1},x'_{1}) + d(x_{2},x'_{2})\right)\\
&\leq c_{L} \left( l^{s}_{q}(\omega)+ l^{u}_{q'}(\omega)\right)\leq r,
\end{align}
% \begin{eqnarray*}
% \text{diam} (\chi(C_{-q,q'}))&=&\underset{ z,z'\in \bigcap _{i=-q}^{q'}T^{-i}(\Pi_{\omega_i})}{\max}\ d(z, z')\\
% &=&\underset{x_{1},x_{2},x'_{1},x'_{2}\in\Pi_{i}}{\max} d\left(\{x_{1},x_{2}\},\{x'_{1},x'_{2}\}\right)\\
% &\leq & c_{L}d\left((x_{1},x_{2}),(x'_{1},x'_{2})\right)\\
% &\leq& c_{L}\left(d(x_{1},x'_{1}) + d(x_{2},x'_{2})\right)\\
% &\leq& c_{L} \left( l^{s}_{q}(\omega)+ l^{u}_{q'}(\omega)\right)\leq r
% \end{eqnarray*}
which proves the lemma due to the conditions considered on $q$ and $q'$.
\end{proof}
\begin{lemma}\label{ballincylinder}
Let $\omega\in\Sigma_A$ such that $x=\chi(\omega)\in X$. Let $r>0$. For all $m,n\geq0$, consider the stable and unstable balls such that:
\begin{equation*}
B^{u}(x, r)\subset \chi(C_{0,n}(\omega)) \mbox{ and } B^{s}(x, r)\subset \chi(C_{-m,0}(\omega)),
\end{equation*}
then we have
\begin{equation*}
B\left(x, \frac{r}{c_{L}}\right)\subset \chi( C_{-m,n}(\omega)).
\end{equation*}
\end{lemma}
\begin{proof}
Let $z\in B\left(x, \frac{r}{c_{L}}\right)$. We have $z=\{z^{u},z^{s}\}$, with $z^{u}\in W^{u}_{loc}(x,\Pi)$ and $z^{s}\in W^{s}_{loc}(x, \Pi)$. Note that $z^s = \{z, x\}$, and 
$x=\{x,x\}$, we have $d^s(x, z^s)\leq c_L d(x,z)\leq r $. Similarly, we have $d^{u}(x,z^{u})\leq r$. Therefore, we get
\begin{equation*}  
z\in B^{u}(x, r)\times B^{s}(x, r)\subset\chi(C_{0,n}(\omega))\times \chi(C_{-m,0}(\omega))\subset\chi(C_{-m,n}(\omega)).
\end{equation*}
\end{proof}

%%%%%%%%%%%%%%%%%%% Lemma A From discussion with Benoît at ESI.
In the following lemma, we refer to Theorem 15 in \cite{Product-Barreira-Saussol}, which states that the $r$-neighborhood of the boundary of the Markov partition is at most polynomial in $r$. This also applies in the context of flows for the unstable boundary of an element of the Markov partition:  
\begin{lemma}\label{LemmaB}
    Let $\Lambda$ be a compact locally maximal hyperbolic set of a topologically mixing $C^{1+\alpha}$ hyperbolic flow $(g_t)_t$, and $\nu$ an equilibrium measure for a H\"older continuous function.
    For any element of the Markov partition $\Pi_i$ of $\Lambda$, there exists constants $c>0$ and $k>0$ such that if $r>0$ is small enough, then:
    $$\chi_{*}\nu\{x: d(x:\partial^{u}\Pi_i)<r\}<cr^k.$$
\end{lemma}
%%%%%%%%%%%%%%%%%%%%%%%%%%%%%%%%%%%%%%%%%%%%%%%%%%%% Lemma B

\begin{lemma}\label{LemmaC}
    Let $p>\dfrac{1}{k}$, where $k$ is the constant used in Lemma \ref{LemmaB}. For any $n\in\mathbb{Z}$, set $\delta_n=|n|^{-p}$, we have:
    $$\underset{n\in\mathbb{Z}}{\sum} \chi_{*}\nu\{x: d(T^{n}x:\partial^{u}\Pi_n)<\delta_n\}<\infty.$$
\end{lemma}
\begin{proof}
    Using Lemma \ref{LemmaB}, and the invariance of the measure $\chi_{*}\nu$ by $T$, we have:   
    \begin{align*}
        \underset{n\in\mathbb{Z}}{\sum} \chi_{*}\nu\{x: d(T^{n}x:\partial^{u}\Pi_n)<\delta_n\} &=\underset{n\in\mathbb{Z}}{\sum} \chi_{*}\nu\{x: d(x:\partial^{u}\Pi_n)<\delta_n\}\\
        &\leq \underset{n\in\mathbb{Z}}{\sum} c\delta_n^{k}=\underset{n\in\mathbb{Z}}{\sum} |n|^{-pk}<\infty.
    \end{align*}
\end{proof}

% \begin{lemma}
%     We have $\chi_{*}\nu$ almost every $x\in X$, there exists $n(x)$ such that for all $|n|>n(x)$
%     $$\Lambda\cap W^{u}_{\delta_n}(T^{n}x)\subset \Pi_{\omega_n}.$$
%     \begin{proof}
%         The proof follows from Lemma \ref{LemmaC}, using the first Borel Cantelli Lemma, $\chi_{*}\nu$ almost every $x\in X$, there exists $n(x)$ such that for all $|n|>n(x)$, $d(T^n x, \partial^u\Pi_n)>\delta_n$.
%     \end{proof}
% \end{lemma}

\begin{lemma}\label{LemmaD}
Let $\delta>0$. Let $x\in X$ such that $x=\chi(\omega)$, where $\omega\in \Sigma$. If $W^{u}_{\delta}(T^n x)\subset \Pi_{\omega_n}$, then 
$$T^{-n}(W^{u}_{\delta}(T^{n}(x)))\subset \chi(C^{+}_{\omega_0,...,\omega_n}) \quad \mbox{ and } \quad W^{u}_{l^{u}_n (\omega)\times \delta}(x)\subset T^{-n}(W^{u}_{\delta}(T^{n}(x))). $$
\end{lemma}
\begin{proof}
Let $\delta>0$ such that $W^{u}_{\delta}(T^n x)\subset \Pi_n$, then using the Markov property, for all $l\in\{0,1,...,n\}$, $T^{-l}(W^{u}_{\delta}(T^{n}x))\subset \Pi_{\omega_{n-l}}.$ Moreover,
    $$T^{-n}(W^{u}_{\delta}(T^{n}x))\subset T^{-j}(T^{-(n-j)}(W^{u}_{\delta}(T^{n} x)))\subset T^{-j}(\Pi_{\omega_{n-(n-j)}}),$$
    then for all $j=1,..., n$, $T^{-n}(W^{u}_{\delta}(T^{n }x))\subset T^{-j} \Pi_{\omega_j}$,
    from which we get $$T^{-n}(W^{u}_{\delta}(T^{n}(x)))\subset \chi(C^{+}_{\omega_0,...,\omega_n}).$$
\end{proof}
    \begin{lemma}\label{Lemmaballincylinder2}
        For $\nu$ almost everywhere $\omega\in \Sigma_{A}$, for any $\epsilon>0$ small enough, and for all $m$ and $n$ that are maximal such that $c_{L}\epsilon\leq\min( l^{u}_{n}(\omega)\times\delta_n, l^{s}_{m}(\omega)\times \delta_{m})$, where $\delta_n$ is defined in Lemma \ref{LemmaB}, the $(m,n)$-cylinder $C_{-m,n}(\omega)$ is such that
\begin{equation*}
B(\chi(\omega), r)\subset\chi(C_{-m,n}(\omega))
\end{equation*}
    \end{lemma}
\begin{proof}
Using Lemma \ref{LemmaC}, by the first Borel Cantelli Lemma, $\chi_{*}\nu$ almost every $x\in X$, there exists $n(x)$ such that for all $|n|>n(x)$, $d(T^n x, \partial^u\Pi{\omega_n})>\delta_n$, and so $\Lambda\cap W^{u}_{\delta_n}(T^{n}x)\subset \Pi_{\omega_n}.$
Take $n$ maximal such that $l^{u}_n (\omega) \times \delta_n\geq c_L \epsilon$, using Lemma \ref{LemmaD} applied to $\delta_n$, we have $W^{u}_{l^{u}_n (\omega)\times \delta_n}(x)\subset T^{-n}(W^{u}_{\delta_n}(T^{n}(x)))$, using again Lemma \ref{LemmaD}, for all $n>n(x)$, we have
    $$W^{u}_{l^{u}_n (\omega) \times \delta_n}(x)\subset \chi(C^{+}_{\omega_0,...,\omega_n}),$$
which in turn implies that $$\Lambda\cap B^u(x, l^{u}_n (\omega)\times \delta_n )\subset \chi(C^{+}_{\omega_0,\omega_1,...,\omega_n}).$$
In the same way, using the stable version of the preceding lemmas, we take $m$ maximal such that $l^{s}_m (\omega)\times \delta_m\geq c_L \epsilon$. Subsequently, it can be demonstrated that, for all $m>m(x)$,
$$\Lambda\cap B^s(x, l^{s}_m (\omega) \times \delta_m)\subset \chi(C^{-}_{\omega_{-m},\omega_{-1},...,\omega_0}).$$
Finally, by using Lemma \ref{ballincylinder}, we conclude that 
$$B(x,\epsilon)\subset\chi(C_{\omega_{-m},...,\omega_n}).$$
\end{proof}

\section{$\mathbb{Z}$-extension of an Axiom A flow}
Consider a Riemannian manifold $\tilde{M}$ endowed with a $\sigma-$finite measure $\tilde{\mu}$. Let $\tilde{g}_{t}:\tilde{M}\rightarrow \tilde{M}$ be a flow on $\tilde{M}$ preserving the measure $\tilde{\mu}$.\\ %be a probability measure on $M$ invariant by the flow. $d$ is the Hausdorff dimension of the measure $\mu$.\\

Let $I:\tilde{M}\circlearrowleft$ be an isometry of $\tilde{M}$ such that: $\Gamma=\{I^{n}, n\in\mathbb{Z}\}$ is an infinite group of isometries (i.e. $I^{n}=I^{m}\Rightarrow n=m$) and such that $I$ preserves $\tilde{\mu}$.
We suppose that:
\begin{itemize}
\item $M=\tilde{M}/\Gamma=\{\Gamma. \tilde{x}, \tilde{x}\in \tilde{M}\}$ is a compact manifold.
\item $\tilde{g}_{t}(I.\tilde{x})=I.\tilde{g}_{t}(\tilde{x})$
\end{itemize}
%%%%%%%%%%%%%%%%%%%%%%%%%%%%%%%%%%%%%%%%%%%%%%%%%%%%%%%%%%%%%%%%%%%%%%%%%%%%%%%%%%%%%%%
This ensures that we can define a flow $(g_{t})_{t}$ on $M$ by:
\begin{equation*}
g_{t}(\Gamma\tilde{x})=\Gamma\tilde{g}_{t}(\tilde{x}).
\end{equation*}
Moreover, we assume that $(M, (g_{t})_{t})$ is an Axiom A flow and that the measure $\mu$ defined on $M$ from the measure $\tilde{\mu}$ by passing through the quotient, is an equilibrium measure for $(g_{t})_{t}$. By construction of $\mu$, for every $ A\in \mathcal{B}(\tilde{M})$ such that $A\cap\bigcup_{k\in\mathbb{Z}^{*}}(I^{k}A)=\emptyset$, we have
\begin{equation*}
\mu(\{\Gamma.x, x\in A\})=\tilde{\mu}(A).
\end{equation*}
% \textcolor{red}{ Let $h_{\nu}$ be the entropy of $(X,T,\nu)$ the measure $\nu$. Let us recall that the entropy $h_\mu=h_\mu(g_1)$ of $(M,(g_t)_t,\mu)$ is given by the following formula
% \begin{equation}\label{h}
%     h_\mu:=\frac{h_{\nu}}{\int_{X}R d\nu_{X}}.
% \end{equation}
% %which will appear in the statement of one of the main results.\\
% }
% \textcolor{blue}{
% Let $h_\nu$ be the entropy of the invariant measure $\nu$ on the system $(\Sigma_A, \sigma)$. For flows, the entropy is defined as the entropy of the time-1 map. Thus, the entropy $h_\mu=h_\mu(g_1)$ of $(M,(g_t)_t,\mu)$ is given by the well known Abramov's formula \cite{Abramov}:
% %We recall the Abramov's formula which establishes a relation between the entropy  $h_\mu:=h_\mu(g_1)$ of the flow $(M,(g_t)_t,\mu)$ and the entropy $h_{\nu_X}$ of the invariant measure $\nu_X$ on $(X,T)$,
% % Let $h_{\nu_X}$ be the entropy of the invariant measure $\nu_X$ on $(X,T)$, for the Poincar\'e transformation. We recall that the Abramov's formula relates $h_{\nu_X}$ with the entropy of the flow $h_\mu:=h_\mu(g_1)$, as follows:
%  \begin{equation}\label{Abramov2}
%     h_\mu(g_t)=\frac{h_{\nu}(\sigma)}{\int_{\Sigma_A}rd\nu}=\frac{h_{\nu_X}(T)}{\int_{X}R d\nu_{X}},
% \end{equation}
%  which we will use in the proof of Theorem \ref{theorem1flow}.\\
% }

Our interest is to study the time needed for the flow $\tilde{g}_{t}$ to return back to an $\epsilon$-neighborhood of its starting point. Thus, for any $y\in \tilde{M}$, we define the return time of the flow $\tilde{g}_{t}$:
\begin{equation*}
\tau_{\epsilon}(y):=\inf\{t>1 \text{ : } \tilde{g}_{t}(y)\in B(y, \epsilon)\},
\end{equation*}
where $B(y,\epsilon)$ is the ball of center $y$ and radius $\epsilon$. We will prove the following results:
\begin{theoreme}\label{theorem1flow}
Let $(\tilde{M},\{\tilde{g}_{t}\},\tilde{\mu})$ be a flow satisfying all the hypotheses above. Then for $\tilde{\mu}$-almost every point $y\in \tilde{M}$,
\begin{equation*}
\underset{\epsilon\rightarrow0}{\lim}\ \frac{\log\sqrt{\tau_{\epsilon}}}{-\log\epsilon}=\dim_H \mu -1,
\end{equation*}

where  $\dim_H \mu$ is the Hausdorff dimension of the measure $\mu$.
% \begin{equation*}
% \underset{\epsilon\rightarrow0}{\lim}\ \frac{\log\sqrt{\tau_{\epsilon}}}{-\log\epsilon}=\frac{h}{L^{u}}+\frac{h}{L^{s}},
% \end{equation*}
% where $h$ is the entropy of the measure $\mu$ defined in \eqref{h}, and $L^{s}$ and $L^{u}$ the expansion rates of the stable and unstable leafs respectively defined in \eqref{LsLu}.
\end{theoreme}
Let $\tilde{y}\in \tilde{M}$, and $D_{0}(\tilde{y})$ a disk centered on $\tilde{y}$ and which is transversal to the flow, and orthogonal to it at $\tilde{y}$.\\
We will show a result of convergence in distribution for $\tau_{\epsilon}$. In the normalization appears a transversal measure $\tilde{\nu}_{0}^{y}$ on ${D}_{0}(\tilde{y})$ defined by:
%There exists a measure $\tilde{\nu}_{0}$ on $\tilde{D}_{0}$ such that
\begin{equation}\label{nu_0}
\tilde{\nu}_{0}^{y}(A)=\underset{\epsilon\rightarrow0}{\lim}\frac{1}{2\epsilon}\tilde{\mu}\left(\underset{-\epsilon<s<\epsilon}{\bigcup}\tilde{g}_{s}({D}_{0}(\tilde{y})\cap A)\right), \quad \forall A \text{ measurable } \subset {D}_{0}(y).
\end{equation}
The existence of this measure is proved later in Lemma \ref{findepreuve}. We call it measure transversal to the flow at $\tilde{y}$.
\begin{theoreme}\label{theorem2}
The family of random variables $\tilde{\nu}_{0}^{.}(B(.,\epsilon))^{2}\tau_{\epsilon}(.)$ converges in distribution, with respect to any probability measure absolutely continuous with respect to $\tilde{\mu}$ with density continuous and compactly supported, as $\epsilon\rightarrow0$ to $\sigma^{2}_{flow}\frac{\mathcal{E}^{2}}{\mathcal{N}^{2}}$, where $\mathcal{E}$ and $\mathcal{N}$ are independent random variables, $\mathcal{E}$ having an exponential distribution of mean 1 and $\mathcal{N}$ having a standard Gaussian distribution, and where $\sigma_{flow}^{2}$ is defined in Proposition \ref{propsigmaflow}.
\end{theoreme}
\begin{remarque}
We actually prove the result for a more general density (see the assumption of Proposition \ref{propthm}).
\end{remarque}
Let $M_{0}\subset\tilde{M}$ be a closed and connected fundamental domain  with the property that its closure is equal to its interior i.e. $
\bar{M_{0}}=\overline{\text{int}{M_{0}}} \mbox{ and that } \tilde{\mu}(\partial M_{0})=0.$
Moreover, assume that $M_{0}$ satisfies a condition in which the number of integers $k$, such that the distance between $M_0$ and $I^k M_0$ is less than the radius of injectivity $\tilde{r}_{inj}$ of $\tilde{M}$, is finite. This is denoted by \#$\{k\in \mathbb{Z}: (M_{0}, I^{k}M_{0})\leq \tilde{r}_{inj}\}<\infty$.
% Moreover, assume that $M_{0}$ is such that \#$\{k\in \mathbb{Z}: dist(M_{0}, I^{k}M_{0})\leq \tilde{r}_{inj}\}<\infty$, where we set $\tilde{r}_{inj}$ to be the radius of injectivity of $\tilde{M}$. 
Let $M_{L}>0$ be a contant defined by:
\begin{equation}\label{M_L}
M_{L}:=\sup\{|k|:d(M_{0}, I^{k}M_{0})\leq \tilde{r}_{inj}\}<\infty.
\end{equation}
%We take the notations and notions from chapter 4 for $(M, (g_{t})_{t}, \mu)$.
For every $\ell\in\mathbb Z$, we define the $\ell$-cell as the set $I^\ell M_0$. Let $y\in M$ and $t>0$.
We define  $\varphi_t(y)$ as the index of the cell containing $\tilde g_t(\tilde y_0)$ where $\tilde y_0$ is the representative of $y$ in $M_0$, that is  $\varphi_t(y)$ is the unique integer such that $\tilde g_t(\tilde y_0)$ is in $I^{\varphi_t(y)} M_0$.
We define $n_t(y)$ as the number of times $(g_s(y))_{s\in[0,t]}$ visits $X$.\\

For an $x\in X$, we define $\varphi_{X}(x)\in \mathbb{Z}$ to be the unique integer such that $\tilde g_{R(x)}(\tilde x_0)$ is in $I^{\varphi_{X}(x)} M_0$. We define also $\varphi:=\varphi_{X}\circ\chi$.\\
%\end{definition}
We assume  from now on that the Poincar\'e section for $(M, (g_{t})_{t})$ is taken such that $\epsilon_{0}<\epsilon_{0}'$ and strictly less than $\tilde{r}_{inj}$. Adapting $M_{0}$, we can suppose that the set of disks $\mathcal{D}$ of the Poincar\'e section  are the projections of disks contained in $M_{0}$,  by replacing $M_0$ by
\[
\left(M_{0}\cup\bigcup_{\tilde{D}\in\mathcal{D}:\tilde{D}\cap\partial M_{0}} \tilde{D}^{[\tilde{\epsilon}_{0}]}\right)\setminus\bigcup_{\tilde{D}\in\mathcal{D}:\tilde{D}\cap\partial M_{0}}\bigcup_{n\neq 0}I^n(\tilde{D}^{[\tilde{\epsilon}_{0}]}), 
\]
 with $\tilde{\epsilon_{0}}$ satisfying $\tilde{\mu}\left(\partial\left(\tilde{D}^{[\tilde{\epsilon_{0}}]}\right)\right)=0$, and that the projections  ${\Gamma.y \, :\, y\in\tilde D^{[\tilde\epsilon_0]}}$ are pairwise disjoint.

\begin{proposition}
The function $\varphi:\Sigma_{A}\rightarrow\mathbb{Z}$ is H\"older.
\end{proposition}
\begin{proof}
$M_{L}$ defined in \eqref{M_L} is finite, which implies that $\underset{x\in\tilde{M}, n\in\mathbb{Z}}{\inf}d(x, I^{n}x)\geq\epsilon_{0}'>0$. 
%%The problem is that when a disk of $M$ is the projection of another disk $\tilde{D}$ intersecting $M_{0}$ but $\tilde{D}$ is not contained in $M_{0}$.
%%$M_{0}$ is replaced by
%%\begin{equation*}
%%M_{0}\cup\tilde{D}^{[\tilde{\epsilon_{0}}]}\left(\bigcup_{n\neq0}I^{n}\left(\tilde{D}^{[\tilde{\epsilon_{0}}]}\right)\right),
%%\end{equation*}
%with $\tilde{\epsilon_{0}}$ satisfying:
%\begin{itemize}
%\item $\tilde{\mu}\left(\partial\left(\tilde{D}^{[\tilde{\epsilon_{0}}]}\right)\right)=0$,
%\item  the projections  ${\Gamma.y \, :\, y\in\tilde D^{[\tilde\epsilon_0]}}$ are pairwise disjoint.
%\end{itemize}
Due to the assumption taken on the Poincar\'e section, $\varphi$ is constant on each $D_{i}\cap T^{-1}D_{j}$ for $i\neq j$.
\end{proof}

Now, we define $\phi:X_{R}\rightarrow\mathbb{R}$, by: $\phi(x,u)=\dfrac{\varphi_X(x)}{R(x)}$ where $x\in X$ and $ u\in[0,R(x)[.$ Set $\sigma_{\varphi}^2=\sum\limits_{k\in\mathbb{Z}}\int_{X}\varphi_X.\varphi_X\circ T^{(k)}d\nu_{X}$. Due to \cite{Burton.Denker}, $\left(\frac{S_{n}\varphi}{\sqrt{n}}\right)_{n}$ converges in distribution with respect to $\nu$ to a Gaussian random variable centered with variance $\sigma_{\varphi}^{2}$.
We have the following lemma:

\begin{lemma}\label{CLTphi}
$\int_{0}^{t}\phi\circ \psi_{s}ds$ satisfies the Central Limit Theorem with the asymptotic variance $\sigma^{2}_{\phi}:=\frac{\sigma_{\varphi}^{2}}{\int_{X}R d\nu_{X}}$, where $\{\psi_{s}\}_{s\in\mathbb{R}}$ is the suspension flow defined in \eqref{psit}. That is, $\frac{1}{\sqrt{t}}\int_{0}^{t}\phi\circ \psi_{s}ds$ converges in distribution with respect to the measure $\mu$ as $t\rightarrow +\infty$ to a centered Gaussian random variable with variance $\sigma_\phi^2$.
\end{lemma}
\begin{proof}
$S_{n}\varphi$ and $S_{n}R$ satisfy the CLT using \cite{Burton.Denker} and $\int_{0}^{R(x)}\phi(x,u)du=\varphi_X(x)$, then we deduce using Theorem 1.1 in \cite{melbourneTorok} that $\phi$ satisfies the CLT with variance $\frac{\sigma^{2}_{\phi}}{\int_{X}R d\nu_{X}}$.
\end{proof}
\begin{lemma}
For every $(x,u)\in X_{R}$, there exists a real number $N_{L}>0$, such that for all $t>0$ and $\ell\in \mathbb{Z}$, we have
\begin{equation}\label{boundationNL}
\Big|\varphi_{t+u}(x)-\int_{0}^{t}\phi(\psi_{s}(x,u))ds\Big|\leq 2 \Vert\varphi_X\Vert_{\infty}+ N_{L}
\end{equation}
\end{lemma}
\begin{proof}
 For any $(x,u)\in X_R$ and any $t>0$, $n_{t+u}(x)$ is the number of visits to $X$ before $t+u$ starting from $x$ (or equivalently the 
number of visits to $X$ before $t$ starting from $g_u(x)$), we have
\begin{equation}\label{decomposition}
    \int_{0}^{t}\phi(\psi_{s}(x,u))ds %\int_{0}^{t}\phi(x,s+u)ds= \int_{u}^{t+u}\phi(x,s)ds= %\int_{0}^{t+u}\phi(x,s)ds - \int_{0}^{u}\phi(x,s)ds\\
    =\sum_{k=0}^{n_{t+u}(x)-1}\int_{S_{k}R(x)}^{S_{k+1}R(x)}\phi(x,s)ds 
    +\int_{S_{n_{t+u}(x)}R(x)}^{t+u}\phi(x,s)ds-\int_{0}^{u}\frac{\varphi_X (x)}{R(x)}ds.
\end{equation}
Using the definition of $\varphi_X$ and that the points $(x,s+S_{k}R(x))$ and $(T^k x,s)$ are identified, we have  
\[
\int_{S_{k}R(x)}^{S_{k+1}R(x)}\phi(x,s)ds= \int_{0}^{R(T^k x)}\phi(x,s+S_{k}R(x))ds= \int_{0}^{R(T^k x)}\phi(T^k x,s)ds= \varphi_X(T^k x),
\]
from which we get that the second term of \eqref{decomposition} is equal to $S_{n_{t+u}(x)}\varphi_X(x).$
Now, whenever $S_{n_{t+u}(x)}R(x)\leq s\leq t+u<S_{n_{t+u}(x)+1}R(x)$, it implies that $0\leq s-S_{n_{t+u}(x)}R(x)<R(T^{n_{t+u}(x)} x)$, then we have $\phi(x,s)$ which is identified with $\phi(T^{n_{t+u}(x)}x,s-S_{n_{t+u}(x)}R(x))$ is in turn equal to $\frac{\varphi_{X}(T^{n_{t+u}(x)}(x))}{R(T^{n_{t+u}(x)}(x))}$. 
Thus, we can conclude that the last two terms in \eqref{decomposition} are bounded by $\Vert\varphi_X\Vert_{\infty}$. 
Furthermore
\begin{align*}
\varphi_{t+u}(x)&=\varphi_{(S_{n_{t+u}(x)}R(x))}(x)+\varphi_{(t+u-S_{n_{t+u}(x)}R(x))}(T^{n_{t+u}(x)}(x))\\
&=S_{(n_{t+u}(x))}\varphi_X(x)+\varphi_{(t+u-S_{(n_{t+u}(x))}R(x))}(T^{n_{t+u}(x)}(x))\, .
\end{align*}
Setting the constant $N_{L}=\underset{x\in M}{\sup}\underset{t\leq R(x)}{\sup}|\varphi_{t}(x)|$, by combining these estimates together, we get the desired result.
\end{proof}

We aim to infer the statistical properties of $\varphi_{t}$ from the corresponding statistical properties of $\phi$ in the base dynamics. Thus we present the following proposition:
\begin{proposition}\label{propsigmaflow}
$\varphi_{t}$ satisfies the Central Limit Theorem with the variance $\sigma^{2}_{flow}:=\frac{\sigma^{2}_{\varphi}}{\int_{X}Rd\nu_{X}}$. That is, $\frac{1}{\sqrt{t}}\varphi_{t}$ converges in distribution with respect to the measure $\mu$ as $t\rightarrow +\infty$ to a centered Gaussian random variable with variance $\sigma_{flow}^2$.
\end{proposition}

\begin{proof}
Setting $\epsilon_{t}:=\frac{2 \Vert\varphi_X\Vert_{\infty}}{\sqrt{t}}+\frac{N_{L}}{\sqrt{t}}$, and using Lemma \ref{boundationNL}, we get $\left\vert\frac{\varphi_{t+u}(x)}{\sqrt{t}}-\frac{\int_{0}^{t}\phi(\psi_{s}(x,u))ds}{\sqrt{t}}\right\vert\leq \epsilon_{t}$, which goes to 0, as $t$ goes to infinity. Therefore, using Lemma \ref{CLTphi}, we deduce the CLT for $\varphi_{t}$.
\end{proof}
%%%%%%%%%%%%%%%%%%%%%%%%%%%%%%%%%%%%%%%%%%

%%%%%%%%%%%%%%%%%%%%%%%%%%%%%%%%%%%%%%%%%

\section{Proof of the almost sure convergence Theorem}
To simplify, we identify $x\in X$ with its representative in $M_{0}\subset \tilde{M}$. We want to know the time needed for the flow $\tilde{g}_{t}$ starting from a point $y\in \tilde{M}$ to return back into an $\epsilon$-neighborhood of this point, say $B(y,\epsilon)$. Setting $y=\tilde{g}_{s}(I^{n_{0}}x)$, the point $y$ corresponds to the coordinates $(x, s, n_{0})$, where $x$ is a point on the Poincar\'e section $X$ $\approx M_{0}\subset \tilde{M} $. Projecting on $D$, let $k_{D}=\underset{D}{\sup}\{ \mbox{constant Lipschitz of }P_{D}\}$, we can observe that:
\begin{equation}\label{inclusion1}
B_{D}\left(x,\frac{\epsilon}{k_{D}}\right)\subset P_{D}(B_{M}(y,\epsilon))\subset B_{D}\left(x,\epsilon k_{D}\right).
\end{equation}
%\begin{figure}
%\includegraphics[scale=0.5]{isomorphism.png}
%\caption{Isomorphism between geodesic flow and suspension flow via $\chi$}
%\label{isomorphism}
%\end{figure}

%\end{recall}
\begin{definition}
Let us consider the function $\varphi_{X}: X\rightarrow \mathbb{Z}$ introduced at the beginning of Section 3. Recall that $\sigma_{\varphi}^{2}$ is the asymptotic variance of $\varphi$. We consider the $\mathbb{Z}-$extension $\tilde{T}$ of $T$ by $\varphi_{X}$:
\begin{eqnarray*}
\tilde{T} : X \times \mathbb{Z} &\rightarrow& X \times \mathbb{Z}\\
(x,m)&\rightarrow& (T x, m+\varphi_{X}(x)).
\end{eqnarray*}
\end{definition}
%Let us denote by $A_{\epsilon(y)}$ the projection of the ball $B(y.\epsilon)$ on $D$, $A_{\epsilon}(y):=P_{D}(B(g_{s}(x),\epsilon))$. We define the first return time into $A_{\epsilon(y)}$ in the cell $n_{0}$ by:
We define the first return time of the transformation $\tilde{T}$ into $P_{D}(B(g_{s}(x),\epsilon))$ in the cell $n_{0}$ by:
\begin{equation*}
w_{\epsilon}(x)=\min\left\{k\geq1: T^{k}(x)\in P_{D}(B(g_{s}(x),\epsilon))\mbox{ and } \sum_{j=0}^{k-1}\varphi_X\circ T^{j}(x)=0 \right\}.
% =\min\{T^{k}x\in B_{X}(x,\epsilon) \mbox{ and } \sum_{j=0}^{k-1}\phi\circ T^{j}=0\}.
\end{equation*}
We recall the first return time of the flow $\tilde{g}_{t}$ to $B(y,\epsilon)$ for $y\in \tilde{M}$:
\begin{equation}
\tau_{\epsilon}(y)=\inf\{t>1, \tilde{g}_{t}(y)\in B(y,\epsilon)\}.
\end{equation}
\begin{lemma}\label{tauwR}
Let $y\in \tilde{M}$ and $x\in X$, such that $y=(x,s)$, where $s<R(x)$. The return time of the flow $\tilde{g}_{t}$ and that of the transformation $\tilde{T}$ are related in a manner such that, for almost every $y=(x,s)$
\begin{equation}\label{xsimy}
\tau_{\epsilon}(y)\sim w_{\epsilon}(x)\int_{X}R d\nu_{X}.
\end{equation}
\end{lemma}
\begin{proof}
We have
$$\sum_{k=0}^{w_{\epsilon}(x)-1}R\circ T^{k}(x)\leq \tau_{\epsilon}(y)+s<\sum_{k=0}^{w_{\epsilon}(x)}R\circ T^{k}(x),$$
from which we get that $\tau_{\epsilon}(y)+s=\sum_{k=0}^{w_{\epsilon}(x)-1}R\circ T^{k}(x)+s',$
%\begin{equation}\label{birkhoff}
%\tau_{\epsilon}(y)+s=\sum_{k=0}^{w_{A_\epsilon(y)}(x)-1}R\circ T^{k}(x)+s',
%\end{equation}
where $0\leq s'=s'_\epsilon(x)<R\circ T^{w_{\epsilon}(x)}\leq \Vert R\Vert_\infty$.\\
%$(X,T,\nu)$ is a probability space and $T:X\rightarrow X$ is an ergodic measure-preserving transformation,
Now using the Birkhoff Ergodic Theorem, for $\nu_{X}-$a.e. $x\in X$, we have:
\begin{equation}\label{birkhoff for roof function}
\underset{n\rightarrow\infty}{\lim}\frac{1}{n}\sum_{k=0}^{n-1}R(T^{k}(x))= \int_{X} R d\nu_{X}.
\end{equation}
Since, the height function $R$ is bounded from below and above by $R_{\min}$ and $R_{\max}$, $\big|\tau_{\epsilon}(y)-\sum_{k=0}^{w_{\epsilon}(x)-1}R\circ T^{k}\big|=|s-s'|\leq R_{\max}$, hence using \eqref{birkhoff for roof function}
\begin{equation*}
w_{\epsilon}(x)\int_{X} R d\nu_{X}\underset{\epsilon\rightarrow0}{\sim}\sum_{k=0}^{w_{\epsilon}(x)-1}R\circ T^{k}-R_{\max}\leq\tau_{\epsilon}(y)\leq\sum_{k=0}^{w_{\epsilon}(x)-1}R\circ T^{k}+R_{\max}\underset{\epsilon\rightarrow0}{\sim}w_{\epsilon}(x)\int_{X} R d\nu_{X},
\end{equation*}
since $w_{\epsilon}\rightarrow +\infty$  $\nu_{X}$-a.s., then as $\epsilon\rightarrow0$, we get \eqref{xsimy}.
%\begin{equation}\label{xsimy}
%\tau_{\epsilon}(y)\sim w_{A_\epsilon(y)}(x)\int_{X}R d\nu.
%\end{equation}
\end{proof}
%Therefore, now it will be enough to study the behavior of $w_{A_\epsilon(y)}(x)$ and conclude the results for $\tau_{\epsilon}(y)$.

Now, proving our results for $\tau_{\epsilon}(.)$ will be done by studying the asymptotic behavior of $w_{\epsilon}(.)$. Due to \eqref{inclusion1} and to Proposition \ref{PropESS}, we will start by studying the asymptotic behavior of the first return time into a cylinder, from which we can conclude the behavior of $w_{\epsilon}(.)$.
Thus we define the first return time in a $(-q,q')$-cylinder of a starting point $\omega\in\Sigma_{A}$ by:
\begin{equation*}
\tau_{q,q'}(\omega):=\min\{m\geq1: S_{m}\varphi(\omega)=0 \mbox{ and } \sigma^{m}(\omega)\in C_{-q,q'}(\omega)\}.
\end{equation*}
 Let us denote by $G_{n}(q,q')$ the set of points for which $n$ is a return time to the cylinder $C_{-q,q'}(\omega)$ :
\begin{equation*}
G_{n}(q, q'):=\{\omega\in \Sigma_{A} : S_{n}\varphi(\omega)=0 \text{ and }\sigma^{n}(\omega)\in C_{-q,q'}(\omega)\}.
\end{equation*}Thus, we have:
%Let us consider the first return time in a $(-q,q')$-cylinder of a starting point $x\in \Sigma$ :
\begin{equation*}
\tau_{q,q'}(\omega)=\min\{m\geq1 : \omega\in G_{m}(q,q')\}.
\end{equation*}
We have already mentioned that $(S_{n}\varphi)_{n}$ satisfies a central limit theorem. Furthermore it satisfies the following mixing local limit theorem:
\begin{proposition}\label{p3}[ see \cite{yassine}]
There exists a real number $C_{1}>0$ such that, for all integers $n,q ,q',k$ such that $n-2k\geq m_{0}$ and $m_{0}<q\leq k$, for all $(-q,q')$-cylinders $A$ of $\Sigma_A$ and all measurable subset $B$ of $\Sigma_A^+$, we have:
\begin{equation*}
\bigg|\nu\left(A\cap\{S_{n}\varphi=0\}\cap \sigma^{-n}(\sigma^{k}((\chi^+)^{-1}(B)))\right)-\frac{\nu(A)\nu^u(B)}{ \sqrt{2\pi}\sqrt{n-k}\sigma_{\varphi}}\bigg| \leq C_{1}\frac{\nu^u(B)k\nu(A)}{n-2k}.
\end{equation*}
\end{proposition}

Now, we study the behavior of $\tau _{q,q'}$, and prove the following lemma:

\begin{lemma}\label{lemmainf}
Let $a,b\geq0$ be such that $2(a+b)>\frac{1}{h_{\nu}}$, where $h_\nu:=h_{\nu}(\sigma)$ is the
%=\frac{h_{\mu}}{\int_{X}Rd\nu_{X}}$ defined in \eqref{h}
entropy of $(\Sigma_A,\sigma,\nu)$.
For any sequences $(q_{n})_{n}$ and $(q'_{n})_{n}$ such that $q_{n}\sim a\log n $, $q_{n}'\sim b\log n $, almost surely, we have:
\begin{equation}
\underset{n\rightarrow\infty}{\liminf}\frac{\log\sqrt{\tau_{q_{n},q_{n}'}}}{q_{n}+q_{n}'}\geq\frac{1}{2(a+b)}.\\
\end{equation}
\end{lemma}
\begin{proof}
Let us denote by $\mathcal{C}_{-q,q{'}}$ the set of $(-q,q')$-cylinders of $\Sigma$. Let $\delta>0$ be such that $\frac{1}{2}+(a+b)h_{\nu}> 1+\delta(a+b)$, denote by $\mathcal{C}_{-q,q{'}}^{\delta}\subset\mathcal{C}_{-q,q{'}}$ the set of cylinders $C \in \mathcal{C}_{-q,q{'}}$ such that $\nu(C)\in (e^{-(h_{\nu}+\delta)(q+q{'})},e^{-(h_{\nu}-\delta)(q+q{'})})$. For any $\omega\in \Sigma_{A}$, recall that  $C_{-q,q{'}}(\omega)\in \mathcal{C}_{-q,q'}$ is the $(-q,q')$-cylinder which contains $\omega$. By the Shannon-McMillan-Breiman theorem, the set $K_{N}^{\delta}=\{\omega\in \Sigma: \forall q,q' \geq N , C_{-q,q'}(\omega)\in  \mathcal{C}^{\delta}_{-q,q'}\}$ has a measure $\nu(K_{N}^{\delta})> 1-\delta$ provided $N$ is sufficiently large.
%\begin{itemize}
%\item First, let us prove that, almost surely :
%\begin{equation*}
%\underset{\epsilon \rightarrow 0}{\liminf}\frac{\log\sqrt{\tau_{q,q'}}}{q+q'}\geq h.
%\end{equation*}
According to Proposition \ref{p3}, whenever $n\geq N$, we have :
\begin{eqnarray*}
\nu(K_{N}^{\delta}\cap  G_{n}(q_{n},q_{n}'))
&=&\nu\left(\{\omega\in K_{N}^{\delta}:  S_{n}\varphi(\omega)=0 \mbox{ and } \theta^{n}(\omega)\in C_{-q_{n},q_{n}'}(\omega)\}\right)\\
&=&\sum_{C\in \mathcal{C}_{-q_{n},q_{n}'}^{\delta}}\nu(C\cap \{S_{n}\varphi=0\}\cap\theta^{-n}(C))\\
%&=&\sum_{C\in \mathcal{C}^{\delta}_{q_{n},q_{n}'}}\nu\left(C\cap \{S_{n}\varphi=0\}\cap\theta^{-n}\theta^{q_{n}}(\theta^{-q_{n}}(C))\right)\\
%&=&\sum_{C\in \mathcal{C}^{\delta}_{q_{n},q_{n}'}}\frac{\nu(C)\nu(C)}{\sqrt{2\pi}\sigma_{\varphi}\sqrt{n-2q_{n}}} + O\left(\frac{\nu(C)q_{n}\nu(C)}{n-q_{n}}\right)\\
&=&\sum_{C\in \mathcal{C}^{\delta}_{-q_{n},q_{n}'}}\left(\frac{\nu(C)^{2}}{\sqrt{2\pi}\sigma_{\varphi}\sqrt{n}} + O\left(\frac{\nu(C)q_{n}\nu(C)}{n} \right)\right).
\end{eqnarray*}
%Now we have :
%\begin{equation*}
%\frac{q_{n}^{2}e^{-\gamma (q_{n}+q_{n}')}}{n}=O\left(\frac{(\log n)^{2}}{n^{1+\gamma(a+b)}}\right).
%\end{equation*}
%Moreover, since $\nu(C)\leq e^{-(h-\delta)(q+q{'})}=n^{-(a+b)(h-\delta)}$, then
%\begin{equation*}
%\frac{\nu(C)}{\sqrt{n}}\leq \frac{1}{n^{\frac{1}{2}+(a+b)(h-\delta)}},
%\end{equation*}
Hence it follows that: $$\nu(K_{N}^{\delta}\cap G_{n}(q_{n},q_{n}'))
=O\left(\sum_{C\in \mathcal{C}_{-q_{n},q_{n}'}^{\delta}}\frac{\nu(C)^{2}}{\sqrt{n}}\right) =O\left(\frac{1}{n^{\frac{1}{2}+(a+b)(h_{\nu}-\delta)}}\right),
$$
%\begin{eqnarray*}
%\nu(K_{N}^{\delta}\cap  G_{n}(q_{n},q_{n}'))
%&=&O\left(\sum_{C\in \mathcal{C}_{-q_{n},q_{n}'}^{\delta}}\frac{\nu(C)^{2}}{\sqrt{n}}\right)\\ &=&O\left(\frac{1}{n^{\frac{1}{2}+(a+b)(h-\delta)}}\right),
%%&=&O\left(\frac{(\log n)^{2}}{n^{\min\{1+\gamma(a+b),\frac{1}{2}+(a+b)(h-\delta)\}}}\right).
%\end{eqnarray*}
but $\frac{1}{2}+(a+b)(h_{\nu}-\delta)>1$ then $\sum_{n}\nu(K_{N}^{\delta}\cap G_{n}(q_{n},q_{n}')) < \infty$.
Hence by the first Borel Cantelli Lemma, for a.e. $\omega\in K_{N}^{\delta}$, if $n$ is large enough, we get  $\tau_{q_{n},q_{n}'}> n$, proceeding similarly as at the end of the second item of section 3.2.\\
Moreover, we have the following implication
\begin{equation*}
\tau_{q_{n},q_{n}'}> n
\Rightarrow (a+b)\log\sqrt{\tau_{q_{n},q_{n}'}}>\frac{1}{2}\log n^{a+b}\sim\frac{1}{2}(q_{n}+q_{n}')\\
\end{equation*}
which in turn implies that
\begin{equation*}
\underset{n\rightarrow \infty}{\liminf}\ \frac{\log\sqrt{\tau_{q_{n},q'_{n}}}}{q_{n}+q'_{n}}\geq \frac{1}{2(a+b)} \quad a.e.,
\end{equation*}
%which proves the the lower bound on the $\liminf$, since $(\epsilon_{n})_{n}$ decreases to zero and $\liminf_{n\rightarrow+\infty}\frac{\epsilon_{n}}{\epsilon_{n+1}}=1$.
\end{proof}
%\item Second, let us prove that almost surely :
%\begin{equation*}
%\underset{\epsilon\rightarrow 0}{\limsup}\frac{\log\sqrt{\tau_{\epsilon}}}{-\log\epsilon}\leq d.
% \end{equation*}
\begin{lemma}\label{lemmasup}
Let $a,b \geq0$ such that $2(a+b)<\frac{1}{h_{\nu}}$, where again
%$h=\frac{h_{\mu}}{\int_{X}Rd\nu_{X}}$ defined in \eqref{h}
$h_\nu:=h_{\nu}(\sigma)$ is the entropy of $(\Sigma_A, \sigma,\nu)$. For any sequences of integers $(q_{n})_{n}$ and $(q'_{n})_{n}$ such that $q_{n}\sim a\log n $, $q_{n}'\sim b\log n $, we have almost surely:
\begin{equation*}
\underset{n\rightarrow\infty}{\limsup}\frac{\log\sqrt{\tau_{q_{n},q_{n}'}}}{q_{n}+q_{n}'}\leq\frac{1}{2(a+b)}.\\
\end{equation*}
\end{lemma}
\begin{proof}
We keep the same notations as in the proof of the previous lemma with $\delta>0$ such that $\frac{1}{2}+(a+b)h_{\nu}< 1+\delta(a+b)$.
For all $l=1,...,n$, we define:
\begin{equation*}
A_{l}(q_{n},q_{n}'):=G_{l}(q_{n},q_{n}')\cap \theta^{-l}\{\tau_{q_{n},q_{n}'}>n-l\}.
\end{equation*}
Let us take $L_{n}:=\lceil n^{a'}\rceil$, with $a'>2(a+b)(h_{\nu}+\delta-\gamma)$. The sets $A_{l}(q_{n},q_{n}')$ are pairwise disjoint thus:
\begin{equation*}
1= \sum_{l=0}^{n}\nu(A_{l}(q_{n},q_{n}'))\geq\sum_{l=L_{n}}^{n}\sum_{C\in \mathcal{C}^{\delta}_{q_{n},q_{n}'}}\nu(C\cap A_{l}(q_{n},q_{n}')).
\end{equation*}
Note that $C\cap\{\tau_{q_{n},q_{n}'}>n-l\}$ is a union of $C_{-q_n,n-l+q'_n}$-cylinders, then using Proposition \ref{p3}, for any $C\in\mathcal{C}^{\delta}_{q_{n}}$ provided $q_{n}\geq N$ and $l\geq L_{n}$, we get
\begin{eqnarray*}
\nu(C\cap A_{l}(q_{n},q_{n}'))
&=&\nu(C\cap\{S_{l}\varphi=0\}\cap\theta^{-l}(C\cap\{\tau_{q_{n},q_{n}'}>n-l\}))\\
%&=&\nu(C\cap\{S_{l}\varphi=0\}\cap\theta^{-l}\theta^{q_{n}}(\theta^{-q_{n}}(C\cap\{\tau_{q_{n},q_{n}'}>n-l\})))\\
%&=&\frac{\nu(C)\nu(C\cap\{\tau_{q_{n},q_{n}'}>n-l\})}{\sqrt{2\pi}\sigma_{\varphi}\sqrt{l-q_{n}}}+ O\left(\frac{\nu(C\cap\{\tau_{q_{n},q_{n}'}>n-l\})q_{n}\nu(C)}{l-q_{n}}\right)\\
&=&\left(\frac{\nu(C)}{\sqrt{2\pi}\sigma_{\varphi}}+O\left(\frac{q_{n}\nu(C)}{\sqrt{l-2q_{n}}}\right)\right)\frac{1}{\sqrt{l-q_{n}}}\nu(C\cap\{\tau_{q_{n},q_{n}'}>n-l\})\\
&\geq & c  n^{-(a+b)(h_{\nu}+\delta)}\frac{1}{\sqrt{l}}\nu(C\cap\{\tau_{q_{n},q_{n}'}>n-l\}).
\end{eqnarray*}
%for any $C\in\mathcal{C}^{\delta}_{q_{n}}$ provided $q_{n}\geq N$ and $l\geq L_{n}$. %indeed, since $a'>2(a+b)(d+\delta-\gamma)$, the error is negligible:
%\begin{eqnarray*}
%\frac{q_{n}^{2}e^{-\gamma (q_{n}+q_{n}')}}{\sqrt{l-q_{n}}}
%&\leq&\frac{a^2 (\log n)^2 n{^{-\gamma(a+b)}}}{n^{\frac{a'}{2}}}\\
%&=&O\left(\frac{\log n}{n^{\frac{a'}{2}+ \gamma(a+b)}}\right)
%\end{eqnarray*}
Note that
\begin{equation*}
\nu\left(K_{N}^{\delta}\cap\{\tau_{q_{n},q_{n}'}>n\}\right)\leq \sum_{C\in \mathcal{C}^{\delta}_{q_{n},q_{n}'}}\nu\left(C\cap\{\tau_{q_{n},q_{n}'}>n\}\right),
\end{equation*}
but, we observe that for any $C\in \mathcal{C}_{-q_{n},q_{n}'}^{\delta}$
$$\sum_{l=L_{n}}^{n}\nu(C\cap A_{l}(q_{n},q_{n}'))
%\geq c n^{-(a+b)(d+\delta)}\nu(C\cap\{\tau_{q_{n},q_{n}'}>n\})\sum_{l=L_{n}}^{n}\frac{1}{\sqrt{l}}\\
\geq  cn^{-(a+b)(h_{\nu}+\delta)}\nu(C\cap\{\tau_{q_{n},q_{n}'}>n\})\left(\sqrt{n}-\sqrt{L_{n}}\right),
$$
%\begin{eqnarray*}
%\sum_{l=L_{n}}^{n}\nu(C\cap A_{l}(q_{n},q_{n}'))
%&\geq& c n^{-(a+b)(d+\delta)}\nu(C\cap\{\tau_{q_{n},q_{n}'}>n\})\sum_{l=L_{n}}^{n}\frac{1}{\sqrt{l}}\\
%&\simeq & cn^{-(a+b)(d+\delta)}\nu(C\cap\{\tau_{q_{n},q_{n}'}>n\})\left(\sqrt{n}-\sqrt{L_{n}}\right)
%\end{eqnarray*}
so we obtain:
\begin{equation*}
1\geq \sum_{C\in \mathcal{C}^{\delta}_{q_{n},q_{n}'}}\sum_{l=L_{n}}^{n}\nu(C\cap A_{l}(q_{n},q_{n}'))\geq \sum_{C\in \mathcal{C}^{\delta}_{-q_{n}+q_{n}'}}n^{-(a+b)(h_{\nu}+\delta)}\nu(C\cap\{\tau_{q_{n},q_{n}'}>n\})\left(\sqrt{n}-\sqrt{L_{n}}\right),
\end{equation*}
from which one gets
\begin{equation*}
\sum_{C\in \mathcal{C}^{\delta}_{k_{n}}}\nu(C\cap\{\tau_{q_{n},q_{n}'}>n\})= O\left(\frac{1}{n^{\frac{1}{2}-(a+b)(h_{\nu}+\delta)}}\right).
\end{equation*}
Now let us take $n_{p}:=p^{-\frac{4}{1-2(a+b)(h_{\nu} +\delta)}}$. We have $\sum_{p\geq1}\nu(K_{N}^{\delta}\cap\{\tau_{q_{n_{p}},q_{n_{p}}'}>n_{p}\})<+\infty.$
%\begin{equation*}
%\sum_{p\geq1}\nu(K_{N}^{\delta}\cap\{\tau_{q_{n_{p}},q_{n_{p}}'}>n_{p}\})=\sum_{p\geq1}O\left(\frac{1}{n_{p}^{\frac{1}{2}-(a+b)(d+\delta)}}\right)=\sum_{p\geq1}O\left(\frac{1}{p^{2}}\right)<+\infty.
%\end{equation*}
Hence, by the first Borel Cantelli lemma, for $\nu$-almost everywhere $\omega\in K_{N}^{\delta}$, for every $p$ large enough, $\tau_{q_{n_{p}},q_{n_{p}}'}\leq n_{p}$.
%\begin{equation*}
%\tau_{q_{n_{p}},q_{n_{p}}'}\leq n_{p}
%%&\Rightarrow&\sqrt{ \tau_{q_{n_{p}},q_{n_{p}}'}}\leq \sqrt{n_{p}}\\
%%&\Rightarrow&\log\sqrt{ \tau_{\epsilon_{n_{p}}}}\leq \log\sqrt{n_{p}}\\
%%&\Rightarrow&\log\sqrt{ \tau_{\epsilon_{n_{p}}}}\leq \log\sqrt{n_{p}}\\
%%&\Rightarrow&-\alpha\log\sqrt{ \tau_{\epsilon_{n_{p}}}}\geq - \alpha\log\sqrt{n_{p}}\\
%%&\Rightarrow&-\alpha\log\sqrt{ \tau_{\epsilon_{n_{p}}}}\geq \log\epsilon_{n_{p}}\\
%\Rightarrow\frac{\log\sqrt{ \tau_{q_{n_{p}},q_{n_{p}}'}}}{q_{n_{p}},q_{n_{p}}'}\leq \frac{1}{2(a+b)},
%\end{equation*}
This implies that
\begin{equation*}
\underset{n\rightarrow +\infty}{\limsup}\frac{\log\sqrt{ \tau_{q_{n},q_{n}'}}}{q_{n}+q_{n}'}\leq \frac{1}{2(a+b)}.
\end{equation*}
\end{proof}
%This gives the estimate $\limsup$ since $(\epsilon_{n_{p}})_{p}$ decreases to 0 and since $\underset{p\rightarrow +\infty}{lim}\frac{\epsilon_{n_{p}}}{\epsilon_{n_{p+1}}}=1$.

\begin{proof}[Proof of Theorem \ref{theorem1flow}]
Let $y=g_s(x)$, with $x$ coded by an $\omega\in\Sigma_{A}$ (i.e. $x=\chi(\omega)$) and such that 
\[
\lim_{n\rightarrow +\infty}\frac{\log a_n^u(x)}{n}=
-L^u
=\mathbb E_{\nu}[\log \mathbf{a}^u]
\quad\mbox{and}\quad
\lim_{n\rightarrow +\infty}\frac{\log a_n^s(x)}{n}=L^s=\mathbb E_{\nu}[\log \mathbf{a}^s]\, ,
\]
and satisfying~\eqref{xsimy}.
Fix an $\epsilon>0$ and $s>0$. 
%We consider a point $x$ in $M$ 
Recall that, due to~\eqref{inclusion1}, 
\[B_D(x,\epsilon/k_D)\subset   P_{D}(B(g_{s}(x),\epsilon))\subset B_D(x,\epsilon k_D)\, .
\]
Using Proposition \ref{PropESS} and the notation $C$ therein, 
let $q_\epsilon$ and $q'_\epsilon$ be the smallest non-negative integers such that
\[
C\max\left(a_{q_\varepsilon}^s(x),a_{q'_\varepsilon}^u(x)\right)\le \varepsilon/k_D< C\min\left(a_{q_\varepsilon-1}^s(x),a_{q'_\varepsilon-1}^u(x)\right)\, .
\]
Using Lemma~\ref{Lemmaballincylinder2}, we fix $p_\epsilon$ and $p'_\epsilon$ the largest integers such that
\[ c_{L}\epsilon k_D\leq\min( l^{u}_{p'_{\epsilon}}(\omega)\times\delta_n, l^{s}_{p_{\epsilon}}(\omega)\times \delta_{m}).\]
% \[
% B^u(x,\epsilon k_D)\subset \chi\left(C_{0,p'_\varepsilon}(\omega)\right)\quad\mbox{and}\quad
% B^s(x,\epsilon k_D)\subset \chi\left(C_{-p_\varepsilon,0}(\omega)\right)\, .
% \]
%
%Let $q_\epsilon$ and $q'_\epsilon$ be the smallest integers such that $\chi(C_{-q_{\epsilon},{q'}_{\epsilon}}(\omega))  \subset B_D(x,\epsilon/k_D)$. 
% \textcolor{blue}{Now, referring to Lemma 6.1 in \cite{Pesin.sadovskaya}, $P_{D}(B(g_{s}(x),\epsilon))$ is contained in the union of the cylinders which intersect $B_D(x,\epsilon k_D)$. The number of these cylinders is bounded above by a constant $M$ which is independent from $x$ and $\epsilon$. This is true, due to the fact that each cylinder contains a ball of radius $k_1 \epsilon$, where $k_1$ is a constant independent of $\epsilon$.}
Observe that $P_{D}(B(g_{s}(x),\epsilon))$ is then controlled by 
\[\chi(C_{-q_{\epsilon},{q'}_{\epsilon}}(\omega))  \subset  
B_D(x,\epsilon/k_D)\subset P_{D}(B(g_{s}(x),\epsilon))\subset B_D(x,\epsilon k_D)\subset\chi(C_{-p_{\epsilon},{p'}_{\epsilon}}(\omega))\, .\]
Hence we have:
%%%%%%%%%%%%%%%%%
%According to Lemma \ref{lemma1}, let $q_\epsilon$ and $q'_\epsilon$ minimal such that $c_{L}l_{q_{\eps}}^{s}(x)\leq \frac{\epsilon}{2}$ and $c_{L} l_{q_{\epsilon}'}^{u}(x)\leq \frac{\epsilon}{2}$, we have   $$\chi(C_{-q_\epsilon,q'_\epsilon}(x))\subset B(x,\epsilon),$$
%and according to Lemma \ref{lemma2}, let $p_\epsilon$ and $p'_\epsilon$ maximal such that $c_{L} \eps\leq\min( l^{u}_{\delta_{p_{\eps}},p_\eps}(x), l^{s}_{\delta_{p'_\eps},p'_\eps}(x))$, we have
%$$B(x,\epsilon)\subset \chi(C_{-p_\epsilon,p'_\epsilon}(x)).$$
%Et ensuite utiliser les propriétés asymptotiques des couples $(q_\epsilon, q'_\epsilon), (m_\epsilon,n_\epsilon) quand  \epsilon $ tend vers 0 pour appliquer les résultats précédents du Chapitre 5.\\
\begin{equation}\label{3 returns}
\tau_{p_{\epsilon},p'_{\epsilon}} (\omega) < w_{\epsilon}(x) < \tau_{q_{\epsilon}, q'_{\epsilon}}(\omega).
\end{equation}
By definition, $\log a_{q_\epsilon}^s(x)\sim\log \epsilon\sim\log a_{q'_\epsilon}^u (x)$ and therefore
\[
\lim_{\epsilon\rightarrow 0}\frac{\log \epsilon}{ q_\epsilon}=
\lim_{\epsilon\rightarrow 0}\frac{\log a^s_{q_\epsilon}(x)}{ q_\epsilon}=L^s\quad\mbox{and}\quad\lim_{\epsilon\rightarrow 0}\frac{\log \epsilon}{ q'_\epsilon}=\lim_{\epsilon\rightarrow 0}\frac{\log a^u_{q'_\epsilon}(x)}{ q'_\epsilon}=
-L^u
\, .
\]
Furthermore, by definition of $p_\epsilon$ and 
$p'_\epsilon$,
\[
l_{1+p_\epsilon}^s(\omega)\times \delta_m<c_L\epsilon k_D<l_{p_\epsilon}^s(\omega)\times\delta_m
\quad\mbox{and}\quad
l_{1+p'_\epsilon}^u(\omega)\times\delta_n<\epsilon k_D<l_{p'_\epsilon}^u(\omega)\times\delta_n\, .
\]
and so
\[
\lim_{\epsilon\rightarrow 0}\frac{\log\epsilon}{p_\epsilon}
=\lim_{\epsilon\rightarrow 0}\frac{\log 
(l_{p_\epsilon}^s(\omega)\times\delta_m)}{p_\epsilon}=\lim_{\epsilon\rightarrow 0}\frac{\log 
l_{p_\epsilon}^s(\omega)}{p_\epsilon}=L^s\]
and
\[
\lim_{\epsilon\rightarrow 0}\frac{\log\epsilon}{p'_\epsilon}
=\lim_{\epsilon\rightarrow 0}\frac{\log 
(l_{p'_\epsilon}^u(\omega)\times\delta_n)}{p'_\epsilon}=\lim_{\epsilon\rightarrow 0}\frac{\log l_{p'_\epsilon}^u(\omega)}{p'_\epsilon}=
-L^u \, .
\]

Let $\eta\in(0,1)$, and $
\alpha_{q}=\frac{1-\eta}{2h_{\nu}}\frac{1}{\frac{1}{L^u}-\frac{1}{L^s}}.$ For $\epsilon>0$, large enough, we set $
a:=\frac{\alpha_{q}}{|L^s|}$, $
b:=\frac{\alpha_{q}}{L^u}$, and $N_{\eta,\epsilon}:=\lceil\log\epsilon^{\frac{1}{\alpha}}\rceil$. We have $a+b= \frac{1-\eta}{2h_{\nu}}<\frac{1}{2h_\nu}$, then using Lemma \ref{lemmasup} (applied to $a, b, N_{\alpha_{q},\epsilon}$), we have almost surely
%$q_{\epsilon}=\lceil a \log\epsilon^{\frac{1}{\alpha}}\rceil$, $q'_{\epsilon}=\lceil b \log\epsilon^{\frac{1}{\alpha}}\rceil$.  For $\epsilon>0$, large enough, we have

%$$q_{\epsilon}<\frac{(1-\eta)(\log\epsilon^{(1-\eta)^{-2}})}{L^s} \mbox{  and  }q'_{\epsilon}<  \frac{(1-\eta)  (\log\epsilon^{(1-\eta)^{-2}})}{L^u}$$,
%We set  $a_{\eta}= \frac{(1-\eta)}{L^s}$, $ b_\eta=\frac{(1-\eta)}{L^u}$ and $N_{\eta,\epsilon}:=\lceil\epsilon^{(1-\eta)^{-2}}\rceil$. We have $a_\eta+b_{\eta}<\frac{1}{L^u}-\frac{1}{L^s}$,

\begin{equation*}
\underset{\epsilon\rightarrow0}{\limsup}\frac{\log\sqrt{\tau_{q_{\epsilon},q_{\epsilon}'}}}{q_{\epsilon}+q_{\epsilon}'}\leq\frac{h_{\nu}}{1-\eta},
\end{equation*}

but $q_{\epsilon}+q'_{\epsilon}\leq \frac{(a+b)|\log\epsilon|}{\alpha}\leq 
\left(\frac{1}{L^u}-\frac{1}{L^s}\right)|\log\epsilon|$, which gives almost everywhere
%in this we can see that $\frac{1}{2(a + b)}=\frac{h}{1+\eta}$, and thus we get
\begin{equation}\label{limsup}
\underset{\epsilon\rightarrow0}{\limsup}\frac{\log\sqrt{\tau_{q_{\epsilon},q_{\epsilon}'}}}{|\log\epsilon|}\leq \frac{h_{\nu}}{1-\eta}
\left(\frac{1}{L^u}-\frac{1}{L^s}\right).
\end{equation}

%but $\frac{1}{2(a + b)}=\frac{h}{1-\eta}$, from which we obtain almost surely that
%\begin{equation*}
%\underset{\epsilon\rightarrow0}{\limsup}\frac{\log\sqrt{\tau_{q_{\epsilon},q_{\epsilon}'}}}{|\log\epsilon|}\leq h,
%\end{equation*}
%
On the other hand, for all $\epsilon>0$ large enough and $\eta\in(0,1)$, let $\alpha_{p}=\frac{1+\eta}{2h_{\nu}}\frac{1}{\frac{1}{L^u}-\frac{1}{L^s}}.$ Setting $a:=\frac{\alpha_{p}}{L^s}$, $b:=\frac{\alpha_{p}}{L^u}$, and $N_{\eta,\epsilon}:=\lceil\log\epsilon^{\frac{1}{\alpha}}\rceil$. As $a+b=\frac{1+\eta}{2h_{\nu}}>\frac{1}{2h_{\nu}}$, then using Lemma \ref{lemmainf} (applied to $a, b, N_{\alpha_{p},\epsilon}$), we have almost everywhere
\begin{equation*}
\underset{\epsilon\rightarrow0}{\liminf}\frac{\log\sqrt{\tau_{p_{\epsilon},p_{\epsilon}'}}}{p_{\epsilon}+p_{\epsilon}'}\geq\frac{h_{\nu}}{1+\eta}.
\end{equation*}
In the same way as previously, $p_{\epsilon}+p'_{\epsilon}\geq \frac{(a+b)|\log\epsilon|}{\alpha}\geq \left(\frac{1}{L^u}-\frac{1}{L^s}\right)|\log\epsilon|$, from which we get
\begin{equation}\label{liminf}
\underset{\epsilon\rightarrow0}{\liminf}\frac{\log\sqrt{\tau_{p_{\epsilon},p_{\epsilon}'}}}{|\log\epsilon|}\geq \frac{h_{\nu}}{1+\eta}\left(\frac{1}{L^u}-\frac{1}{L^s}\right).
\end{equation}

%Hence we have almost surely
%\begin{equation*}
%\underset{\epsilon\rightarrow0}{\liminf}\frac{(1+\eta)\log\sqrt{\tau_{p_{\epsilon},p_{\epsilon}'}}}{|\log\epsilon|}\geq\frac{1}{2(1+{\eta})},
%\end{equation*}

Therefore for all $\eta\in(0,1)$, due to \eqref{3 returns}, using \eqref{limsup}, \eqref{liminf} and Lemma \ref{tauwR}, we have the following two formulas, for almost everywhere
\begin{equation*}
\underset{\epsilon\rightarrow0}{\limsup}\frac{\log\sqrt{\tau_{\epsilon}}}{|\log\epsilon|}\leq\underset{\epsilon\rightarrow0}{\limsup}\frac{\log\sqrt{\tau_{q_{\epsilon},q_{\epsilon}'}}}{|\log\epsilon|}\leq\frac{h_{\nu}}{1-\eta}\left(\frac{1}{L^u}-\frac{1}{L^s}\right),
\end{equation*}

\begin{equation*}
\underset{\epsilon\rightarrow0}{\liminf}\frac{\log\sqrt{\tau_{\epsilon}}}{|\log\epsilon|}\geq\underset{\epsilon\rightarrow0}{\liminf}\frac{\log\sqrt{\tau_{p_{\epsilon},p_{\epsilon}'}}}{|\log\epsilon|}\geq\frac{h_{\nu}}{1+\eta}\left(\frac{1}{L^u}-\frac{1}{L^s}\right).
\end{equation*}

Now using the relation between $L^u$, $L^s$ and the Lyapunov exponents of the flow $(g_t)_t$ (see Proposition \ref{relationlLyapynov-flot-T} ), and the Abramov's formula \eqref{Abramov2}, we get

\begin{equation*}
\underset{\epsilon\rightarrow0}{\lim}\ \frac{\log\sqrt{\tau_{\epsilon}}}{|\log\epsilon|}=h_{\nu}\left(\frac{1}{L^u}-\frac{1}{L^s}\right)=h_{\mu}\left(\dfrac{1}{\lambda^u} -\dfrac{1}{ \lambda^s}\right), \quad \mbox{ almost everywhere.} 
\end{equation*}

It follows from Theorem 6.2 in \cite{barreira} that this limit coincides with the Hausdorff dimension $\dim_H\nu$ of the measure $\nu$, i.e. with $\dim_H\mu-1$.

%{\color{red} It follows from Ledrappier and Young's formula \cite{LY1,LY2} that this limit coincides with the Hausdorff dimension $\dim_H\nu$ of the measure $\nu$, i.e. with $\dim_H\mu-1$.}
\end{proof}

\section{Proof of the convergence in distribution Theorem}

Given $\epsilon>0$, we fix $\theta_{\epsilon}=\frac{\epsilon}{|\log\epsilon|}$. We construct a partition $\mathcal{D}_{\epsilon}$ of $\Sigma_{A}$ of diameter $<\theta_{\epsilon}$. Considering $c_{g}>0$ the Lipschitz constant of $(x,s)\mapsto g_{s}(x)$ on $X\times[0,\beta_{0}]$, $c_{L}>0$ the Lipschitz constant in Proposition \ref{barreiracL}, and the constant $C>0$ in Proposition 1.1.6. Let $\omega=(\omega_{n})_{n\in\mathbb{Z}}\in\Sigma_{A}$, we set $n_{1}(\omega)=n_{1}(\omega,\epsilon)$, and $n_{2}(\omega)=n_{2}(\omega,\epsilon)$, as follows:
$$n_{1}(\omega):=\min\left\{n\geq1 \mbox{ s.t.} \underset{\underset{\omega'_{0}=\omega_{0},..,\omega'_{n}=\omega_{n}}{\omega'\in\Sigma_{A}}}{\sup}\prod_{k=0}^{n}\mathbf{a}^{(u)}(\sigma^{k}\omega')^{-1}<\frac{\theta_{\epsilon}}{2Cc_{g}c_{l}}\right\},$$
and $$n_{2}(\omega):=\min\left\{n\geq1 \mbox{ s.t.} \underset{\underset{\omega'_{-1}=\omega_{-1},..,\omega'_{-n}=\omega_{-n}}{\omega'\in\Sigma_{A}}}{\sup}\prod_{k=-n}^{-1}\mathbf{a}^{(s)}(\sigma^{-k}\omega')<\frac{\theta_{\epsilon}}{2Cc_{g}c_{l}}\right\}.$$ We set the following quantities:
\begin{equation}\label{MepsQeps}
M_{\epsilon}:=2\ \underset{\omega\in\Sigma_{A}}{\max}\ n_{1}(\omega)+1 \mbox{. }  Q_{\epsilon}:= \sup_{\omega\in\Sigma_{A}}n_2(\omega) \mbox{ and define } M_{\epsilon}:=2Q_{\epsilon}+1.
\end{equation}

One can verify that $n_{1}(\omega)$ and $n_{2}(\omega)$ are well defined and finite. Thus for all $\omega\in\Sigma_{A}$, the family of cylinders $D_{\epsilon}(\omega)=C_{-n_{2}(\omega),n_{1}(\omega)}(\omega)$ forms a finite partition $\mathcal{D}_{\epsilon}=\{D_{\epsilon,i}, i=1,...,N_{\epsilon}\}$ of $\Sigma_{A}$. Now, we define the family of sets $(\mathcal{P}_{\epsilon,i,j})_{\epsilon,i,j}$, such that 
\[
\mathcal{P}_{\epsilon,i,j}=\left\{g_{s}(\chi(x)), (x,s)\in D_{\epsilon,i}\times\left[\frac{j\theta_{\epsilon}}{c_{g}},\frac{(j+1)\theta_{\epsilon}}{c_{g}}\right[\right\},
\] and for any element $y_{\epsilon,i,j}\in\mathcal{P}_{\epsilon,i,j}$, there exists $\omega_{\epsilon,i}\in D_{\epsilon,i}$ the smallest element in the lexicographic order such that $y_{\epsilon,i,j}=\chi(\omega_{\epsilon,i},j\theta_{\epsilon})$.

\begin{lemma}\label{diamP}
$\forall \epsilon>0,$ $i=1,...,N_{\epsilon}$ and $j\geq0$, the family $(\mathcal{P}_{\epsilon,i,j})_{\epsilon,i,j}$ satisfies
$$\text{diam}(\mathcal{P}_{\epsilon,i,j})<\theta_{\epsilon}.$$
\end{lemma}
\begin{proof}
Let $\omega,\omega'\in\Sigma_{A}$, due to \eqref{l3eqn lemma cylinball}, there exists $c_{L}>0$ such that $$\text{diam}(\chi(D_{\epsilon,i}))\leq c_{L}(l^{s}_{n_{2}}(\omega_{\epsilon,i})+l^{u}_{n_{1}}(\omega_{\epsilon,i}))
\leq c_{L}\left(\frac{\theta_{\epsilon}}{2c_{g}c_{L}}+\frac{\theta_{\epsilon}}{2c_{g}c_{L}}\right)<\frac{\theta_{\epsilon}}{c_{g}},$$
from which we get that $$d(g_{s}(\chi(\omega), g_{s'}(\chi(\omega'))))\leq c_{g}\max(|s-s'|, d(\chi(\omega), \chi(\omega')))<\theta_{\epsilon}.$$
\end{proof}
\begin{lemma}\label{ballincylpartition}
For all $\epsilon>0$, there exists $K>0$ such that and for every $i=1,...,N_{\epsilon}$, there exists $\omega_{i}:=\omega_{i,\epsilon}\in\Sigma_{A}$ such that
\begin{equation*}
B\left(\chi(\omega_{i}),\frac{\theta_{\epsilon}}{K}\right)\subset \chi(D_{\epsilon,i}).
\end{equation*}
\end{lemma}
\begin{proof}
%For all $\omega\in\Sigma_{A}$, set $q=n_{1}(\omega)$, and $q'=n_{2}(\omega)$. Using the definition of $n_{1}(\omega)$, then there exists $\omega'$ such that $$\prod_{k=0}^{n_{1}(\omega)}\mathbf{a}^{u}(\sigma^{k}(\omega'))^{-1}<\frac{\theta_{\epsilon}}{2Cc_{g}c_{L}}\leq\prod_{k=0}^{n_{1}(\omega)-1}\mathbf{a}^{u}(\sigma^{k}(\omega'))^{-1},$$ %%%%%%%%%%%%%%%%%%%
%%which implies that $$\prod_{k=0}^{n_{1}(\omega)-1}\mathbf{a}^{u}(\sigma^{k}(\omega'))^{-1}>\frac{\theta_{\epsilon}}{2Cc_{g}c_{L}},$$
%from which we get, using the definition of $l^{u}_{q}$ that there exists $c_{1}=\frac{\mathbf{a}^{u}(\sigma^{k}(\omega'))^{-1}}{2Cc_{g}c_{L}}>0$, such that
%$$l^{u}_{q}(\omega')\geq c_{1}\theta_{\epsilon},$$
%hence there exists $\omega^{u}\in W^{u}_{q}(\omega')$ such that
%$$B^{u}\left(\omega^{u}, \frac{c_{1}\theta_{\epsilon}}{3}\right)\subset W^{u}_{q}(\omega').$$
%Similarly, using the definition of $n_{2}(\omega)$ and that of $l^{s}_{q'}$, there exists $\omega^{3}\in\Sigma_{A}$, $c_{2}>0$ such that
%$l^{s}_{q'}(\omega^{3})\geq c_{2}\theta_{\epsilon}$ from which in turn, we have that there exists $\omega^{s}\in W^{s}_{q'}(\omega^{3})$ such that
%$$B^{s}\left(\omega^{s},\frac{c_{2}\theta_{\epsilon}}{3}\right)\subset W^{s}_{q'}(\omega).$$
%Let $k>\frac{3c_{L}}{\min(c_{1},c_{2})}$, then using Lemma \ref{ballincylinder}, there exists $\omega_{i}\in \Sigma_{A}$ such that $\chi(\omega_{i})=\{\omega^{s},\omega^{u}\}$ and $$B\left(\chi(\omega_{i}),\frac{\theta_{\epsilon}}{k}\right)\subset C_{-q,q'}.$$

For all $\omega\in\Sigma_{A}$, set $q=n_{1}(\omega)$, and $q'=n_{2}(\omega)$. Using the definition of $n_{1}(\omega)$ and that of $l^{u}_{q}$, there exist $\omega'$ and $c_{1}=\frac{\mathbf{a}^{u}(\sigma^{k}(\omega'))^{-1}}{2Cc_{g}c_{L}}>0$, such that
$l^{u}_{q}(\omega')\geq c_{1}\theta_{\epsilon},$
thus in turn there exists $\omega^{u}\in W^{u}_{q}(\omega')$ such that
$$B^{u}\left(\omega^{u}, \frac{c_{1}\theta_{\epsilon}}{3}\right)\subset W^{u}_{q}(\omega').$$
Similarly, using the definition of $n_{2}(\omega)$ and that of $l^{s}_{q'}$, there exist $\omega''\in\Sigma_{A}$, $c_{2}>0$ such that
$l^{s}_{q'}(\omega'')\geq c_{2}\theta_{\epsilon}$, and $\omega^{s}\in W^{s}_{q'}(\omega'')$ such that
$B^{s}\left(\omega^{s},\frac{c_{2}\theta_{\epsilon}}{3}\right)\subset W^{s}_{q'}(\omega).$
Let $k>\frac{3c_{L}}{\min(c_{1},c_{2})}$. Set $\omega_{i}\in \Sigma_{A}$ such that $\chi(\omega_{i})=\{\omega^{s},\omega^{u}\}$, then using Lemma \ref{ballincylinder}, we get $$B\left(\chi(\omega_{i}),\frac{\theta_{\epsilon}}{k}\right)\subset C_{-q,q'}.$$
which finishes the proof.

\end{proof}

Set $\mathcal{P}_{\epsilon}=\{\mathcal{P}_{\epsilon,i,j}\mbox{ : } i=1,...,N_{\epsilon}; j\geq0, D_{\epsilon,i}\times[j\theta_{\epsilon},(j+1)\theta_{\epsilon}[\subset \Delta\}$. We will consider the elements of $\mathcal{P}_{\epsilon}$ which intersect the boundary of $\Delta$, thus we consider the set \[
\bar{\mathcal{P}}_{\epsilon}:=\{\mathcal{P}_{\epsilon,i,j}\mbox{ : } i=1,...,N_{\epsilon}; j\geq0, D_{\epsilon,i}\times[j\theta_{\epsilon},(j+1)\theta_{\epsilon}[\cap\partial\Delta\neq\emptyset\}.
\]
In the following lemma, we show that the measure of the union of the elements of $\bar{\mathcal{P}}_{\epsilon}$ can be neglected in our study.
\begin{lemma}
The $\nu$-measure of the union of the elements of $\bar{\mathcal{P}}_{\epsilon}$ is $O(\theta_{\epsilon})$.
\end{lemma}

\begin{proof}

Let $P\in \bar{\mathcal{P}}_{\epsilon}$. Then according to Lemma \ref{diamP}, $P$ is contained in the $\theta_{\epsilon}$-neighborhood of the graph of $r:\Sigma_{A}\rightarrow[0,+\infty]$. So
\begin{equation*}
    \bigcup_{P\in \bar{\mathcal{P}}_{\epsilon}} P\subset\{(\omega,s)\in\Delta: r(\omega)-\theta_{\epsilon}<s\leq r(\omega)\},
\end{equation*}   
whose measure (w.r.t. $\mu_{\Delta}$) is bounded by $\dfrac{\theta_{\epsilon}}{\int_{\Sigma_{A}}rd\nu}$.
\end{proof}
From now on, let us set $A_{\epsilon}(y):=\chi^{-1}(P_{D}(B(y,\epsilon)))$. We denote by $A_{\epsilon-\theta_{\epsilon}}(y_{\epsilon,i,j})$ and $A_{\epsilon+\theta_{\epsilon}}(y_{\epsilon,i,j})$ the projections of the balls $B(y_{\epsilon,i,j},\epsilon-\theta_{\epsilon})$ and $B(y_{\epsilon,i,j},\epsilon+\theta_{\epsilon})$  on $\Sigma_{A}$ respectively, we set $A_{\epsilon}^{-}=\cup_{k\in I}D_{\epsilon,k}$ and $A^{+}_{\epsilon}=\cup_{l\in J}D_{\epsilon,l}$, where $I=\{k: D_{\epsilon,k}\subset A_{\epsilon-\theta_{\epsilon}}(y_{\epsilon,i,j})\}$ and
$J=\{l:D_{\epsilon,l}\cap A_{\epsilon+\theta_{\epsilon}}(y_{\epsilon,i,j})\neq\emptyset\}$.
Then we have the following inclusions, for $y\in\mathcal{P}_{\epsilon,i,j} $:
\begin{equation}\label{inclusionshift}
A_{\epsilon}^{-}(y_{\epsilon,i,j})\subset A_{\epsilon-\theta_{\epsilon}}(y_{\epsilon,i,j})\subset A_{\epsilon}(y)\subset A_{\epsilon+\theta_{\epsilon}}(y_{\epsilon,i,j})\subset A_{\epsilon}^{+}(y_{\epsilon,i,j}).
\end{equation}
Let $y\in M$, if $\epsilon>0$ is small enough, $y\in \bigcup_{P\in\mathcal{P}_{\epsilon}}P$, there exists $i,j$ such that $y\in \mathcal{P}_{\epsilon,i,j}$. Consider the ball $B(y,\epsilon)$, and let $y_{\epsilon,i,j}\in\mathcal{P}_{\epsilon,i,j}$, we have:
\begin{equation}
B(y_{\epsilon,i,j},\epsilon-\theta_{\epsilon})\subset B(y,\epsilon)\subset B(y_{\epsilon,i,j},\epsilon+\theta_{\epsilon}).
\end{equation}
This comes from the fact that the diameter of $\mathcal{P}_{\epsilon,i,j}$ is bounded by $\theta_{\epsilon}$. Furthermore, after projecting on $\Sigma_{A}$, let $x_{\epsilon,i}\in X$ such that $x_{\epsilon,i}=\chi(\omega_{\epsilon,i})$, we note that there exist $\alpha_{1}$, $\alpha_{2}>0$, such that:
\begin{equation}\label{alpha1alpha2}
\nu(\chi^{-1}(B(x_{\epsilon,i}, \alpha_{1}\epsilon)))\leq \nu(A_{\epsilon}(y_{\epsilon,i,j}))\leq \nu(\chi^{-1}(B(x_{\epsilon,i},\alpha_{2}\epsilon))).
\end{equation}

In what follows, we consider the notations 
$$A_{\epsilon}^{-}(y):=A_{\epsilon}^{-}(y_{\epsilon,i,j})\mbox{, } A_{\epsilon}^{+}(y):=A^{+}_{\epsilon}(y_{\epsilon,i,j}) \mbox{ and } D_{\epsilon}(x):=D_{\epsilon,i},$$ 
where $x$ is the point of $X$ such that $y=g_{s}(x)$ for some $s\in[0,R(x)[$ and if $y\in B(y_{\epsilon,i,j},\epsilon)$.

%\begin{theorem}\label{theorem2}
%The sequence of random variables $\frac{\nu(A_{\epsilon}(.))}{\int Rd\nu}\sqrt{\tau_{\epsilon}(.)}$ converges in distribution with respect to a probability measure absolutely continuous with respect to $\mu$ as $\epsilon\rightarrow0$ to $2\sigma^{2}_{flow}\frac{\mathcal{E}}{|\mathcal{N}|}$, where $\mathcal{E}$ and $\mathcal{N}$ are independent random variables, $\mathcal{E}$ having an exponential distribution of mean 1 and $\mathcal{N}$ having a standard Gaussian distribution.
%\end{theorem}
%
%\begin{corollary}\label{corollary}
%If the measure $\nu$ is not the measure of maximal entropy, then the sequence of random variables $\frac{\log\sqrt{\tau_{\epsilon}}+(d-1)\log{\epsilon}}{\sqrt{-\log\epsilon}}$ converges in distribution as $\epsilon\rightarrow 0$ to a centered Gaussian random variable of variance $2\sigma^{2}_{h}$.
%\end{corollary}

%\begin{proposition}
%Let $\nu$ be a Gibbs measure associated to a H\"older continuous potential $\varphi: X\rightarrow\mathbb{R}$. There exists $k_{\varphi}\geq1$ such that $\forall$ $C_{-q,q'}$ a $(q,q')-$cylinder, $\forall$ $ x\in C_{-q,q'}$, we have
%\begin{equation}\label{gibbs}
%\frac{1}{k_{\varphi}}\leq\frac{\nu(C_{-q,q'}(x))}{\exp\left(\sum_{k=-q}^{q'}\varphi \circ T^{k}(x)\right)}\leq k_{\varphi}.
%\end{equation}
%\end{proposition}
\begin{proposition}
Recall that a measure $\nu$ on a metric space is called \textbf{Federer} if there exists a constant $c_{F}>0$ such that for any point $x$ and any $r>0$,
\begin{equation*}
\nu(B(x,2r))\leq c_{F}\nu(B(x,r)).
\end{equation*}
\end{proposition}
\begin{remarque}\label{federer remark}
Due to Theorem 5.1 in \cite{Pesin.sadovskaya}, the measure $\nu$ is Federer. By a direct induction, for all $m\geq0$,
\begin{equation}
\nu(B(x,2^{m}r))\leq c_{F}^{m}\nu(B(x,r)).
\end{equation}
\end{remarque}

%\section{Proof of the convergence in distribution}
We will show a result of convergence in distribution for $\tau_{\epsilon}$. We will start by proving the two following lemmas:

\begin{lemma}\label{divideboundary}
There exist $\eta_{0}>0$ and $\epsilon>0$ such that if $\frac{\delta}{\epsilon}<\eta_{0}$, for any $x\in (\partial A_{\epsilon}(y))^{[\delta]}$, then there exists $z\in \partial A_{\epsilon}(y)$, such that $z\in W^{u}(x)$ or $W^{s}(x)$.
% such that $d(x,z)\leq \bar{c}\delta$ and
\end{lemma}
\begin{proof}

This comes from the transversality of $W^{s}(x)$ and $W^{u}(x)$ and the fact that $\partial A_{\epsilon}(y)$ is $C^{1}$ of order $\frac{1}{\epsilon}$.

\end{proof}
\begin{lemma}\label{equivalent}
For all $\epsilon>0$, the $\nu$-measures of the sets $A_{\epsilon}^{+}(y)$ and $A_{\epsilon}^{-}(y)$ are equivalent.
\end{lemma}
\begin{proof}
%\begin{eqnarray*}
%But $\chi(A_{\epsilon})$ is a diffeomorphic image of a circle $C_{\epsilon}$ of radius $\epsilon$ and center $x_{\epsilon}$. $\chi(A_{\epsilon})=\phi(C_{\epsilon})$ where $\phi:D\rightarrow D$ a diffeomorphism such that $||\phi||<k_{\phi}$.
%%%%%
%A_{\epsilon}^{+}(y_{\epsilon})\backslash A_{\epsilon}^{-}(y_{\epsilon})&=&\bigcup_{i}\{ D_{i}: D_{i}\cap A_{\epsilon}\neq\emptyset \texttt{ and } D_{i}\subset A_{\epsilon} \}\\
%&=& \bigcup_{i}\{ D_{i}: D_{i}\cap \partial A_{\epsilon}\neq\emptyset \}\\
%&=& \left(\partial A_{\epsilon}^{[\nu]}\right).
%\end{eqnarray*}
%First we will consider the part where $W^{u}_{loc}(x)$ cuts $B_{X}\left(\chi\left(A_{\epsilon}(y_{\epsilon})\right), \delta\right)$ excluding the case when it's tangent. Then to cover the whole domain, we will similarly consider that with $W^{s}_{loc}(x)$.
 %there exists a constant $c>0$ such that $\ell(I^{u})\leq c \delta$.
%On the other hand, let $I^{s}:=W^{s}_{loc}(x)\cap \left(\partial A_{\epsilon}(y_{\epsilon})\right)^{[\delta]}$. Due to the transversal intersection between the stable and unstable manifold, there exists a minimal angle $\eta_{0}>0$ such that $\left(W^{u}(x), W^{s}(x)\right)=\eta_{0}$. Then we have $\ell (I^{s})<cos(\eta)c\delta$.
We know that the cylinders of $\mathcal{D}_{\epsilon}$ are of diameter less than $\theta_{\epsilon}$, and the projection map on $\Sigma_{A}$ is Lipschitz with constant $k_{L}$. Set $\delta=(k_{L}+1)\theta_{\epsilon}$ and let $k_{0}$ be the smallest positive number satisfying $k_{L} \epsilon +\delta+ \epsilon\leq k_{0}\epsilon$. By the coding map $\chi$, we have on the Poincar\'e section $X$ that:
\begin{eqnarray*}
A_{\epsilon}^{+}(y)\backslash A_{\epsilon}^{-}(y)&=&A_{\epsilon}(y)^{[(k_L+1)\theta_{\epsilon}]}\cap (X\backslash A_{\epsilon}(y))^{[(k_L+1)\theta_{\epsilon}]}\\
&=&\left(\partial A_{\epsilon}(y)\right)^{[\delta]},
\end{eqnarray*}
thus we need to prove that the measure of the $\delta$-neighborhood of $A_{\epsilon}(y)$ is negligible with respect to the measure of $A_{\epsilon}(y)$. Indeed, let $m_{F}>0$ be such that $k_{0}\leq 2^{m_{F}}\alpha_{1}$, then using the Federer property (Remark \ref{federer remark}), and the inclusion in \eqref{alpha1alpha2}, there exists $c_{F}>0$ such that:
\begin{equation}\label{usingFederer}
\nu(\chi^{-1}(B(x, k_{0}\epsilon)))\leq c_{F}^{m_{F}}\nu(\chi^{-1}(B(x,\alpha_{1}\epsilon)))\leq c_{F}^{m_{F}}\nu(A_{\epsilon}(x)),
\end{equation}
hence it's sufficient to prove that $\nu\left(\left(\partial A_{\epsilon}(y)\right)^{[\delta]}\right)$ is negligible with respect to $\nu\left(\chi^{-1}(B(x, k_{0}\epsilon))\right)$.\\
%Setting $I=W^{u}(x)\cap A_{\epsilon}(y_{\epsilon})$, we will show that the measure of $I^{u}$ is negligible with respect to that of $I$.
%First we will consider the zone $D^{u}$ which is the part of $A_{\epsilon}$ along the unstable direction such that $\forall x\in D^{u}$, $d(W^{u}(x),W^{u}(x_{\epsilon}))<(1-a)\epsilon$, $\delta\ll a\epsilon$.\\
%Let $C_{-q,...,+\infty}$ a cylinder with $q$ minimal such that $cl^{s}_{q}\leq \epsilon$,  then from Lemma \ref{lemma1} we see that
%$$\chi(C_{-q,..., +\infty})\subset I.$$ \\
%%
%(We have that $I^{u}$ is contained in a union of cylinders $\chi\left(C_{-\infty,...,q_{i}}\right)$.)
%such that the elements $\omega_{-q}, \omega_{-q+1},...$ of $\chi\left(C_{(-q+l),..., +\infty}\right)$ coincide with that of $\chi\left(C_{-q,..., +\infty}\right)$.\\
%(Let $q_{i}$ be the smallest integer such that $I^{u}\subset\bigcup_{i=1}^{M}\chi\left(C_{-\infty,..., q_{i}}\right)$, where $M\leq2$. Let $k_{0}$ such that $k\epsilon+\delta+\epsilon\leq k_{0}\epsilon$, we have)
%\begin{equation}
%I^{u}\subset\bigcup_{i=1}^{M}\chi\left(C_{-\infty,..., q_{i}}\right)\subset B(x,k_{0}\epsilon)
%\end{equation}
Now, due to Lemma \ref{divideboundary}, we will work on two parts of $\partial A_{\epsilon}(y)^{[\delta]}$:\\
Let $x_{0}$ be a point in $D$. Let $x\in W^{s}(x_{0})$, we set $I^{u}(x):=W^{u}_{loc}(x)\cap \left(\partial A_{\epsilon}(y)\right)^{[\delta]}$, such that $\forall \xi_{x}\in I^{u}(x)$, there exists $z\in \partial A_{\epsilon}$ as in Lemma \ref{divideboundary} and where $z\in W^{u}(\xi_{x})$. On the other hand, we set $I^{s}(x):=W^{s}_{loc}(x)\cap \left(\partial A_{\epsilon}(y)\right)^{[\delta]}$, such that $\forall \xi_{x}'\in I^{s}(x)$, there exists $z'\in \partial A_{\epsilon}$ as in Lemma \ref{divideboundary} and where $z'\in W^{s}(\xi_{x}')$.\\
Hence, we have $\partial A_{\epsilon}(y)^{[\delta]}=\mathcal{I}^{u}\cup\mathcal{I}^{s}$, where
\begin{equation*}
\mathcal{I}^{u}:=\bigcup_{x\in W^{s}(x_{0})}I^{u}(x) \mbox{ and } \mathcal{I}^{s}:=\bigcup_{x\in W^{s}(x_{0})}I^{s}(x).
\end{equation*}
Our aim is to estimate the measure of $\mathcal{I}^{u}$ with respect to that of $B(x, k_{0}\epsilon)$. We will start working on $\nu^{u}(\pi^{+}(\chi^{-1}(I^{u})))$ and then after integrating over $\Sigma_{A}^{-}\times\Sigma_{A}^{+}$, we will conclude the estimation of $\nu(\chi^{-1}(\mathcal{I}^{u}))$.\\
Let $q_{i}$ be the greatest integer such that $l^{u}_{q_{i}}(\omega_{i})>3c\delta$, then we have:
\begin{equation*}
I^{u}(x)\subset \bigcup_{i=1}^{2}\chi(C_{-\infty,q_{i}}(\omega_{i}))\subset B(x,k_{0}\epsilon),
\end{equation*} and so $\nu^{u}(\pi^{+}(\chi^{-1}(I^{u})))\leq \sum_{i=1}^{2}\nu^{u}(C_{0,q_{i}}(\pi^{+}(\omega_{i})))$.
Knowing that $W^{u}(x)=\chi\left(C_{-\infty,0}(\omega)\right)=\bigcap_{i=1}^{\infty}\chi\left(C_{-i, 0}(\omega)\right)$, let $p_{i}$ be the smallest integer such that
\begin{equation*}
    \displaystyle\prod_{k=0}^{p_i -1}\left(\mathbf{a}^{u}(\sigma^{k}(\omega_i)\right)^{-1}> k_{0}\epsilon, \quad \displaystyle\prod_{k=0}^{p_i}\left(\mathbf{a}^{u}(\sigma^{k}(\omega_i)\right)^{-1} \leq k_{0}\epsilon
\end{equation*}
then using Proposition \ref{PropESS}, we have
\begin{equation}
\chi\left(C_{-\infty, p_{i}}(\omega)\right)\subset W^{u}(x)\cap B(x,k_{0}\epsilon),
\end{equation}
but $\nu^{u}(C_{0,p_{i}}(\pi^{+}(\omega_{i})))\leq \nu^{u}(\pi^{+}(\chi^{-1}(B(x, k_{0}\epsilon)\cap W^{u}(x))))$, thus, since $\nu^u$ is a Gibbs measure (see \eqref{measure-u} and \eqref{gibbs.formula}), by Proposition \ref{gibbs} there exists a constant $k_{\tilde{h}}\geq1$, and knowing that $\max{\tilde{h}}<-\gamma$ for some $\gamma>0$, we get:
\begin{eqnarray*}
\frac{\nu^{u}(\pi^{+}(\chi^{-1}(I^{u})))}{\nu^{u}(\pi^{+}(\chi^{-1}(B(x, k_{0}\epsilon)\cap W^{u}(x))))}&\leq& \sum_{i=1}^{2}\frac{\nu^{u}(C_{0,q_{i}}(\pi^{+}(\omega_{i})))}{\nu^{u}(C_{0,p_{i}}(\pi^{+}(\omega_{i})))}\\
&\leq& \sum_{i=1}^{2} k^{2}_{\tilde{h}} \exp\left(\sum_{k=p_{i}}^{q_{i}-1}\tilde{h}(\sigma^{k}(\omega_{i}))-(p_i -q_i)P_\sigma(\tilde{h})\right)\\
&\leq& \sum_{i=1}^{2} k^{2}_{\tilde{h}} \exp((-\gamma+P_{\sigma}(\tilde{h}))(q_{i}-p_{i}))\\
&\leq&\eta(\epsilon),
\end{eqnarray*}
where $\eta(\epsilon)\rightarrow0$ as $\epsilon$ goes to 0, since $q_{i}=o(p_{i})$, and $p_{i}\rightarrow\infty$.
Set $\tilde{I^{u}}=\{(\omega^{-}, \omega^{+}): \exists \omega=(\omega_{n})_{n\in\mathbb{Z}}\in \Sigma_{A} \mbox{ s.t.} \chi(\omega)\in I^{u}; \omega^{-}=(\omega_{n})_{n\leq0} \mbox{ and } \omega^{+}= (\omega_{n})_{n\geq 0}\}$, and in same way we define the set $\tilde{B}(x, k_{0}\epsilon)$. We have the following:
\begin{eqnarray*}
\nu^{u}\otimes\nu^{s}(\chi^{-1}(\mathcal{I}^{u}))&=&\int_{\Sigma_{A}^{-}}\left(\int_{\Sigma_{A}^{+}}\mathds{1}_{\tilde{I}^{u}({\omega}^{-},{\omega}^{+})}d\nu^{u}(\omega^{+})\right)d\nu^{s}(\omega^{-})\\
&=& \int_{\Sigma_{A}^{-}}\nu^{u}(\{\omega^{+}: (\omega^{-},\omega^{+})\in \tilde{I}^{u} \})d\nu^{s}(\omega^{-})\\
&\leq& \eta(\epsilon)\int_{\Sigma_{A}^{-}}\nu^{u}(\{\omega^{+}: (\omega^{-},\omega^{+})\in \tilde{B}(x,k_{0}\epsilon) \})d\nu^{s}(\omega^{-})\\
&\leq& \eta(\epsilon) \nu^{u}\otimes\nu^{s}(\chi^{-1}(B(x,k_{0}\epsilon)))
\end{eqnarray*}
On the other hand, the measure $\nu$ is absolutely continuous with respect to the product measure $\nu^{u}\otimes\nu^{s}$ with bounded density function $h$ (see Proposition 2.3 in \cite{Pesin.sadovskaya}), then there exist $\kappa_{1}$, $\kappa_{2}>0$ such that, using the previous inequality, we get:
\begin{eqnarray*}
\nu(\chi^{-1}(\mathcal{I}^{u}))&=& \int_{\Sigma_{A}}\mathds{1}_{\chi^{-1}(\mathcal{I}^{u})}h(\omega)d(\nu^{u}\otimes\nu^{s})(\omega)\\
&\leq& \kappa_{2} \nu^{u}\otimes\nu^{s}(\chi^{-1}(\mathcal{I}^{u}))\\
&\leq& \kappa_{2} \eta(\epsilon)\nu^{u}\otimes\nu^{s}(\chi^{-1}(B(x,k_{0}\epsilon)))\\
&\leq&\tilde{\eta}(\epsilon)\nu(\chi^{-1}(B(x,k_{0}\epsilon))).
\end{eqnarray*}
where $\tilde{\eta}(\epsilon)=\frac{\kappa_{2}}{\kappa_{1}}\eta(\epsilon)$.\\
Proceeding with the same strategy as above, there exists $\bar{\eta}(\epsilon)$ such that $\nu(\chi^{-1}(\mathcal{I}^{s}))\leq \bar{\eta}(\epsilon)\nu(\chi^{-1}(B(x,k_{0}\epsilon)))$ and this finishes the proof since $\tilde{\eta}(\epsilon)$ and $\bar{\eta}(\epsilon)$ goes to 0 as $\epsilon\rightarrow0$.
%$$\nu(I^{u})\leq\sum_{i=1}^{M}e^{-\kappa_{2}(l_{i})}\nu(\chi(C_{-q'_{i},..., +\infty}))\leq \underset{i}{\max}\ e^{-\kappa_{2}(l_{i})}\sum_{i}\nu(\chi(C_{-q'_{i},..., +\infty}))$$
%\begin{equation}
%\frac{l^{u}_{q'_{i}}}{l^{u}_{q_{i}}}=\Pi_{q_{i}}^{q'_{i}}a^{u}>(\max a^{u})^{q'_{i}-q_{i}}=e^{-l_{i}\log(\max a^{u})}
%\end{equation}
%If we show that $\frac{l^{u}_{q'_{i}}}{l^{u}_{q_{i}}}<\frac{\delta}{\epsilon}$, then $l_{i}\geq \log(\frac{\epsilon}{\delta})/\log(\max a^{u})$ and hence
%$$\frac{\nu^{u}(I^{u})}{\nu^{u}(I)}\leq e^{-\kappa_{2}\log(\frac{\epsilon}{\delta})/\log(\max a^{u})}=\left(\frac{\delta}{\epsilon}\right)^{\frac{\kappa_{2}}{\log(\max a^{u})}}$$.
%\begin{equation*}
%\frac{\nu^{u}(I^{u})}{\nu^{u}(I)}\leq \frac{\nu^{u}(\chi\left(C_{(-q+l),..., +\infty}\right))}{\nu^{u}(\chi(C_{-q,..., +\infty}))}\leq k^{2}_{\varphi}\exp\left(\sum_{k=-q-l}^{-q}\varphi\circ \sigma^{k}(\omega)\right)\leq k^{2}_{\varphi}l ||\varphi||_{\infty}
%\end{equation*}
\end{proof}
\begin{definition}
We fix $y\in \Sigma_{A}$. For $x\in A_{\epsilon}(y)$, we define the first return time to $A_{\epsilon}(y)\times\{0\}$ in $\Sigma_{A}\times\mathbb{Z}$, by
\begin{equation*}
w_{A_{\epsilon}(y)}(x)%&:=&\inf\{n\geq1, \tilde{T}^{n}(x,0)\in A_{\epsilon}(y)\times\{0\}\}\\
=\inf\{n\geq1, S_{n}\varphi(x)=0\mbox{ and }\sigma^{n}x\in A_{\epsilon}(y)\}.
\end{equation*}
We define in the same way $w_{A_{\epsilon}^{\pm}(y)}(x)=\min\{n\geq1 : \sigma^{n}(x)\in A_{\epsilon}^{\pm}(y)\}$ the first return time of a point $x$ into $A_{\epsilon}^{+}(y)$.
\end{definition}
\begin{remarque}
We will start by studying the asymptotic behavior of $w_{A_{\epsilon}(y)}(x)$ to deduce that of $\tau_{\epsilon}(x)$.
\end{remarque}
%Due to \eqref{inclusionshift}, we have:
%\begin{equation}\label{3returnsx}
%w_{A_{\epsilon}^{+}(y)}(x)\leq w_{A_{\epsilon}(y_{\epsilon})}(x)\leq w_{A_{\epsilon}^{-}(y)}(x)
%\end{equation}

\begin{proposition}\label{ineq1}
For every $y\in M$, every $\epsilon>0$, consider $n_\epsilon=n_\epsilon(t,y):=\left(\frac{t}{\nu(A_{\epsilon}^{\pm}(y))}\right)^{2}$ and $M_\epsilon$ and $Q_\epsilon$ as defined in \eqref{MepsQeps}. Let $D$ be a set satisfying either ($D\in \mathcal{D}_{\epsilon}$ and $D\subset A_{\epsilon}^{\pm}(y)$) or $D=A_{\epsilon}^{\pm}(y)$, we have:
\begin{itemize}
\item[a.]$1\geq \nu(w_{A_{\epsilon}^{\pm}(y)}>n_{\epsilon}|D)+\dfrac{1}{\sqrt{2\pi}\sigma_{\varphi}}\displaystyle\sum_{r=M_{\epsilon}}^{n_{\epsilon}}\frac{\nu(A_{\epsilon}^{\pm}(y)\cap \{w_{A_{\epsilon}^{\pm}(y)}>n_{\epsilon}-r\})}{\sqrt{r-Q_{\epsilon}}}-o(1)$.
\item[b.]$1\leq \nu(w_{A_{\epsilon}^{\pm}(y)}>n_{\epsilon}|D)+\dfrac{1}{\sqrt{2\pi}\sigma_{\varphi}}\displaystyle\sum_{r=M_{\epsilon}}^{n_{\epsilon}}\frac{\nu(A_{\epsilon}^{\pm}(y)\cap \{w_{A_{\epsilon}^{\pm}(y)}>n_{\epsilon}-r\})}{\sqrt{r-Q_{\epsilon}}}+ \nu(w_{A_{\epsilon}^{\pm}(y)}\leq M_{\epsilon}|D)+o(1)$.
\end{itemize}
\end{proposition}

\begin{proof}
We use the same method as Dvoretzky and Erd\"os in \cite{DvoretzkyErdos}, \cite{penesaussolbilliard} and \cite{penesaussol} and\cite{pzs}. We make a partition of the cylinder $D$ according to the last passage $r\leq n$ in the time interval $0,...,n$ of the orbit of $(x,0)$ by the map $\tilde{T}$ into $A_{\epsilon}^{\pm}(y)\times\{0\}$. This can be seen as follows:\\
For simplicity, we will denote by $A_{\epsilon}^{\pm}$ instead of $A_{\epsilon}^{\pm}(y)$,
\begin{equation}\label{dvoretzky}
\nu(D)=\sum_{r=0}^{n} \nu\left(x\in D;S_{r}\varphi(x)=0\mbox{ , } \sigma^{r}(x)\in A_{\epsilon}^{\pm};\forall l=r,...,n, S_{l}\varphi(x)\neq0 \mbox{ or } \sigma^{l}(x)\notin A_{\epsilon}^{\pm}\right).
\end{equation}
%knowing that $S_{r}\varphi(x)=0$, then $$\{\forall l=r,..,n, S_{l}\varphi(x)-S_{r}\varphi(x)\neq0 \mbox{ , } \sigma^{l}(x)\notin A_{\epsilon}^{+}\}=\{\forall k=1,..,n-r, S_{k+r}\varphi(x)-S_{r}\varphi(x)\neq0 \mbox{ , } \sigma^{k+r}(x)\notin A_{\epsilon}^{+}\},$$
Knowing that $S_{r}\varphi(x)=0$, then $$\{\forall l=r,..,n, S_{l}\varphi(x)-S_{r}\varphi(x)\neq0 \mbox{ or } \sigma^{l}(x)\notin A_{\epsilon}^{\pm}\}$$ is equal to $$\{\forall k=1,..,n-r, S_{k+r}\varphi(x)-S_{r}\varphi(x)\neq0 \mbox{ or } \sigma^{k+r}(x)\notin A_{\epsilon}^{\pm}\},$$
but $S_{k+r}\varphi(x)-S_{r}\varphi(x)=\sum_{l=0}^{k-1}\varphi\circ\sigma^{l+r}(x)=S_{k}\varphi(\sigma^{r}x)$, it follows that $\nu(D)$ is equal to
\begin{equation*}
\sum_{r=0}^{n} \nu\left(x\in D;S_{r}(x)=0,\sigma^{r}(x)\in A_{\epsilon}^{\pm}; \sigma^{-r}\{z\in\Sigma_{A}: \forall k=1,..,n-r, S_{k}\varphi(z)\neq0,\sigma^{k}(z)\notin A_{\epsilon}^{\pm}\}\right),
\end{equation*}
from which we get
\begin{eqnarray*}
\nu(D)&=&\sum_{r=0}^{n}\nu\left(D\cap\{S_{r}=0\}\cap \sigma^{-r}\left(A^{\pm}_{\epsilon}\cap\{w_{A^{\pm}_{\epsilon }}>n-r\}\right)\right)\\
&\geq& \nu\left(D\cap A^{\pm}_{\epsilon}\cap\{w_{A^{\pm}_{\epsilon}}>n\}\right)\\
&&+\sum_{r=2q+1}^{n}\nu\left(D\cap\{S_{r}=0\}\cap \sigma^{-r}\left(A_{\epsilon}^{\pm}\cap\{w_{A_{\epsilon}^{\pm}}(x)>n-r\}\right)\right).
\end{eqnarray*}
Recall that $D$ is a $(-q, q')-$cylinder with $q\leq Q_{\epsilon}$, and $A_{\epsilon}^{\pm}$ is a union of $(-Q_{\epsilon}, q'')-$cylinders for some $q''>0$. Let $E=A_{\epsilon}^{\pm}(y)\cap \{w_{A_{\epsilon}^{\pm}}>n\}$ the set of points of $A_{\epsilon}^{\pm}(y)$ which don't return to $A_{\epsilon}^{\pm}$ before $n$, that is $E=A_{\epsilon}^{\pm}\cap\bigcap_{s=1}^{n}\left(\{S_{s}\varphi\neq0\}\cup\sigma^{-s}(\Sigma_{A}\setminus A_{\epsilon}^{\pm})\right)$.\\
we know that $\varphi$ is constant on $1$-cylinders, then $\{S_{s}\varphi\neq0\}$ is  a union of $s$-cylinders, and $\sigma^{-s}(\Sigma_{A}\setminus A_{\epsilon}^{\pm})=\Sigma_{A}\setminus \sigma^{-s}(A_{\epsilon}^{\pm})$ is a union of $(-Q_{\epsilon}+s,q''+s)$-cylinders. Then we get that $E$ is a union of $(-Q_{\epsilon}, q'')$-cylinders, for some $q''>0$. Hence $E$ can be written as $\sigma^{Q_{\epsilon}}({\Pi_{+}}^{-1}B)$, where $B$ is a subset of $\Sigma_{A}^{+}$, and $\Pi_{+}$ is the canonical projection of $\Sigma_{A}$ on $\Sigma_{A}^{+}$.\\
Now, let $n=n_{\epsilon}$, with all these conditions, we apply the Local Limit Theorem in Proposition \ref{p3}, thus there is $C_{1}>0$ such that:
\begin{eqnarray*}
\nu(D)
&\geq&\nu\left(D\cap\{w_{A_{\epsilon}^{\pm}}(x)>n_{\epsilon}\}\right)+\sum_{r=M_{\epsilon}}^{n_{\epsilon}}\frac{\nu(D)\nu(A_{\epsilon}^{\pm}\cap\{w_{A_{\epsilon}^{\pm}}(x)>n_{\epsilon}-r\})}{\sqrt{2\pi}\sqrt{r-(Q_{\epsilon}+l)}\sigma_{\varphi}}\\
&-& C_{1}\sum_{r=M_{\epsilon}}^{n_{\epsilon}}\frac{\nu(A_{\epsilon}^{\pm}\cap\{w_{A_{\epsilon}^{\pm}}(x)>n_{\epsilon}-r\})Q_{\epsilon}\nu(D)}{r-2Q_{\epsilon}}.
\end{eqnarray*}
%%We notice that using Lemma \ref{errorterm}, such that $l+l'\leq\frac{\gamma}{\kappa_{1}-\gamma}(q+q')$, the error term is controlled, for some $C_{2}>0$, by
%$D$ is a $(q,q')-$cylinder, using Lemma \ref{ballincylpartition}, there exists $k>0$, there exists $x_{0}\in\Sigma_{A}$ such that $B(x_{0},\frac{\theta_{\epsilon}}{K})\subset D$. Furthermore, we know that there is $k_{0}>0$ such that $A_{\epsilon}^{+}(y)\subset B(x_{0},k_{0}\epsilon)$. Choose $m_{\epsilon}>0$ the smallest integer such that $2^{m_{\epsilon}}\theta_{\epsilon}\geq k_{0}K\epsilon$, which gives $2^{m_{\epsilon}}<\frac{2K k_{0}\epsilon}{\theta_{\epsilon}}$. Thus using the Federer property, there exists $C_{2}>0$ such that
%\begin{equation*}
%\nu(B(x_{0},k_{0}\epsilon))\leq\nu(B(x_{0},2^{m_{\epsilon}}\frac{\theta_{\epsilon}}{K}))\leq C_{2}^{m_{\epsilon}}\nu(B(x_{0},\frac{\theta_{\epsilon}}{K}))\leq C_{2}^{m_{\epsilon}}\nu(D)
%\end{equation*}
%Now notice that $C_{2}^{m_{\epsilon}}=2^{m_{\epsilon}\frac{\log C_{2}}{\log2}}<\left(\frac{2k_{0}k_{\theta}\epsilon}{\theta_{\epsilon}}\right)^{\frac{\log C_{2}}{\log2}}$. This in addition that $\theta_{\epsilon}=\frac{\epsilon}{\log\epsilon}$,
There is $C_{3}>0$ such that the error term is controlled by
\begin{equation}\label{errorterm>}
\sum_{r=M_{\epsilon}}^{n_{\epsilon}}\frac{\nu(A_{\epsilon}^{\pm})q\nu(D)}{r-2Q_{\epsilon}}\leq C_{3}Q_{\epsilon}\log{n_{\epsilon}}\nu(D)\nu(A_{\epsilon}^{\pm}),
%&\leq&C_{3}q^{2}e^{-\gamma(q+q')}\left(\frac{2k_{0}k_{\theta}\epsilon}{\theta_{\epsilon}}\right)^{\frac{\log C_{2}}{\log2}}\nu(D)=o(\nu(D)).
\end{equation}
%Since $\nu(A_{\epsilon}^{\pm})\log\left(\frac{n_{\epsilon}}{M_{\epsilon}}\right)\leq $
since $Q_{\epsilon}\nu(A_{\epsilon}^{\pm})\log{n_{\epsilon}}=o(1)$, hence we get
\begin{eqnarray*}
\nu(D)&\geq& \nu\left(D\cap\{w_{A^{\pm}_{\epsilon}}(x)>n_{\epsilon}\}\right) +\nu(D)\sum_{r=M_{\epsilon}}^{n_{\epsilon}}\frac{\nu(A_{\epsilon}^{\pm}\cap\{w_{A_{\epsilon}^{\pm}}(x)>n_{\epsilon}-r\}}{\sqrt{2\pi}\sqrt{r-Q_{\epsilon}}\sigma_{\varphi}}\\
&&-o(\nu(D)),
\end{eqnarray*}
and thus dividing by $\nu(D)$ yields (a) for $D\in\mathcal{D}_{\epsilon}$ such that $D\subset A_{\epsilon}^{\pm}(y)$. %Applying sum over $i$ on the terms of inequality (a),
%\begin{eqnarray*}
%\underset{i:D_{i}\subset A_{\epsilon}}{\sum}\nu(D_{i})&\geq& \underset{i:D_{i}\subset A_{\epsilon}}{\sum}\nu( D_{i}\cap\{w_{A_{\epsilon}^{+}}>n\})+\underset{i:D_{i}\subset A_{\epsilon}}{\sum}\nu(D_{i})\nu(A_{\epsilon}^{+}\cap \{w_{A_{\epsilon}^{+}}>n\})\sum_{0}^{n}\frac{1}{\sqrt{r-q-l}}\\
%&+& o\left(\underset{i:D_{i}\subset A_{\epsilon}}{\sum}\nu(D_{i})\right)
%\end{eqnarray*}
%from which we can get inequality (b), since $A_{\epsilon}^{+}(y_{\epsilon})$ is the union of cylinders $D_{i}$ that are contained in $A_{\epsilon}$.\\

%Set $n_{\epsilon}:=\left(\frac{t}{\nu_{\varphi}(A_{\epsilon}^{+}(y_{\epsilon})}\right)^{2}$,
%\end{proof}
%\begin{lemma}\label{ineq2}
%Having the same properties of $D_{i}$ and $A_{\epsilon}^{+}$, we have the other inequalities:
%\begin{itemize}
%\item[a.]$1\geq \nu(w_{A_{\epsilon}^{+}}>n|A_{\epsilon}^{+})+\sum_{0}^{n}\frac{\nu(A_{\epsilon}^{+}\cap \{w_{A_{\epsilon}^{+}}>n-r\})}{\sqrt{r-q-l}}-o(1)$.
%\item[b.]$1\leq \nu(w_{A_{\epsilon}^{+}}>n|A_{\epsilon}^{+})+\nu(A_{\epsilon}^{+}\cap \{w_{A_{\epsilon}^{+}}>n\})\sum_{0}^{n}\frac{1}{\sqrt{r-q-l}}+ o(1)$.
%\end{itemize}
%\end{lemma}
%\begin{proof}
We prove inequality (b) following the same strategy. For $n_{\epsilon}=\left(\frac{t}{\nu(A_{\epsilon}^{\pm}(y))}\right)^{2}$, using the same decomposition as in \eqref{dvoretzky}, we have
\begin{equation*}
\nu(D)=\sum_{r=0}^{n_{\epsilon}}\nu\left(D\cap\{S_{r}=0\}\cap \sigma^{-r}\left(A_{\epsilon}^{\pm}(y)\cap\{w_{A^{\pm}_{\epsilon}(y)}(x)>n_{\epsilon}-r\}\right)\right):=\sum_{r=0}^{n_{\epsilon}}G_{r,\epsilon}(x)
\end{equation*}
For $r=0$, the first term is equal to $\nu\left(D\cap\{w_{A^{\pm}_{\epsilon}(y)}(x)>n_{\epsilon}\}\right)$. We then compute the sum of terms between 1 and $M_{\epsilon}$
%Let $M_{\epsilon}=o\left(\frac{1}{\nu(A_{\epsilon}^{+})}\right)$, the sum of terms between 1 and $M_{\epsilon}$ is negligible, indeed $\nu$ is $\varphi$- mixing, then there exists $0<\theta<1$ such that we have

\begin{eqnarray*}
&\sum\limits_{r=1}^{M_{\epsilon}}&\nu\left(D\cap\{S_{r}=0\}\cap \sigma^{-r}\left(A_{\epsilon}^{\pm}(y)\cap\{w_{A_{\epsilon}^{\pm}(y)}>n_{\epsilon}-r\}\right)\right)\\
&=& \nu\left(\bigcup_{r=1}^{M_{\epsilon}} D\cap\{S_{r}=0\}\cap \sigma^{-r}\left(A_{\epsilon}^{\pm}(y)\cap\{w_{A_{\epsilon}^{\pm}(y)}>n_{\epsilon}-r\}\right)\right)\\
&\leq&\nu(D\cap\{w_{A_{\epsilon}^{\pm}(y)}\leq M_{\epsilon}\}).
%&=&\nu\left(\bigcup_{k=1}^{M_{\epsilon}}\nu(D_{i}\cap T^{-k}A_{\epsilon}^{+})\right)\\
%&\leq&\sum_{k=1}^{M_{\epsilon}}\nu(D_{i}\cap T^{-k}A_{\epsilon}^{+})\\
%&=&\sum_{k=1}^{q+q'+l'}\nu(D_{i}\cap T^{-k}A_{\epsilon}^{+})+\sum_{k=q+q'+l'}^{M_{\epsilon}}\nu(D_{i}\cap T^{-k}A_{\epsilon}^{+})\\
%&\leq&\nu(D_{i}\cap\{R_{A_{\epsilon}^{+}}\leq q+q'+l'\})+\sum_{k=q+q'+l'}^{M_{\epsilon}}\nu(D_{i})\nu(A_{\epsilon}^{+})\theta^{q+q'+l'}\\
%&\leq&\nu(R_{A_{\epsilon}^{+}}\leq q+q'+l'|D_{i})\nu(D_{i})+\nu(D_{i})\frac{1}{1-\theta}e^{-\kappa_{2}(q+q')}\\
%%&\leq&\nu(D_{i})(1-\nu(R_{A_{\epsilon}^{+}}>M_{\epsilon}|D_{i}))\\
%&\leq& \nu(R_{A_{\epsilon}^{+}}\leq q+q'+l'|D_{i})\nu(D_{i})+o(\nu(D_{i})).
\end{eqnarray*}
Then using the local limit theorem in Proposition \ref{p3}, there exists $C_{4}>0$ such that we have the following:
\begin{eqnarray*}
\sum^{n_{\epsilon}}_{r=M_{\epsilon}} G_{r,\epsilon}(x)&\leq& \frac{1}{\sqrt{2\pi}\sigma_{\varphi}}\sum^{n_{\epsilon}}_{r=M_{\epsilon}}\frac{\nu(D)\nu(A^{\pm}_{\epsilon}\cap\{w_{A_{\epsilon}^{\pm}}>n_{\epsilon}-r\})}{\sqrt{r-Q_{\epsilon}}}\\
&&+C_{4}\frac{\nu(A^{\pm}_{\epsilon}\cap\{w_{A_{\epsilon}^{\pm}}>n_{\epsilon}-r\})Q_{\epsilon}\nu(D)}{r-2Q_{\epsilon}}.
\end{eqnarray*}
Similarly as in \eqref{errorterm>}, the error term is controlled by $o(\nu(D))$. %$$\sum^{n_{\epsilon}}_{r=M_{\epsilon}+1}\frac{\nu(A^{+}_{\epsilon})q^{2}e^{-\gamma(q+q')}}{r-2q}=o(\nu(D)).$$
%%Now we consider the sum of the main term between $M_{\epsilon+1}$ and $m_{\epsilon}$,
%whenever $l\leq m_{\epsilon}$ we have $N_{\epsilon}-l\geq n_{\epsilon}$, therefore we get
%\begin{equation}
%\sum^{n_{\epsilon}}_{r=M_{\epsilon}+1}\frac{\nu(D_{i})\nu(A^{+}_{\epsilon}\cap\{w_{A_{\epsilon}^{+}}>n_{\epsilon}-r\})}{\sqrt{r-(q+l)}}\leq \nu(D_{i})\nu(A^{+}_{\epsilon}\cap\{w_{A_{\epsilon}^{+}}>n_{\epsilon}\})\sum^{n_{\epsilon}}_{r=M_{\epsilon}+1}\frac{1}{\sqrt{r-(q+l)}}.
%\end{equation}
Combining these assumption together, we obtain
\begin{eqnarray}
\nu(D)&\leq& o(\nu(D))+\nu(D\cap\{w_{A_{\epsilon}^{\pm}}>n_{\epsilon}\})+\nu(w_{A_{\epsilon}^{\pm}}\leq M_{\epsilon}|D)\nu(D)\nonumber\\
&+&\frac{1}{\sqrt{2\pi}\sigma_{\varphi}}\nu(D)\nu(A^{\pm}_{\epsilon}\cap\{w_{A_{\epsilon}^{\pm}}>n_{\epsilon}\})\left(\sum^{n_{\epsilon}}_{r=M_{\epsilon}}\frac{1}{\sqrt{r-Q_{\epsilon}}}\right)
\end{eqnarray}
Dividing by $\nu(D)$ yields the desired inequality.% Applying sum over $i$ on the terms of inequality (a),
%\begin{eqnarray*}
%\underset{i:D_{i}\subset A_{\epsilon}}{\sum}\nu(D_{i})&\leq& \underset{i:D_{i}\subset A_{\epsilon}}{\sum}\nu( D_{i}\cap\{w_{A_{\epsilon}^{+}}>n\})+\underset{i:D_{i}\subset A_{\epsilon}}{\sum}\nu(D_{i})\nu(A_{\epsilon}^{+}\cap \{w_{A_{\epsilon}^{+}}>n\})\sum_{0}^{n}\frac{1}{\sqrt{r-q-l}}\\
%&+& o(\underset{i:D_{i}\subset A_{\epsilon}}{\sum}\nu(D_{i}))
%\end{eqnarray*}
%from which we can get inequality (b), since $A_{\epsilon}^{+}(y_{\epsilon})$ is the union of cylinders $D_{i}$ that are contained in $A_{\epsilon}$.
\end{proof}
\begin{corollary}\label{Aeps=Di}
Under the hypothesis of Proposition \ref{ineq1}, where $n_{\epsilon}$, $M_{\epsilon}$ and $D$ are considered similarly, we have:
\begin{itemize}
\item[a.] $1+o(1)= \nu(w_{A_{\epsilon}^{\pm}(y)}>n_{\epsilon}|D)+\frac{1}{\sqrt{2\pi}\sigma_{\varphi}}\displaystyle\sum_{r=M_{\epsilon}}^{n_{\epsilon}}\frac{\nu(A_{\epsilon}^{\pm}(y)\cap \{w_{A_{\epsilon}^{\pm}(y)}>n_{\epsilon}-r\})}{\sqrt{r-Q_{\epsilon}}}+O\left(\nu(w_{A_{\epsilon}^{\pm}(y)}\leq M_{\epsilon}|D)\right)$
\item[b.]$1+o(1)= \nu(w_{A_{\epsilon}^{\pm}(y)}>n_{\epsilon}|A_{\epsilon}^{\pm}(y))+\frac{1}{\sqrt{2\pi}\sigma_{\varphi}}\displaystyle\sum_{M_{\epsilon}}^{n_{\epsilon}}\frac{\nu(A_{\epsilon}^{\pm}(y)\cap \{w_{A_{\epsilon}^{\pm}(y)}>n_{\epsilon}-r\})}{\sqrt{r-Q_{\epsilon}}}+O\left(\nu(w_{A_{\epsilon}^{\pm}(y)}\leq M_{\epsilon}|A_{\epsilon}^{\pm}(y))\right).$
\end{itemize}
\end{corollary}
%\begin{proof}
%The equality in (a) is directly obtained since $M_{\epsilon}>2(q+l)$. On the other hand, we have $A_{\epsilon}^{+}=\cup_{i\in I} D_{i}$ as defined in the beginning, then using (a) in \ref{ineq1} with $D=D_{i}$, we get:
%\begin{eqnarray*}
%\nu(w_{A_{\epsilon}^{+}}>n_{\epsilon}|A_{\epsilon}^{+})&=&\frac{w_{A_{\epsilon}^{+}}>n_{\epsilon}\cap A_{\epsilon}^{+}}{\nu(A_{\epsilon}^{+})}\\
%&=&\sum_{i}\frac{w_{A_{\epsilon}^{+}}>n_{\epsilon}\cap D_{i}}{\nu(A_{\epsilon}^{+})}\\
%&=&\sum_{i}\nu(w_{A_{\epsilon}^{+}}>n_{\epsilon}|D_{i})\frac{\nu(D_{i})}{\nu(A_{\epsilon}^{+})}\\
%&\leq&\sum_{i}\left[1-\sum_{r=2(q+l)+1}^{n_{\epsilon}}\frac{\nu(A_{\epsilon}^{+}\cap \{w_{A_{\epsilon}^{+}}>n_{\epsilon}-r\})}{\sqrt{r-q-l}}+o(1)\right]\frac{\nu(D_{i})}{\nu(A_{\epsilon}^{+})}\\
%&\leq& 1-\sum_{r=2(q+l)+1}^{n_{\epsilon}}\frac{\nu(A_{\epsilon}^{+}\cap \{w_{A_{\epsilon}^{+}}>n_{\epsilon}-r\})}{\sqrt{r-q-l}}+o(1),
%\end{eqnarray*}
%and similarly for the reverse inequality, and hence getting (b).
%
%\end{proof}
\begin{lemma}\label{remarkconclude}
From the previous Corollary, we have
\begin{eqnarray*}
\nu(w_{A_{\epsilon}^{\pm}(y)}>n_{\epsilon}|D)&=&\nu(w_{A_{\epsilon}^{\pm}(y)}>n_{\epsilon}|A_{\epsilon}^{\pm}(y))+O\left(\nu(w_{A_{\epsilon}^{\pm}(y)}\leq M_{\epsilon}|A_{\epsilon}^{\pm}(y))\right)\\
&&+O\left(\nu(w_{A_{\epsilon}^{\pm}(y)}\leq M_{\epsilon}|D)\right)+o(1).
\end{eqnarray*}
Thus we will study the convergence in distribution of $w_{A_{\epsilon}^{\pm}(y)}$ with respect to $\nu(.|A_{\epsilon}^{\pm}(y))$ and then we will deduce it with respect to $\nu(.|D)$. So, first we have the following proposition:
\end{lemma}
\begin{proposition}\label{tightness}
The family of distributions of $\left(\nu(A_{\epsilon}^{\pm}(y))^{2}w_{A_{\epsilon}^{\pm}(y)}\right)_{\epsilon>0}$ is tight with respect to the family of probability measures $\left(\nu(.|A_{\epsilon}^{\pm}(y))\right)_{\epsilon>0}$.
\end{proposition}
\begin{proof}
Let $t>0$, using inequality (a) of Proposition \ref{ineq1} with $n_{\epsilon}=\left(\frac{t}{\nu(A^{\pm}_{\epsilon}(y))}\right)^{2}$, we have
\begin{eqnarray*}
1&\geq&\nu(w_{A_{\epsilon}^{\pm}(y)}>n_{\epsilon}|A^{\pm}_{\epsilon}(y))+\frac{1}{\sqrt{2\pi}\sigma_{\varphi}}\sum_{M_{\epsilon}}^{n_{\epsilon}}\frac{\nu(A_{\epsilon}^{\pm}(y)\cap \{w_{A_{\epsilon}^{\pm}(y)}>n_{\epsilon}-r\})}{\sqrt{r-Q_{\epsilon}}}- o(1)\\
&\geq&\nu(w_{A_{\epsilon}^{\pm}(y)}>n_{\epsilon}|A^{\pm}_{\epsilon}(y))+\frac{1}{\sqrt{2\pi}\sigma_{\varphi}}\nu(A_{\epsilon}^{\pm}(y)\cap \{w_{A_{\epsilon}^{\pm}(y)}>n_{\epsilon}\})\sum_{M_{\epsilon}}^{n_{\epsilon}}\frac{1}{\sqrt{r-Q_{\epsilon}}}- o(1),
\end{eqnarray*}
from which it follows that
\begin{equation}
\forall t>0 \quad \nu\left(\ w_{A_{\epsilon}^{\pm}(y)}>\left(\frac{t}{\nu(A^{\pm}_{\epsilon}(y))}\right)^{2}\Big| A^{\pm}_{\epsilon}(y)\right)\leq\frac{1+o(1)}{1+\frac{1}{\sqrt{2\pi}\sigma_{\varphi}}\nu(A_{\epsilon}^{\pm}(y))\underset{M_{\epsilon}}{\overset{n_{\epsilon}}{\sum}}\frac{1}{\sqrt{r-Q_{\epsilon}}}},
\end{equation}
but $\underset{M_{\epsilon}}{\overset{n_{\epsilon}}{\sum}}\frac{1}{\sqrt{r-Q_{\epsilon}}}=2(\sqrt{n_{\epsilon}}-\sqrt{Q_{\epsilon}-M_{\epsilon}})\geq\sqrt{n_{\epsilon}}\geq\frac{t}{\nu(A_{\epsilon}^{\pm}(y))}$, then we get
\begin{equation}
\underset{\epsilon\rightarrow0}{\limsup}\ \nu\left(\nu(A^{\pm}_{\epsilon}(y))^{2}w_{A_{\epsilon}^{\pm}(y)}>t^{2}\mid A^{\pm}_{\epsilon}(y)\right)\leq\frac{1}{1+t},
\end{equation}
which proves the tightness.
\end{proof}
\begin{theoreme}[\cite{Rousseau}]\label{rousseau}
Define $\bar{\tau_{\epsilon}}=\inf\{t>1: g_{t}(y)\in B(y,\epsilon)\}$. For $\mu$-almost every $y\in M$
\begin{equation*}
\underset{\epsilon\rightarrow0}{\lim}\ \frac{\log \bar{\tau_{\epsilon}}(y)}{-\log\epsilon}=\bar{d}-1, 
\end{equation*}
 where $\bar{d}$ is the Hausdorff dimension of the measure $\mu$.
\end{theoreme}
Consider the following lemma, where we will define a family of sets $D_{\epsilon}$, and which we will apply later in our proofs with $D_{\epsilon}=A_{\epsilon}(y)$ and $D_{\epsilon}\in\mathcal{D}_{\epsilon}$.
\begin{lemma}\label{jerome}
For $\nu$-almost every $x\in\Sigma_{A}$, for all $k_{0}>0,m_{F}>0,c_{F}>0$, let $(D_{\epsilon})_{\epsilon}$ be a family of sets containing $x$ such that $D_{\epsilon}\subset B(x,k_{0}\epsilon)$ and $\nu(B(x, \text{diam}(D_{\epsilon})))\leq c_{F}^{m_{F}}\nu(D_{\epsilon})$. For $M_{\epsilon}$ defined as in \eqref{MepsQeps}, we have
\begin{equation*}
\underset{\epsilon\rightarrow0}{\lim}\ \nu(R_{B(x,k_{0}\epsilon)}\leq M_{\epsilon}|D_{\epsilon})=0,
\end{equation*}
where $R_{B(x, k_{0}\epsilon)}=\inf\{n\geq 1: \sigma^{n}(x)\in B(x,k_{0}\epsilon)\}.$
\end{lemma}
\begin{proof}
Let $\alpha>0, k_{0}>0, m_{F}>0, c_{F}>0$. Choose some $a\in(0,\alpha)$ and set for some $\epsilon_{0}>0$,
\begin{equation*}
F_{a}=\{x\in\Sigma_{A}:\forall\epsilon\leq\epsilon_{0},M_{\epsilon}^{-1}R_{B(x,2k_{0}\epsilon)}(x)>a\}
\end{equation*}
We will show that $\nu(F_{a})\rightarrow1$ as $\epsilon_{0}\rightarrow0$. To this end, let us prove that $\nu-$a.e. $x\in\Sigma_{A}$,
\begin{equation}\label{jeromeinfinity}
\underset{\epsilon\rightarrow0}{\lim}\ M_{\epsilon}^{-1}R_{B(x,2k_{0}\epsilon)}(x)=+\infty.
\end{equation}
An intermediate step to do that is to prove that
\begin{equation}
\underset{\epsilon\rightarrow0}{\lim}\ \frac{\log R_{B(x,2k_{0}\epsilon)}}{-\log\epsilon}= dim_{H}\mu-1, \quad \nu- a.e.
\end{equation}
Let $r_{\min}>0$ and $r_{\max}>0$ the lower and upper bounds of the height function $r$. Using \eqref{xsimy}, for $\mu_{\Delta}-$almost every $y=(x,s)\in \Delta$, we have
\begin{equation*}
\bar{\tau_{\epsilon}}(y)\frac{1}{r_{\max}}\leq R_{B(x,2k_{0}\epsilon)}(x)\leq\bar{\tau_{\epsilon}}(y)\frac{1}{r_{\min}},
\end{equation*}
which gives
\begin{equation*}
\frac{\log r_{\max}}{\log\epsilon}+\frac{\log\bar{\tau_{\epsilon}}(y)}{-\log\epsilon}\leq\frac{\log R_{B(x,2k_{0}\epsilon)}(x)}{-\log\epsilon}\leq \frac{\log r_{\min}}{\log\epsilon}+\frac{\log\bar{\tau_{\epsilon}}(y)}{-\log\epsilon},
\end{equation*}
using Theorem \ref{rousseau}, it follows that $\underset{\epsilon\rightarrow0}{\lim}\ \frac{\log R_{B(x,2k_{0}\epsilon)}}{-\log\epsilon}= dim_{H}\mu-1$, $\nu-$ almost everywhere.\\
Now we go back to show \eqref{jeromeinfinity}. Let $0<\delta_{0}<dim_{H}\mu-1$, we have the following:
\begin{equation}\label{M_eps.eps^-dtendsinfnty}
\log(M_{\epsilon}^{-1}R_{B(x,2k_{0}\epsilon)}(x))=\log(M_{\epsilon}^{-1}\epsilon^{-\delta_{0}})+\log(\epsilon^{\delta_{0}}R_{B(x,2k_{0}\epsilon)}(x))
\end{equation}
Note that the first term of \eqref{M_eps.eps^-dtendsinfnty} goes to $+\infty$, as $\epsilon$ goes to 0. This is true using the definition of $M_{\epsilon}$, there exists $\rho_{0}>0$ such that $M_{\epsilon}<2\frac{|\log\left(\frac{\theta_{\epsilon}}{2Cc_{g}c_{l}}\right)|}{|\log \rho_{u}|}+1\leq \rho_{0}|\log\epsilon|$, which is negligible with respect to $\epsilon^{-\delta_{0}}$.\\
Now, considering the second term, we observe that
\begin{equation*}
\frac{\log(\epsilon^{\delta_{0}}R_{B(x,2k_{0}\epsilon)}(x))}{-\log\epsilon}=-\delta_{0}+\frac{\log R_{B(x,2k_{0}\epsilon)}(x)}{-\log\epsilon},
\end{equation*}
from which we get $\log(\epsilon^{\delta_{0}}R_{B(x,2k_{0}\epsilon)}(x))\underset{\epsilon\rightarrow0}{\rightarrow} (-\delta_{0}+dim_{H}\mu-1)(+\infty)$. And therefore \eqref{jeromeinfinity} is proved.\\
Taking $\epsilon_{0}\rightarrow0$, by Egoroff's Theorem, $M_{\epsilon}^{-1}R_{B(x,2k_{0}\epsilon)}$ converges uniformly to $+\infty$ on $F_{a}$, and $\nu(F_{a}^{C})\rightarrow0$.\\
Let $\epsilon<\epsilon_{0}$. if $x'\in B(x,k_{0}\epsilon)$ and $R_{B(x,k_{0}\epsilon)}(x')\leq M_{\epsilon}$, we have $R_{B(x,2k_{0}\epsilon)}(x')\leq M_{\epsilon}$ as well, hence $x'\in F_{a}^{C}$. Therefore
%On the other hand, there exists $\epsilon_{1}>0$ such that, for any $\epsilon<\epsilon_{1}$, we have the following inclusion
\begin{equation*}
D_{\epsilon}\cap\{x'\in\Sigma_{A},R_{B(x,k_{0}\epsilon)}(x')\leq M_{\epsilon}\}\subset D_{\epsilon}\cap F^{C}_{a}.
\end{equation*}
Then, this means that for any density point $x$ of the set $F_{a}$ relative to the Lebesgue basis given by $\left(B(.,\delta)\right)_{\delta}$, $\frac{\nu(F_{a}\cap B(x,\delta))}{\nu(B(x,\delta))}\rightarrow 1$, as $\delta\rightarrow0$, and this is true for almost every point $x$. Therefore, we get
\begin{eqnarray*}
\nu(R_{B(x,k_{0}\epsilon)}\leq M_{\epsilon}|D_{\epsilon})&\leq&\nu(F_{a}^{C}|D_{\epsilon})\\
&=&\frac{\nu(F^{C}_{a}\cap D_{\epsilon})}{\nu(D_{\epsilon})}\\
&\leq&\frac{\nu(F^{C}_{a}\cap B(x,\text{diam}(D_{\epsilon})))}{\nu(D_{\epsilon})}\\
&\leq& \nu(F^{C}_{a}| B(x,\text{diam}(D_{\epsilon})))\frac{\nu(B(x,\text{diam}(D_{\epsilon})))}{\nu(D_{\epsilon})},
\end{eqnarray*}
which converges to 0 as $\epsilon\rightarrow0$, since $\underset{\epsilon\rightarrow0}{\lim}\ \nu(F_{a}^{C}|B(x,diam (D_{\epsilon})))=0$ and $\dfrac{\nu(B(x,\text{diam}(D_{\epsilon}))}{\nu(D_{\epsilon})}$ is bounded above by $c_{F}^{m_{F}}$.
\end{proof}
\begin{proposition}\label{convergence+-}
For almost every $y\in M$,
\begin{equation*}
\underset{\epsilon\rightarrow0}{\lim}\ \nu\left((\nu(A^{+}_{\epsilon}(y))^{2}w_{A^{+}_{\epsilon}(y)}>t|A_{\epsilon}^{+}(y)\right)=\mathbb{P}\left(\sigma_{\varphi}^{2}\frac{\mathcal{E}^{2}}{\mathcal{N}^{2}}>t\right), \quad \forall t>0,
\end{equation*}
and 
\begin{equation*}
\underset{\epsilon\rightarrow0}{\lim}\ \nu\left((\nu(A^{-}_{\epsilon}(y))^{2}w_{A^{-}_{\epsilon}(y)}>t|A_{\epsilon}^{-}(y)\right)=\mathbb{P}\left(\sigma^{2}_{\varphi}\frac{\mathcal{E}^{2}}{\mathcal{N}^{2}}>t\right), \quad \forall t>0,
\end{equation*}
where $\mathcal{E}$ and $\mathcal{N}$ are independent random variables, $\mathcal{E}$ follows the exponential distribution of mean 1 and $\mathcal{N}$ having the standard normal distribution.
\end{proposition}
\begin{proof}
After proving the tightness in Proposition \ref{tightness}, then it will be enough to prove that the advertised limit is the only possible accumulation point of the distribution.
Let $(\epsilon_{p})_{p\geq1}$ be a positive sequence with $\underset{p\rightarrow\infty}{\lim}\ \epsilon_{p}=0$ and such that the conditional distributions of the $(\nu(A_{\epsilon_{p}}^{+})\sqrt{w_{A_{\epsilon_{p}}^{+}}}|A^{+}_{\epsilon_{p}})_{p\geq0}$ converges to the law of some random variable $\mathcal{X}$. Then using Lemma 3.5 in \cite{yassine}, it follows from Lemma \ref{jerome} that $\mathcal{X}$ satisfies the integral equation:
\begin{equation}\label{integraleqn}
1=\mathbb{P}(\mathcal{X}>t)+ t\frac{1}{\sqrt{2\pi}\sigma_{\varphi}}\int_{0}^{1}\frac{\mathbb{P}(\mathcal{X}>t\sqrt{1-u})}{\sqrt{u}}du \quad \forall t>0.
\end{equation}
Then the conclusion follows as in Theorem 3.2 in \cite{yassine}, the distribution of $\mathcal{X}$ coincides with that of $\sigma_{\varphi}\dfrac{\mathcal{E}}{\mid\mathcal{N}\mid}$.
%with $c_{\beta}:=\left(\beta\Gamma\left(\frac{1}{2}\right)\right)^{-1}=\frac{\sqrt{2\pi}\sigma_{\varphi}}{\Gamma(1/2)}=\sqrt{2}\sigma_{\varphi}$.
\end{proof}
\begin{proposition}\label{propthm}
    The family $(\nu(A_{\epsilon}(.))^{2}\tau_{\epsilon})_{\epsilon}$ converges in distribution to $\left(\sigma^{2}_{\varphi}\int_{X}R d\nu_{X}\right)\dfrac{\mathcal{E}^{2}}{\mathcal{N}^{2}}$ with respect to any probability measure $\mathbb{P}$ absolutely continuous with respect to $\mu$.
\end{proposition}
\begin{proof}
By density in $L^1(\mu)$, it is enough to prove
the result for $\mathbb P$ with density $h$ such that $\sum_{k\in\mathbb{Z}} h\circ I^{k}$ is uniformly continuous on $\Lambda$.
Let $\tilde{\mu}=\nu\otimes Leb\otimes\sum_{k\in\mathbb{Z}}\delta_{k}$ be the measure on $\tilde{M}$. Set $\bar{\mu}:=\nu\otimes Leb\otimes\delta_{0}$ the measure restricted on the 0-cell and let $\mathbb{P}$ be a probability measure absolutely continuous with respect to $\mu$ with a density function $h$. Set $H(.):=\sum_{k\in\mathbb{Z}}h(.,k)$, by $\mathbb{Z}-$periodicity, the distribution of $\tau_{\epsilon}$ under $\mathbb{P}$ is the same as under the probability measure absolutely continuous with respect to $\bar{\mu}$ with density $H$.
%\begin{equation}
%\nu(A_{\epsilon}^{-}(y_{\epsilon,i,j}))^{2}w_{A_{\epsilon}^{+}}(x)\leq\nu(A_{\epsilon}(y))^{2}w_{A_{\epsilon}(y)}(x)\leq\nu(A_{\epsilon}^{+}(y_{\epsilon,i,j}))^{2}w_{A^{-}_{\epsilon}}(x),
%\end{equation}
%\begin{proof}[Proof of Theorem]
We have $y=(x,s)\in\bigcup_{P\in\mathcal{P}_{\epsilon}}P$, and $\exists$ $i,j$ such that $y\in\mathcal{P}_{\epsilon,i,j}$. Let $y_{\epsilon,i,j}\in\mathcal{P}_{\epsilon,i,j}$, then:
\begin{equation*}
H(y_{\epsilon,i,j})-w(H,\theta_{\epsilon})\leq H(y)\leq H(y_{\epsilon,i,j})+\omega(H,\theta_{\epsilon}), \quad y\in\mathcal{P}_{\epsilon,i,j}
\end{equation*}
then we have,
\begin{equation}\label{minH}
\underset{y\in\mathcal{P}_{\epsilon,i,j}}{\min}H(y)\geq H(y_{\epsilon,i,j})-\omega(H,\theta_{\epsilon}),
\end{equation}
and
\begin{equation}\label{maxH}
\underset{y\in\mathcal{P}_{\epsilon,i,j}}{\max}H(y)\leq H(y_{\epsilon,i,j})-\omega(H,\theta_{\epsilon}),
\end{equation}
Now we are ready to proof our main result. For simplicity we will denote by $A^{\pm}_{\epsilon,i,j}:=A^{\pm}_{\epsilon}(y_{\epsilon,i,j})$.
%\begin{proof}[Proof of Theorem \ref{theorem2}]
Due to \eqref{xsimy} and to the Stutzky's Lemma, it's enough to prove that $\left(\nu(A_{\epsilon})(.))^{2}w_{A_{\epsilon}(.))}\circ P_{D}(.)\right)_{\epsilon}$ converges in distribution with respect to $(H\bar{\mu})$ to $\sigma_{\varphi}^{2}\dfrac{\mathcal{E}^{2}}{\mathcal{N}^{2}}$.\\
For all $\epsilon>0$, $\forall t>0$ we set
\begin{eqnarray}\label{pricipalformula}
E_{\epsilon,t}&:=&(H\bar{\mu})\left(y=(x,s,0):\nu(A_{\epsilon}(y))^{2}w_{\epsilon}(x)>t\right)\nonumber\\
&=&\sum_{i,j}\int_{\mathcal{P}_{\epsilon,i,j}} \mathds{1}_{\left\{\nu(A_{\epsilon,i,j}(y))^{2}w_{A_{\epsilon,i,j}(y)}(x)>t\right\}}H(x,s)d\nu(x)ds.
\end{eqnarray}
Due to the inclusions in \eqref{inclusionshift}, then for $y=(x,s)\in \mathcal{P}_{\epsilon,i,j}$,
\begin{equation}\label{A-w+<Aw<A+w-}
\nu(A_{\epsilon,i,j}^{-})^{2}w_{A_{\epsilon,i,j}^{+}}(x)\leq\nu(A_{\epsilon}(x,s))^{2}w_{A_{\epsilon}(x,s)}(x)\leq\nu(A_{\epsilon,i,j}^{+})^{2}w_{A^{-}_{\epsilon,i,j}}(x),
\end{equation}
%we also have from \eqref{xsimy} that there exists $\zeta(\epsilon)>0$ such that
%\begin{equation}\label{zetaeps}
%(1-\zeta(\epsilon))\tau_{\epsilon}\leq w_{A_{\epsilon,i,j}}\int_{X} R d\nu\leq(1+\zeta(\epsilon))%\tau_{\epsilon}
%\end{equation}
therefore, using the inequality \eqref{maxH}, we get:
\begin{eqnarray*}
%&&\sum_{i,j}\int_{\mathcal{P}_{\epsilon,i,j}} 1_{\{y:\nu(A_{\epsilon}(y))^{2}\tau_{\epsilon}(y)>t\}}H(x,s)d\nu(x)dx\\
E_{\epsilon,t}&\leq&\sum_{i,j}\int_{\mathcal{P}_{\epsilon,i,j}} \mathds{1}_{\left\{x:\nu(A^{+}_{\epsilon,i,j})^{2}w_{A^{-}_{\epsilon,i,j}}(x)>t\right\}}\underset{y\in\mathcal{P}_{\epsilon,i,j}}{\max}{H(y)}d\nu(x)ds\\
&\leq&\sum_{i,j}\int_{\mathcal{P}_{\epsilon,i,j}}\mathds{1}_{\left\{x:\nu(A_{\epsilon,i,j}^{+})^{2}w_{A^{-}_{\epsilon,i,j}}(x)>t\right\}}H(y_{\epsilon,i,j})d\nu(x)ds +\sum_{i,j}\int_{\mathcal{P}_{\epsilon,i,j}}\omega(H,\theta_{\epsilon})d\nu ds\\
&\leq&\sum_{i,j}H(y_{\epsilon,i,j})\theta_{\epsilon}\nu\left(\left\{x\in D_{\epsilon,i}:\nu(A^{+}_{\epsilon,i,j})^{2}w_{A^{-}_{\epsilon,i,j}}(x)>t\right\}\right) +\sum_{i,j}\omega(H,\theta_{\epsilon})\nu(D_{\epsilon,i})\theta_{\epsilon}\\
&\leq&\sum_{i,j}H(y_{\epsilon,i,j})\nu(D_{\epsilon,i})\theta_{\epsilon}\nu\left(\nu(A_{\epsilon,i,j}^{+})^{2}w_{A^{-}_{\epsilon,i,j}}(x)>t| D_{\epsilon,i}\right)+\omega(H,\theta_{\epsilon}).
\end{eqnarray*}
For all $y\in \mathcal{P}_{\epsilon,i,j}$, we have $A_{\epsilon}^{-}(y)=A_{\epsilon,i,j}^{-}$, $A_{\epsilon}^{+}(y)=A_{\epsilon,i,j}^{+}$, and $D_{\epsilon,i}=D_{\epsilon}(x)$. Then $\forall t>0$
\begin{eqnarray}\label{Eepsilont}
E_{\epsilon,t}&\leq&\sum_{i,j}H(y_{\epsilon,i,j})\int_{P_{i,j}}\nu\left(\nu(A_{\epsilon}^{+}(y))^{2}w_{A^{-}_{\epsilon}(y)}>t| D_{\epsilon}(x)\right)d\bar{\mu}(y)+\omega(H,\theta_{\epsilon})\nonumber\\
&\leq&\int_{M}H(y)\nu\left(\nu(A_{\epsilon}^{+}(y))^{2}w_{A^{-}_{\epsilon}(y)}>t| D_{\epsilon}(x)\right)d\bar{\mu}(y)+2\omega(H,\theta_{\epsilon}).
\end{eqnarray}
Now we have $\nu(A^{-}_{\epsilon}(y))$ and $\nu(A_{\epsilon}^{+}(y))$ are equivalent (Lemma \ref{equivalent}), then there exists $\alpha(\epsilon)^{(y)}>0$ such that $\alpha(\epsilon)^{(y)}\rightarrow0$ as $\epsilon\rightarrow0$ and
\begin{equation}\label{use+eq-}
\nu(A_{\epsilon}^{-}(y))(1-\alpha(\epsilon)^{(y)})\leq\nu(A_{\epsilon}^{+}(y))\leq\nu(A_{\epsilon}^{-}(y))(1+\alpha(\epsilon)^{(y)}),
\end{equation}
then using  Lemma \ref{remarkconclude}, we get
\begin{align*}
\nu\left(\nu(A_{\epsilon}^{+}(y)^{2}w_{A^{-}_{\epsilon}(y)}>t| D_{\epsilon}(x)\right)
&\leq\nu\left((1+\alpha(\epsilon)^{(y)})^{2}\nu(A_{\epsilon}^{-}(y))^{2}w_{A^{-}_{\epsilon}(y)}>t| D_{\epsilon}(x)\right)\\
&\leq\nu\left((1+\alpha(\epsilon)^{(y)})^{2}\nu(A^{-}_{\epsilon}(y))^{2}w_{A^{-}_{\epsilon}(y)}>t| A^{-}_{\epsilon}(y)\right)\\
&\quad +\ \nu(w_{A_{\epsilon}^{-}(y)}\leq M_{\epsilon}|A_{\epsilon}^{-}(y))-\nu(w_{A_{\epsilon}^{-}(y)}\leq M_{\epsilon}| D_{\epsilon}(x))+o(1).
\end{align*}

%Going back to $E_{\epsilon,t}$, for all $y\in\mathcal{P}_{\epsilon,i,j}$, we $\forall t>0$, we get:
%\begin{eqnarray*}
%E_{\epsilon,t}(y)&\leq&\sum_{i,j}\nu(D_{\epsilon,i})\nu_{\epsilon}\nu\left((1+\alpha(\epsilon)^{(y)})^{2}\nu(A^{-}_{\epsilon,i,j})^{2}w_{A^{-}_{\epsilon,i,j}}>\frac{t(1-\zeta(\epsilon))}{\int R d\nu}| A^{-}_{\epsilon,i,j}\right)\\
%&+&\sum_{i,j}\nu(D_{\epsilon,i})\nu_{\epsilon}\left(\nu(R_{A_{\epsilon}^{-}}\leq M_{\epsilon}|A_{\epsilon}^{-})-\nu(R_{A_{\epsilon}^{-}}\leq M^{-}_{\epsilon}| D_{\epsilon,i})+o(1)\right)+w(H,\nu_{\epsilon})
%\end{eqnarray*}
    %Set $Q^{-}_{\epsilon}=\frac{(1+\alpha(\epsilon)^{(y)})^{2}}{1-\zeta(\epsilon)}$, we have $\underset{\epsilon\rightarrow0}{\lim}\ Q^{-}_{\epsilon}=1$, then using Proposition \ref{convergence+-} and due to Slutsky's Lemma, we get the following convergence in distribution:
  %\begin{equation*}
  %\underset{\epsilon\rightarrow0}{\lim}\ \nu\left((1+\alpha(\epsilon)^{(y)})^{2}\nu(A^{-}_{\epsilon}(y))^{2}w_{A^{-}_{\epsilon}(y)}>\frac{t(1-\zeta(\epsilon))}{\int R d\nu}| A^{-}_{\epsilon}(y)\right)=\mathbb{P}\left(\sigma^{2}_{\varphi}\frac{\mathcal{E}^{2}}{\mathcal{N}^{2}}>t\right), \quad \forall t>0
   %\end{equation*}
Moreover, due to Lemma \ref{jerome}, applied with $A_{\epsilon}^{-}(y)$ and $D_{\epsilon}(x)$, we have $\underset{\epsilon\rightarrow0}{\lim}\ \nu(w_{A_{\epsilon}^{-}(y)}\leq M_{\epsilon}|A_{\epsilon}^{-}(y))=0$, and $\underset{\epsilon\rightarrow0}{\lim}\ \nu(w_{A_{\epsilon}^{-}(y)}\leq M_{\epsilon}| D_{\epsilon}(x))=0$. Moreover, since $H$ is uniformly continuous then $\underset{\epsilon\rightarrow0}{\lim}\ \omega(H,\nu_{\epsilon})=0$. Thus using the Lebesgue's dominated convergence theorem in \eqref{Eepsilont}, we end up having
\begin{equation}\label{convergence<}
\underset{\epsilon\rightarrow0}{\limsup}\ E_{\epsilon,t}\leq \mathbb{P}\left(\left(\sigma^{2}_{\varphi}\int R d\nu_{X}\right)\frac{\mathcal{E}^{2}}{\mathcal{N}^{2}}>t\right), \quad \forall t>0.
\end{equation}
We do analogously the second inequality, thus using \eqref{minH} and the left hand side inequality in \eqref{A-w+<Aw<A+w-}, we get
\begin{eqnarray*}
E_{\epsilon,t}(y)&\geq&\sum_{i,j}\int_{\mathcal{P}_{\epsilon,i,j}} \mathds{1}_{\left\{x:\nu(A^{-}_{\epsilon,i,j})^{2}w_{A^{+}_{\epsilon,i,j}}(x)>t\right\}}\underset{y\in\mathcal{P}_{\epsilon,i,j}}{\min}{H(y)}d\nu(x)ds\\
%&\geq&\sum_{i,j}\int_{\mathcal{P}_{\epsilon,i,j}}1_{\{x:\nu(A_{\epsilon}^{-}(y))^{2}w_{A^{+}_{\epsilon}(y)}(x)>\frac{t(1-\zeta(\epsilon))}{\int R d\nu}\}}H(y_{\epsilon,i,j})d\nu(x)ds +\sum_{i,j}\int_{\mathcal{P}_{\epsilon,i,j}}w(H,\nu_{\epsilon})d\nu ds\\
%&\geq&\sum_{i,j}H(y_{\epsilon,i,j})\nu_{\epsilon}\nu\left(\{x\in D_{\epsilon,i}:\nu(A^{+}_{\epsilon}(y))^{2}w_{A^{-}_{\epsilon}(y)}(x)>\frac{t(1-\zeta(\epsilon))}{\int R d\nu}\}\right)\\
%&& \ +\sum_{i,j}w(H,\nu_{\epsilon})\nu(D_{\epsilon,i})\nu_{\epsilon}\\
&\geq&\sum_{i,j}H(y_{\epsilon,i,j})\nu(D_{\epsilon,i})\theta_{\epsilon}\nu\left(\nu(A_{\epsilon,i,j}^{-})^{2}w_{A^{+}_{\epsilon,i,j}}(x)>t| D_{\epsilon,i}\right)+\omega(H,\theta_{\epsilon}).
\end{eqnarray*}
Similarly, due to Lemma \ref{remarkconclude}, there exists $\alpha(\epsilon)^{(y)}>0$ such that
\begin{align*}
\nu\left(\nu(A_{\epsilon}^{-}(y))^{2}w_{A^{+}_{\epsilon}}>t| D_{\epsilon}(x)\right)
%&\leq&\nu\left((1+\alpha(\epsilon)^{(y)})^{2}\nu(A_{\epsilon}^{-}(y))^{2}w_{A^{-}_{\epsilon}(y)}(x)>\frac{t(1-\zeta(\epsilon))}{\int R d\nu}| D_{\epsilon,i}\right)\\
&\geq\nu\left(\frac{\nu(A^{+}_{\epsilon}(y))^{2}}{(1+\alpha(\epsilon)^{(y)})^{2}}w_{A^{+}_{\epsilon}}>t| A^{+}_{\epsilon}(y)\right)\\
&+\ \nu\left(w_{A_{\epsilon}^{+}(y)}\leq M_{\epsilon}|A_{\epsilon}^{+}(y))-\nu(w_{A_{\epsilon}^{+}(y)}\leq M_{\epsilon}| D_{\epsilon}(x)\right)+o(1).
\end{align*}
Then using Proposition \ref{convergence+-}, we get:
\begin{equation*}
\underset{\epsilon\rightarrow0}{\lim}\ \nu\left(\frac{\nu(A^{+}_{\epsilon})^{2}}{(1+\alpha(\epsilon)^{(y)})^{2}}w_{A^{+}_{\epsilon}(y)}(x)>t| A^{+}_{\epsilon}(y)\right)=\mathbb{P}\left(\frac{\sigma^{2}_{\varphi}\mathcal{E}^{2}}{\mathcal{N}^{2}}>t\right), \quad \forall t>0
\end{equation*}
By using the same reasoning as demonstrated in the proof of \eqref{convergence<}, we can also establish:
\begin{equation*}
\underset{\epsilon\rightarrow0}{\liminf}\ E_{\epsilon,t}\geq \mathbb{P}\left(\sigma^{2}_{\varphi}\frac{\mathcal{E}^{2}}{\mathcal{N}^{2}}>t\right), \quad \forall t>0,
\end{equation*}
combining these results, we can deduce the convergence in law of $\nu(A_{\epsilon})^{2}w_{A_{\epsilon}(.)}\circ P_{D}:\underset{\epsilon\rightarrow0}{\lim}\ E_{\epsilon,t}= \mathbb{P}\left(\sigma^{2}_{\varphi}\frac{\mathcal{E}^{2}}{\mathcal{N}^{2}}>t\right)$. Therefore by applying Slutsky's Lemma, we obtain
\begin{equation*}
\underset{\epsilon\rightarrow0}{\lim} \ \mathbb{P}(\nu(A_{\epsilon}(.))^{2}\tau_{\epsilon}(.)>t)=\mathbb{P}\left(\left(\sigma^{2}_{\varphi}\int R d\nu_{X}\right)\frac{\mathcal{E}^{2}}{\mathcal{N}^{2}}>t\right), \quad \forall t>0.
\end{equation*}
%The conclusion follows from Lemma \ref{findepreuve} below.
\end{proof}

Theorem \ref{theorem2} follows from Proposition \ref{propthm} combined with the fact that $\sigma_{flow}^{2}=\dfrac{\sigma_{\varphi}^{2}}{\int_{X}R d\nu_{X}}$ and the following lemma.

\begin{lemma}\label{findepreuve}
For $\tilde{\nu}$ almost every $\tilde{y}\in\tilde{M}$, if $\tilde{D}_{0}$ is small enough, the measure $\tilde{\nu}_{0}$ given by \eqref{nu_0} is well defined and
\begin{equation*}
\frac{\nu(A_{\epsilon}(y))}{\int_{X} R d\nu_{X}} \mbox{ is equivalent to } \tilde{\nu}^{y}_{0}(B(\tilde{y},\epsilon)),
\end{equation*}
with $y:=\Gamma(\tilde{y})\in M$.
\end{lemma}
\begin{proof}
Let $\tilde{y}\in \tilde{M}$. Note that since we are interested in the asymptotic behavior of  $\tilde\nu_0(B(\tilde y,\epsilon))$ when $\epsilon$ tends to $0$, we can replace $\tilde D_0(\tilde y)$ by a smaller disk if we want.\\
If the disk $\tilde D_0(\tilde y)$ and $\epsilon$ are small enough, the image measure of $\tilde\mu(\underset{-\epsilon<s<\epsilon}{\cup}\tilde g_s(\tilde D_0(\tilde y) \cap \cdot))$ by the projection of  $\tilde M$ to $M$ is equal to the measure   $\mu(\underset{-\epsilon<s<\epsilon}{\cup}g_s(D_0(y) \cap \cdot)$, where $D_0=D_0(y)$ is the disk around $y$ obtained by the projection of the disk $\tilde D_0=\tilde D_0(\tilde y)$.\\
So now we have to show that, for $\mu$-almost every $y$ in $M$, the measure $\tilde{\nu}_0^{y}$ defined on $D_0(y)$ by
\begin{equation*}
\tilde{\nu}_0^{y}(A)=\underset{\epsilon\rightarrow0}{\lim}\frac{1}{2\epsilon}\mu\left(\underset{-\epsilon<s<\epsilon}{\bigcup}g_{s}(D_{0}(y)\cap A)\right), \quad \forall A\subset D_{0}(y).
\end{equation*}
is well defined and that $\tilde{\nu}_0^{y}(B(y,\epsilon))$ is equivalent, when $\epsilon$ tends to 0, to  $\nu(A_\epsilon(y))/\int_X  R d\nu_{X}$.\\
We assume $y\in M\setminus X$ and we assume that $y$ is represented by a couple $(x,s)\in X_R$ with $x\not\in\partial X$. Let us call $D$ the connected component of $X$ containing the point $x$. For $\epsilon$ small enough, $D_0(y)$ small enough, $\underset{-\epsilon<s<\epsilon}{\bigcup}g_s(D_0 (y))$ is contained in $\{(x',s')\in X_R\ :\ x\in D\}$.\\

In $X_R$, the set  $D_0 (y)$ is represented by a curve  $\{(x',h(x')) : x'\in P_{D}(D_0)\}$ with $h$ of class $C^1$ of uniformly bounded derivative. We have, for every measurable set $A\subset D_0(y)$:
\begin{equation*}
\mu\left(\underset{-\epsilon<s<\epsilon}{\bigcup}g_s(D_0(y) \cap A)\right) =  \frac{2\epsilon \nu(\chi^{-1}( P_{D}(D_0(y)\cap A)))}{\int_X R\ d\nu}
\end{equation*}
This ensures the existence of $\nu_0$ and the fact that:
\begin{equation*}
\nu_{0}^{y}(.)=\frac{\nu(\chi^{-1}(P_{D}(.\cap D_{0}(y))))}{\int_{X} R d\nu_{X}}.
\end{equation*}

We have to prove that $\nu(A_{\epsilon}(y))$ is equivalent to $\nu(\chi^{-1}(P_{D}(B(y,\epsilon)\cap D_{0}(y))))$. We observe that $D_{0}(y)\cap B(y,\epsilon)\subset B(y,\epsilon)$. Therefore, we have
\begin{equation*}
A_{\epsilon}(y)=\chi^{-1}(P_{D}(B(y,\epsilon))))\supset \chi^{-1}(\Pi(D_{0}(y)\cap B(y,\epsilon))).
\end{equation*}

We need to prove that $\nu(A_\epsilon(y) \setminus (\chi^{-1}(P_{D}(D_0(y)\cap B(y,\epsilon))))$ is negligible with respect to $\nu(A_\epsilon(y))$.\\
We recall that $(x,s)\mapsto g_s(x)$ is a $C^1$-diffeomorphism of $X_R$ on its image which contains the ball $B(y,\epsilon)$, assuming $\epsilon$ is sufficiently small.\\

Let $z\in A_\epsilon(y) \setminus \chi^{-1}(P_{D}(D_0(y)\cap B(y,\epsilon)))$. Let us consider the point $z'=g_{h(z)}(\chi(z))\in D_0\setminus B(y,\epsilon)$. The diameter of $ A_\epsilon(y)$ is bounded by  $C_y\epsilon$ and $g_{t}$ is $C^1$. Then $d(z',y)\leq C'_y \epsilon.$ Since $z\in A_\epsilon(y)$, there exists $s\in (-h(z), R(z)-h(z))$ such that $z''=g_s(z')\in B(y,\epsilon)$.
Since we have to compose by the exponential map, we are placed in a local chart containing $y$ and in this chart, the picture is the following:\\
\begin{center}
\begin{figure}[!h]
%\caption{The part of the partition intersecting the boundary}
\centering
\includegraphics[scale=0.14]{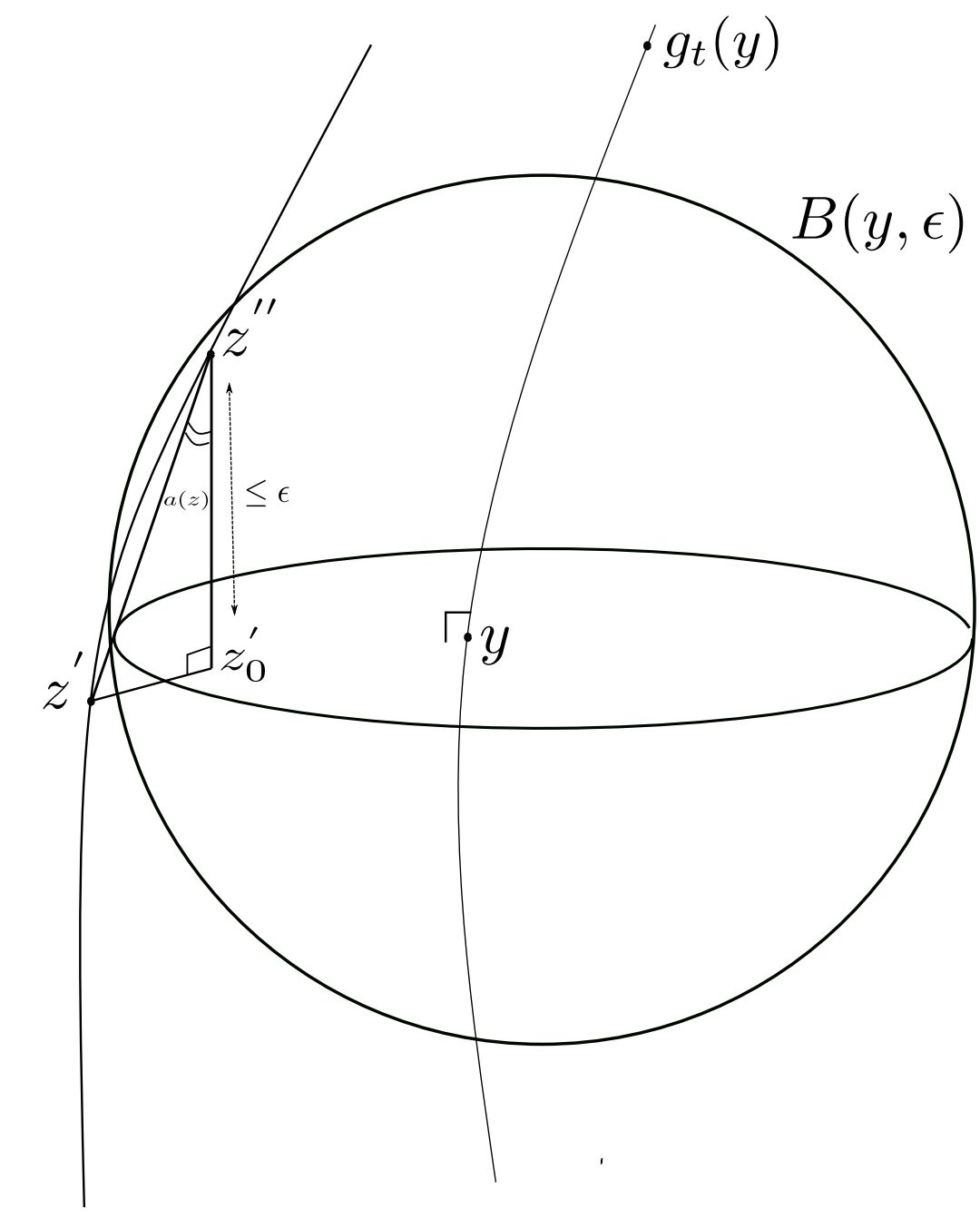}
\end{figure}
\label{boundary}
\end{center}

The angle  $a(z)$ between the direction of the flow at $y$ and the vectors of extremities $z'$ and $z''$ is bounded by $c_g C'_y\epsilon$. So the distance between $z'$ and the orthogonal projection  $z''_0$ of $z''$ on $D_0(y)$ is less than or equal $2\epsilon\tan(c_g C'_y\epsilon)\leq 4 c_g C'_y\epsilon^2$. Then $d(\chi(z),P_{D}(z''_0))<K\epsilon^2$ (for some $K>0$). But $z''_0$ is in $D_0\cap B(y,\epsilon)$. We have thus shown that $A_\epsilon \setminus P_{D}(\chi^{-1}(D_0\cap B(y,\epsilon))$ is contained in $\chi^{-1}\left(   (P_{D}(D_0\cap B(y,\epsilon)))^{[K\epsilon^2]}\right)$. Now the $\nu$-measure of this set is negligible with respect to $\nu(A_\epsilon)$, (see the proof of Lemma \ref{equivalent}).

%Moreover, we have
%\begin{equation*}
%A_{\epsilon-\theta_{\epsilon}}\subset \chi^{-1}(P_{D}(D_{0})\cap B(.,\epsilon))\subset A_{\epsilon}
%\end{equation*}
%Recall that $(x,s)\mapsto g_{s}(x)$ on $X\times[0,\beta_{0}]$ is Lipschitz of Lipschitz constant $c_{g}$. The variation of the direction of the flow is bounded by $c_{g}\epsilon$. Then observing the figure, we have
%\begin{equation*}
%e_{\epsilon}\leq\sin(c_{g}\epsilon)\times \epsilon=c_{g}\epsilon\times\epsilon=c_{g}\epsilon^{2}
%\end{equation*}

%Thus we get the following
%\begin{eqnarray*}
%&&\nu(A_{\epsilon}(.))-\nu(\chi^{-1}(P_{D}(D_{0}\cap B(.,\epsilon))))\\
%&=&\nu\left(A_{\epsilon}(.)-\chi^{-1}(P_{D}(D_{0}\cap B(.,\epsilon)))\right)\\
%&\leq& \nu\left((\partial A_{\epsilon}((.))^{[c_{g}\epsilon^{2}]}\right)
%\end{eqnarray*}
%But using Lemma \ref{equivalent}, we have $\nu\left((\partial A_{\epsilon}(.))^{[c_{g}\epsilon^{2}]}\right)\ll \nu(A_{\epsilon}(.))$, which permits us to deduce that $\nu(A_{\epsilon}(.))$ is equivalent $\nu(\chi^{-1}(P_{D}(B(.,\epsilon)\cap D_{0})))$, and thus finishing the proof.
\end{proof}

\section{Case of geodesic flows on compact smooth surfaces of negative curvature.}

We consider the unit geodesic flow $(\tilde{g}_t)_{t}$ on a $\mathbb{Z}-$cover $\tilde{\mathcal{S}}$ of a $C^{2}$ compact Riemannian surface $\mathcal{S}$ with negative curvature. This flow preserves the Liouville measure $L_{0}$. We recall that $(\tilde{g}_t)_{t}$ defines a flow on the unit tangent bundle $\tilde{M}=T^{1}\tilde{\mathcal{S}}$ of $\tilde{\mathcal{S}}$, and preserves the Liouville measure on $T^{1}\tilde{\mathcal{S}}$. Theorems \ref{theorem1flow} and \ref{theorem2} apply to this model with 
\begin{equation*}
    \mu=\frac{L_0}{2\pi \text{Vol}(\mathcal{S})}.
\end{equation*}

We prove the following result:

\begin{theoreme}\label{geodesicflow}
The following convergence in distribution holds with respect to any probability measure absolutely continuous with respect to $\tilde{\mu}$ with uniformly continuous density:
\[
\frac{\epsilon^2}{2 \text{Vol}(\mathcal{S})}\sqrt{\tau_{\epsilon}(.)}\underset{\epsilon\rightarrow0}{\longrightarrow} \sigma_{flow}\frac{\mathcal{E}}{|\mathcal{N}|}.
\]
where $\mathcal{E}$ and $\mathcal{N}$ are two independent random variables with respective exponential distribution of mean 1 and standard normal distribution.
\end{theoreme}
\begin{proof}
Due to Theorem \ref{theorem2}, we have to prove that $\tilde{\nu}_{0}^{y}(B(y,\epsilon))\sim\dfrac{\epsilon^2}{2 \text{Vol}(\mathcal{S})}$, as $\epsilon\rightarrow 0$. 
Let $y=(q_0,v_0)\in T^{1}M$ and $\gamma$ the geodesic passing by $q_{0}$ and which is orthogonal to $v_{0}$ at $q_{0}$. And let $D_{0}$ be the disk centered at $q_{0}$ and which is transversal to the flow and orthogonal to it at $q_{0}$, such that:\\
$$D_{0}=\{(q,v)\in T^{1}M \text{ such that } q\in\gamma\}$$

We define the transversal measure $\nu_{0}$ on $D_{0}$, for all measurable $A\subset D_{0}$, by:

\begin{equation*}
\nu_{0}(A)=\underset{\epsilon'\rightarrow 0}{\lim}\frac{1}{2\epsilon'}\mu\left(\underset{-\epsilon'<s<\epsilon'}{\bigcup}g_{s}\left(D_{0}\cap A\right)\right).
\end{equation*}

%\begin{equation*}
%\nu_{0}(A)\underset{\epsilon'\rightarrow 0}{\lim}\frac{1}{2\epsilon'}\mu\left(\underset{-\epsilon'<s<\epsilon'}{\bigcup}g_{s}\left(D_{0}\cap B((q_{0},v_{0}),\epsilon)\right)\right).
%\end{equation*}
Let us prove that $\nu_{0}(B((q_0,v_0),\epsilon))\sim\dfrac{\epsilon^2}{2 \text{Vol}(\mathcal{S})}$, as $\epsilon\rightarrow 0$
We consider the chart $\Phi: \mathbb{R}^3\rightarrow T^{1}M$ around $(q_{0}, v_{0})$ which is given by $\Phi(x,y,\alpha)=(q,v)$ such that the geodesic flow in this chart is defined by :
$$\Phi(x+t\cos(\alpha), y+t\sin(\alpha),\alpha)=g_{t}(\Phi(x,y,\alpha))$$
We have:
\begin{itemize}
\item $\Phi(0,0,0)=(q_{0},v_{0})$
\item $\Phi(0,y,\alpha)=(q,v)$, where $q\in\gamma$ and at a distance $|y|$ from $q_{0}$.
\item $\alpha $ is the angle between $v$ and the normal to $\gamma$
\end{itemize}

$\Phi$ is well defined and is a $C^{1}$-diffeomorphism.

Our goal is to compute the $\nu_{0}$-measure of the ball centered at $(q_{0},v_{0})$ and of radius $\epsilon$, that is to estimate the following quantity:

 $$\nu_{0}\left(B((q_{0},v_{0}),\epsilon)\right)=\underset{\epsilon'\rightarrow 0}{\lim}\frac{1}{2\epsilon'}\ \mu\left(\underset{-\epsilon'<s<\epsilon'}{\bigcup}\left(g_{s}(D_{0}\cap B((q_{0},v_{0}),\epsilon))\right)\right).$$

To do this, we consider the set $E_{\epsilon}:=\Phi^{-1}\left(D_{0}\cap B((q_{0},v_{0}),\epsilon)\right)$, then $$E_{\epsilon}=\{(0,y,\alpha) \text{ such that } y^2+\alpha^2 <\epsilon^2 \text{ and } (0,y,\alpha)\in \Phi^{-1}(B((q_{0},v_{0}),\epsilon))\}.$$ 
Since $\Phi$ is a $C^{1}$-diffeomorphism, and $\Phi(0,0,0)=(q_{0},v_{0})$, there exists $k_{\epsilon}>0$ such that
\begin{equation}\label{Eeps}
\{(0,y,\alpha) \text{ s.t. } y^2+\alpha^2 <(1-k_{\epsilon})\epsilon^2\}\subset E_{\epsilon}\subset \{(0,y,\alpha) \text{ s.t. } y^2+\alpha^2 <(1+k_{\epsilon})\epsilon^2\},
\end{equation}
with $k_{\epsilon}\rightarrow 0$ as $\epsilon\rightarrow 0$. Let us set $$V_{\epsilon, \epsilon'}=\underset{-\epsilon'<s<\epsilon'}{\bigcup}g_{s}(D_{0}\cap B((q_{0},v_{0}),\epsilon))$$ and let $L$ be the image measure by $\Phi^{-1}$ of the Liouville measure $L_0$. $L$ has a continuous density with respect to Lebesgue measure and this density is 1 at (0,0,0) with respect to the Lebesgue measure on the chart (associated to the canonical metric $dx^2 +dy^2+d\alpha^2$). Then there exists $k_{L,\epsilon}>0$ such that we have:
$$(1-k_{L,\epsilon})Leb(\Phi^{-1}(V_{\epsilon,\epsilon'}))\leq L(\Phi^{-1}(V_{\epsilon,\epsilon'}))\leq (1+k_{L,\epsilon})Leb(\Phi^{-1}(V_{\epsilon,\epsilon'})), \mbox{ and } \underset{\epsilon \rightarrow 0}{\lim}\ k_{L,\epsilon}=0.$$

Now, our goal is to estimate $Leb(\Phi^{-1}(V_{\epsilon,\epsilon'}))$. But we have

$$\Phi^{-1}(V_{\epsilon,\epsilon'})=\{(s\cos(\alpha),y+s\sin(\alpha),\alpha), |s|<\epsilon', \text{ and } (0,y,\alpha)\in E_{\epsilon}\}.$$

Thus its measure,
\begin{eqnarray*}
\displaystyle\int\displaylimits_{-\epsilon}^{\epsilon}\left(\int\displaylimits_{-\epsilon'\cos(\alpha)}^{\epsilon'\cos(\alpha)}\left(\int_{-\sqrt{(1-k_{\epsilon})\epsilon^2-\alpha^{2}}}^{\sqrt{(1-k_{\epsilon})\epsilon^2-\alpha^{2}}}dy\right)\ ds\ \right) d\alpha\leq &Leb(\Phi^{-1}(V_{\epsilon,\epsilon'}))&\leq \int\displaylimits_{-\epsilon}^{\epsilon}\left(\int\displaylimits_{-\epsilon'\cos(\alpha)}^{\epsilon'\cos(\alpha)}\left(\int_{-\sqrt{(1+k_{\epsilon})\epsilon^2-\alpha^{2}}}^{\sqrt{(1+k_{\epsilon})\epsilon^2-\alpha^{2}}}dy\right)\ ds\ \right)d\alpha\\
4\epsilon'\int_{-\epsilon}^{\epsilon}\cos(\alpha)\sqrt{(1-k_{\epsilon})\epsilon^2-\alpha^{2}}\ d\alpha\leq &Leb(\Phi^{-1}(V_{\epsilon,\epsilon'}))&\leq 4\epsilon'\int_{-\epsilon}^{\epsilon}\cos(\alpha)\sqrt{(1+k_{\epsilon})\epsilon^2-\alpha^{2}}\ d\alpha
\end{eqnarray*}
%Computing the later integral, we get:
%$$\displaystyle\int_{-\epsilon}^{\epsilon}\sqrt{(1\overset{-}{+}k_{\epsilon})\epsilon^2-\alpha^{2}}\ d\alpha=\frac{(1\overset{-}{+}k_{\epsilon})\epsilon^2}{2}\left(2\arcsin(\sqrt{1\overset{-}{+}k_{\epsilon}})+\sqrt{1\overset{-}{+}k_{\epsilon}}\right)$$
%
As we have $\cos(\epsilon)<\cos(\alpha)<1$, then we get

$$(1-k_{L,\epsilon})\cos(\epsilon)4\epsilon'\int_{-\epsilon}^{\epsilon}\sqrt{(1+k_{\epsilon})\epsilon^2-\alpha^{2}}\ d\alpha\leq L(\Phi^{-1}(V_{\epsilon,\epsilon'}))\leq (1+k_{L,\epsilon})4\epsilon'\int_{-\epsilon}^{\epsilon}\sqrt{(1+k_{\epsilon})\epsilon^2-\alpha^{2}}\ d\alpha$$
from which we get that
$$2(1-k_{L,\epsilon})\cos(\epsilon)\int_{-\epsilon}^{\epsilon}\sqrt{(1+k_{\epsilon})\epsilon^2-\alpha^{2}}\ d\alpha\leq \underset{\epsilon'\rightarrow 0}{\lim}\frac{1}{2\epsilon'}\ L_{0}(V_{\epsilon,\epsilon'})\leq 2(1+k_{L,\epsilon})\int_{-\epsilon}^{\epsilon}\sqrt{(1+k_{\epsilon})\epsilon^2-\alpha^{2}}\ d\alpha$$

As $\epsilon$ goes to zero, we have $k_{\epsilon}$ and $k_{L,\epsilon}$ go to zero, thus
$$\nu_{0}(B((q_{0},v_{0}),\epsilon)=\underset{\epsilon'\rightarrow 0}{\lim}\frac{1}{2\epsilon'}\mu(V_{\epsilon,\epsilon'})\underset{\epsilon\rightarrow0}{\sim}\frac{2}{2\pi \text{Vol}(\mathcal{S})}\int_{-\epsilon}^{\epsilon}\sqrt{\epsilon^2-\alpha^{2}}$$

Computing this latter integral, we have
$$\int_{-\epsilon}^{\epsilon}\sqrt{\epsilon^2-\alpha^{2}}d\alpha=\frac{\pi}{2} \epsilon^2,$$

and hence we get:
$$\nu_{0}( B((q_{0},v_{0}),\epsilon)\underset{\epsilon\rightarrow0}{\sim}\dfrac{\epsilon^2}{2 \text{Vol}(\mathcal{S})}. $$

\end{proof}

\subsection*{Acknowledgement}
I would like to thank F. P\`{ene} and B. Saussol for their availability to answer all my questions throughout this work. I would  also like to thank A. Boulanger for his help in the proof of a technical point in Riemannian Geometry.

\end{document}